\documentclass[preprint,12pt]{elsarticle}

\pdfoutput=1

\usepackage[english]{babel}

\usepackage[letterpaper,top=2cm,bottom=2cm,left=3cm,right=3cm,marginparwidth=1.75cm]{geometry}

\usepackage{amsmath, amsthm}
\usepackage{graphicx}
\usepackage[colorlinks=true, allcolors=blue]{hyperref}

\usepackage{amsmath, amssymb, amsthm, fullpage, color, comment, mathtools, tikz, dsfont, tikz-cd}

\usepackage{adjustbox}

\theoremstyle{plain}
\newtheorem{theorem}{Theorem}[section]
\newtheorem{proposition}[theorem]{Proposition}
\newtheorem{lemma}[theorem]{Lemma}
\newtheorem{corollary}[theorem]{Corollary}

\theoremstyle{definition}

\theoremstyle{remark}
\newtheorem{remark}[theorem]{Remark}
\newtheorem{example}[theorem]{Example}

\journal{Advances in Mathematics}

\begin{document}

\begin{frontmatter}



\title{A degenerate version of Brion's formula}


\author[first]{Carsten Peterson} 

\affiliation[first]{organization={Sorbonne Université, Université Paris Cité, CNRS, IMJ-PRG},
            postcode={F-75005},
            city={Paris}, 
            country={France}}

\begin{abstract}

    Let $\mathfrak{p} \subset V$ be a polytope and $\xi \in V_{\mathbb{C}}^*$. We obtain an expression for $I(\mathfrak{p}; \alpha) := \int_{\mathfrak{p}} e^{\langle \alpha, x \rangle} dx$ as a sum of meromorphic functions in $\alpha \in V^*_{\mathbb{C}}$ parametrized by the faces $\mathfrak{f}$ of $\mathfrak{p}$ on which $\langle \xi, x \rangle$ is constant. Each term only depends on the local geometry of $\mathfrak{p}$ near $\mathfrak{f}$ (and on $\xi$) and is holomorphic at $\alpha = \xi$. When $\langle \xi, \cdot \rangle$ is only constant on the vertices of $\mathfrak{p}$ our formula reduces to Brion's formula.

    Suppose $\mathfrak{p}$ is a rational polytope with respect to a lattice $\Lambda$. We obtain an expression for $S(\mathfrak{p}; \alpha) := \sum_{\lambda \in \mathfrak{p} \cap \Lambda} e^{\langle \alpha, \lambda \rangle}$ as a sum of meromorphic functions parametrized by the faces $\mathfrak{f}$ on which $e^{\langle \xi, x \rangle} = 1$ on a finite index sublattice of $\text{lin}(\mathfrak{f}) \cap \Lambda$. Each term only depends on the local geometry of $\mathfrak{p}$ near $\mathfrak{f}$ (and on $\xi$ and $\Lambda$) and is holomorphic at $\alpha = \xi$. When $e^{\langle \xi, \cdot \rangle} \neq 1$ at any non-zero lattice point on a line through the origin parallel to an edge of $\mathfrak{p}$, our formula reduces to Brion's formula, and when $\xi = 0$, it reduces to the Ehrhart quasi-polynomial.

    Our formulas are particularly useful for understanding how $I(\mathfrak{p}(h); \xi)$ and $S(\mathfrak{p}(h); \xi)$ vary in a family of polytopes $\mathfrak{p}(h)$ with the same normal fan. When considering dilates of a fixed polytope, our formulas may be viewed as polytopal analogues of Laplace's method and the method of stationary phase. Such expressions naturally show up in analysis on symmetric spaces and affine buildings.  
\end{abstract}


\end{frontmatter}



\tableofcontents

\section{Introduction}
\subsection{Degenerate Brion's formula: continuous setting} Let $\mathfrak{p}$ be a polytope in a real vector space $V$ with an inner product $\langle \cdot, \cdot \rangle$. Let $\textnormal{Vert}(\mathfrak{p})$ denote the set of vertices of $\mathfrak{p}$, and let $\textnormal{Face}(\mathfrak{p})$ denote the set of faces. Each vertex $\mathfrak{v}$ has an associated \textit{tangent cone} $\mathfrak{t}_{\mathfrak{v}}^{\mathfrak{p}}$, namely the cone based at $\mathfrak{v}$ generated by the directions one can move in from $\mathfrak{v}$ and stay inside $\mathfrak{p}$. Let $\xi \in V^*_{\mathbb{C}}$. Brion's formula \cite{brion} concerns the Fourier-Laplace transform of the indicator function of a polytope. It says that
\begin{gather}
    \int_{\mathfrak{p}} e^{\langle \alpha, x \rangle} dx \ ``=" \sum_{\mathfrak{v} \in \textnormal{Vert}(\mathfrak{p})} \int_{\mathfrak{t}_{\mathfrak{v}}^{\mathfrak{p}}} e^{\langle \alpha, x \rangle} dx, \label{fake_brion}
\end{gather}
where $\alpha \in V^*_{\mathbb{C}}$. However, this formula as written does not make sense as there does not exist any $\alpha$ such that every integral on the right hand side converges absolutely. Rather one must interpret each term in the right hand side as arising from the meromorphic continuation in $\alpha$ of the function $I(\mathfrak{t}_{\mathfrak{v}}^{\mathfrak{p}}; \alpha) := \int_{\mathfrak{t}_{\mathfrak{v}}^{\mathfrak{p}}} e^{\langle \alpha, x \rangle} dx$ from the locus on which the integral converges absolutely. More generally, given a polyhedron $\mathfrak{q}$ not containing a straight line, one can define a meromorphic function $I(\mathfrak{q}; \cdot)$ in this way; otherwise we define $I(\mathfrak{q}; \cdot)$ to be zero. The right hand side of \eqref{fake_brion} should be properly interpreted as a sum of the meromorphic functions $I(\mathfrak{t}_{\mathfrak{v}}^{\mathfrak{p}}; \alpha)$, and Brion's formula says this sum is equal to the function on the left hand side.

Notice that the function on the left hand side of \eqref{fake_brion} is an entire function. Each function $I(\mathfrak{t}_{\mathfrak{v}}^{\mathfrak{p}}; \alpha)$ on the right hand side is non-singular for \textit{generic} choices of $\alpha$; it has singularities precisely when $\alpha$ is constant on some extremal ray of $\mathfrak{t}_{\mathfrak{v}}^{\mathfrak{p}}$, or equivalently on some edge of $\mathfrak{p}$ containing $\mathfrak{v}$. If one were to try to evaluate the left hand side of \eqref{fake_brion} by summing up term by term on the right hand side for such ``\textit{degenerate}'' choices of $\alpha$, one would encounter terms resembling ``$\infty - \infty$''. 

We shall follow the convention that $\alpha$ is a variable varying over $V^*_{\mathbb{C}}$, and $\xi$ denotes some fixed element in $V^*_{\mathbb{C}}$. The purpose of this work is to, given a fixed ``degenerate'' choice of $\xi$, re-express the right hand side of \eqref{fake_brion} as a sum of meromorphic functions which are non-singular at $\alpha = \xi$. Furthermore, we want this expression to be \textit{local} in the sense that each term only depends on local geometry of the polytope; for example in the usual Brion's formula, each term only depends on the local geometry near the corresponding vertex.

The terms appearing in our ``degenerate'' Brion's formula will be parametrized by the faces on which $\xi$ is constant. We call such faces \textit{$\xi$-constant}, and we denote the set of all $\xi$-constant faces by $\{\mathfrak{p}\}_\xi$. Let $\mathfrak{f}$ be such a face. We associate to $\mathfrak{f}$ a certain \textit{virtual cone} $\textnormal{LC}^\mathfrak{p}_\mathfrak{f}(\xi)$, namely a formal linear combination of cones with integral coefficients; we call this the \textit{alternating Levi cone}, as it is an alternating sum over terms which arise in a similar way to how Levi subgroups of parabolic subgroups of $\textnormal{SL}(n)$ show up (i.e., given the data of a flag, choosing a decomposition of a vector space into subspaces compatible with the flag). This cone lives in the space $V/\textnormal{lin}(\mathfrak{f}) \simeq \textnormal{lin}(\mathfrak{f})^\perp$, where $\textnormal{lin}(\mathfrak{f})$ is the linear subspace of $V$ parallel to $\mathfrak{f}$. It is a pointed virtual cone with vertex at $\mathfrak{f}^{\mathfrak{f}^\perp}$, i.e. the point in $\textnormal{lin}(\mathfrak{f})^\perp$ that $\mathfrak{f}$ gets maps to under orthogonal projection. Furthermore, this cone only depends on the local geometry of $\mathfrak{p}$ near $\mathfrak{f}$, namely on the tangent cone of $\mathfrak{f}$, or, equivalently, on the transverse cone. The \textit{transverse cone} of $\mathfrak{f}$ in $\mathfrak{p}$, denoted $\mathfrak{t}_{\mathfrak{f}}^{\mathfrak{p}}$, is defined as the orthogonal projection of the tangent cone of $\mathfrak{f}$ to $\textnormal{lin}(\mathfrak{f})^\perp$. When $\mathfrak{f}$ is \textit{$\xi$-maximal}, meaning it is $\xi$-constant and no bigger face containing it is also $\xi$-constant, then $\textnormal{LC}_{\mathfrak{f}}^{\mathfrak{p}}(\xi) = \mathfrak{t}_{\mathfrak{f}}^{\mathfrak{p}}$. More generally we have:
\begin{gather}
    \textnormal{LC}_{\mathfrak{f}}^{\mathfrak{p}}(\xi) := \sum_{\mathfrak{f} \subset \mathfrak{h}_1 \subset \dots \subset \mathfrak{h}_\ell} (-1)^\ell \mathfrak{t}_{\mathfrak{f}}^{\mathfrak{h}_1} \times \mathfrak{t}_{\mathfrak{h}_1}^{\mathfrak{h}_2} \times \dots \times \mathfrak{t}_{\mathfrak{h}_{\ell-1}}^{\mathfrak{h}_\ell} \times \mathfrak{t}_{\mathfrak{h}_\ell}^{\mathfrak{p}}, \label{levi_cone}
\end{gather}
where the sum is over all chains of  $\xi$-constant faces of any dimension satisfying $\mathfrak{f} \subset \mathfrak{h}_1 \subset \dots \subset \mathfrak{h}_\ell$. We use the notation ${}^0\textnormal{LC}^{\mathfrak{p}}_{\mathfrak{f}}(\xi)$ to denote the alternating Levi cone shifted to be based at the origin; we call this the \textit{shifted alternating Levi cone}. 

Part of the significance of the alternating Levi cone is explained by the following polytope decomposition, which generalizes the classical Brianchon-Gram decomposition and to the best of our knowledge has not been observed before.
\begin{theorem}[proof in Section \ref{subsec_polyhedral_decomp}] \label{thm_polytope_decomp_intro}
    Let $\mathfrak{p}$ be a polytope, and $\xi \in V^*_{\mathbb{C}}$. Modulo indicator functions of polyhedra containing straight lines (denoted $\equiv$), we have
    \begin{gather}
        \mathfrak{p} \equiv \sum_{\mathfrak{f} \in \{\mathfrak{p}\}_\xi} \mathfrak{f}^{\mathfrak{f}} \times \textnormal{LC}^\mathfrak{p}_\mathfrak{f}(\xi), \label{polytope_decomp}
    \end{gather}
    where $\mathfrak{f}^{\mathfrak{f}}$ is the orthogonal projection of $\mathfrak{f}$ onto $\textnormal{lin}(\mathfrak{f})$.
\end{theorem}
\noindent We interpret this equation as a statement about \textit{virtual polyhedra}, i.e. we work in the vector space generated by indicator functions of polyhedra. We remark that the definition of the alternating Levi cone and the decomposition in \eqref{polytope_decomp} bear some resemblance to the Brion-Vergne decomposition (Theorem 1.2 of \cite{brion_vergne2}; see also Theorem 19 of \cite{BBKV}).

Fixing $\xi \in V^*_{\mathbb{C}}$ and a face $\mathfrak{f} \in \{\mathfrak{p}\}_{\xi}$, the alternating Levi cone $\textnormal{LC}_{\mathfrak{f}}^{\mathfrak{p}}(\xi)$ lives in $\textnormal{lin}(\mathfrak{f})^\perp$, and $I(\textnormal{LC}_{\mathfrak{f}}^{\mathfrak{p}}(\xi); \tau)$ is a meromorphic function on $(\textnormal{lin}(\mathfrak{f})^\perp)^*_{\mathbb{C}}$. Since $\xi$ is zero on $\textnormal{lin}(\mathfrak{f})$, we may naturally view $\xi \in (\textnormal{lin}(\mathfrak{f})^\perp)^*_{\mathbb{C}}$.

\begin{theorem}[see Theorem \ref{thm_holomorphic} for a more refined version] \label{thm_holomorphic_cont}
    Let $\mathfrak{p}$ be a polytope, $\xi \in V^*_{\mathbb{C}}$, and $\mathfrak{f} \in \{\mathfrak{p}\}_{\xi}$. The function $I(\textnormal{LC}_{\mathfrak{f}}^{\mathfrak{p}}(\xi); \tau)$ with $\tau \in (\textnormal{lin}(\mathfrak{f})^\perp)^*_{\mathbb{C}}$ is holomorphic at $\tau = \xi$.
\end{theorem}
The proof of Theorem \ref{thm_holomorphic_cont} is the most delicate part of this paper. Note that if $\mathfrak{f}$ is not $\xi$-maximal, we have that $I(\mathfrak{k}; \cdot)$ is singular at $\xi$ for every individual cone $\mathfrak{k}$ involved in defining $\textnormal{LC}^\mathfrak{p}_\mathfrak{f}(\xi)$ in \eqref{levi_cone}.

Because the map $I(\cdot; \cdot)$ is a \textit{valuation}, i.e. a linear map on the space of virtual polyhedra, and it sends polyhedra containing lines to zero, Theorem \ref{thm_polytope_decomp_intro} immediately implies the following identity of meromorphic functions (again, treating $\xi$ as fixed and $\alpha$ as variable):
\begin{gather}
    I(\mathfrak{p}; \alpha) = \sum_{\mathfrak{f} \in \{\mathfrak{p}\}_{\xi}} I(\mathfrak{f}^\mathfrak{f} \times \textnormal{LC}_{\mathfrak{f}}^{\mathfrak{p}}(\xi); \alpha) = \sum_{\mathfrak{f} \in \{\mathfrak{p}\}_{\xi}} I(\mathfrak{f}^\mathfrak{f}; \alpha^{\mathfrak{f}}) \cdot I(\textnormal{LC}_{\mathfrak{f}}^{\mathfrak{p}}(\xi); \alpha^{\mathfrak{f}^\perp}), \label{eqn_degen_brion_decomp}
\end{gather}
where $\alpha^{\mathfrak{f}}$ is the restriction of $\alpha$ to $\textnormal{lin}(\mathfrak{f})$, and $\alpha^{\mathfrak{f}^\perp}$ is the restriction of $\alpha$ to $\textnormal{lin}(\mathfrak{f})^\perp$. Each term in \eqref{eqn_degen_brion_decomp} is clearly meromorphic in $\alpha$. When $\alpha = \xi$, then $\alpha^{\mathfrak{f}} = 0$ and $I(\mathfrak{f}^{\mathfrak{f}}; 0) = \textnormal{vol}(\mathfrak{f})$. Theorem \ref{thm_holomorphic_cont} tells us that $I(\textnormal{LC}_{\mathfrak{f}}^{\mathfrak{p}}(\xi); \alpha^{\mathfrak{f}^\perp})$ is in fact holomorphic at $\alpha = \xi$. We thus obtain:

\begin{corollary}[Degenerate Brion's formula: continuous setting] \label{thm_degen_brion_cont_intro}
    Let $\mathfrak{p}$ be a polytope, and let $\xi \in V_{\mathbb{C}}^*$. Then
    \begin{gather}
        \int_{\mathfrak{p}} e^{\langle \xi, x \rangle} dx = \sum_{\mathfrak{f} \in \{\mathfrak{p}\}_{\xi} } \textnormal{vol}(\mathfrak{f}) \cdot I(\textnormal{LC}_{\mathfrak{f}}^{\mathfrak{p}}(\xi); \xi) = \sum_{\mathfrak{f} \in \{\mathfrak{p}\}_{\xi} } \textnormal{vol}(\mathfrak{f}) \cdot I({}^0\textnormal{LC}_{\mathfrak{f}}^{\mathfrak{p}}(\xi); \xi) \cdot e^{\langle \xi, \mathfrak{f} \rangle}, \label{degenerate_brion_cont}
    \end{gather}
    and each term is well-defined.
\end{corollary}

Notice that because $\langle \xi, \cdot \rangle$ is constant on $\mathfrak{f}$, the expression $e^{\langle \xi, \mathfrak{f} \rangle}$ is well-defined. In case $\xi$ is generic, then the only $\xi$-constant faces are the vertices. In that case for each vertex $\mathfrak{v}$, we have $\textnormal{vol}(\mathfrak{v}) = 1$ and $\textnormal{LC}^\mathfrak{p}_\mathfrak{v}(\xi) = \mathfrak{t}_{\mathfrak{v}}^{\mathfrak{p}}$, so Corollary \ref{thm_degen_brion_cont_intro} simply reduces to Brion's formula \eqref{fake_brion}. 

Our degenerate Brion's formula is particularly useful when considering dilates of a fixed polytope; in that case we obtain
\begin{gather}
    I(t \cdot \mathfrak{p}; \xi) = \int_{t \cdot \mathfrak{p}} e^{\langle \xi, x \rangle} dx = \sum_{\mathfrak{f} \in \{\mathfrak{p}\}_{\xi}} \textnormal{vol}(\mathfrak{f}) \cdot I({}^0 \textnormal{LC}^\mathfrak{p}_\mathfrak{f}(\xi); \xi)  \cdot t^{\textnormal{dim}(\mathfrak{f})} \cdot e^{t \cdot \langle \xi, \mathfrak{f} \rangle}. \label{brion_cont_dilation}
\end{gather}
More generally, if we have a family of polytopes $\mathfrak{p}(h)$ with the same normal fan (for example, we may translate different supporting hyperplanes by different amounts), then we can naturally identify corresponding faces $\mathfrak{f}(h)$ in the family. The shifted alternating Levi cone ${}^0 \textnormal{LC}^{\mathfrak{p}(h)}_{\mathfrak{f}(h)}(\xi)$ is independent of $h$, so our formula is of a nice form for understanding how $I(\mathfrak{p}(h); \xi)$ varies in $h$.

\subsection{Degenerate Brion's formula: discrete setting}
The original version of Brion's formula \cite{brion} was about sums of exponential functions over lattice points in polytopes. Brion's proof applied the Lefschetz-Riemann-Roch theorem in equivariant $K$-theory to the toric variety associated to a lattice polytope. A nice elementary introduction to the discrete version of Brion's formula is \cite{beck_haase_sottile}. 

Suppose $V$ contains a lattice $\Lambda$, and suppose that $\mathfrak{p}$ is a rational polytope with respect to $\Lambda$. We shall also assume that our inner product $\langle \cdot, \cdot \rangle$ is rational, i.e. it takes rational values on pairs of lattice points. The discrete version of Brion's formula concerns the discrete Fourier-Laplace transform of the indicator function of the lattice points inside a polytope; it states that
\begin{gather}
    \sum_{\lambda \in \mathfrak{p} \cap \Lambda} e^{\langle \alpha,  \lambda \rangle} \ ``=" \sum_{\mathfrak{v} \in \textnormal{Vert}(\mathfrak{p})} \sum_{\lambda \in \mathfrak{t}_{\mathfrak{v}}^{\mathfrak{p}} \cap \Lambda} e^{\langle \alpha, \lambda \rangle}. \label{fake_brion_discrete}
\end{gather}
Again, this formula as written does not make sense as there does not exist any $\alpha$ such that every inner sum on the right hand side converges absolutely. We instead interpret each inner sum on the right hand side of \eqref{fake_brion_discrete} as arising from the meromorphic continuation in $\alpha$ of the function $S_\Lambda(\mathfrak{t}_{\mathfrak{v}}^{\mathfrak{p}}; \alpha) := \sum_{\lambda \in \mathfrak{t}_{\mathfrak{v}}^{\mathfrak{p}} \cap \Lambda} e^{\langle \alpha, \lambda \rangle}$ from the locus on which this expression converges absolutely. As before, we define $S_\Lambda(\mathfrak{q}; \cdot)$ for any rational polyhedron $\mathfrak{q}$ not containing a straight line through meromorphic continuation from the locus on which the expression converges absolutely; otherwise we set $S_\Lambda(\mathfrak{q}; \cdot)$ to be zero.

Like in the continuous case, the function on the left hand side of \eqref{fake_brion_discrete} is an entire function, whereas each inner sum on the right hand side is merely meromorphic. The function $S_\Lambda(\mathfrak{t}_{\mathfrak{v}}^{\mathfrak{p}}; \alpha)$ has singularities precisely when $e^{\langle \alpha, \cdot \rangle} = 1$ for all lattice points on a line through the origin parallel to an edge containing $\mathfrak{v}$. Note that this is weaker than $\langle \alpha, \cdot \rangle$ being constant on an edge containing $\mathfrak{v}$; for example if $\alpha \in 2 \pi i \cdot \Lambda^* \setminus \{0\}$, then $e^{\langle \alpha, \lambda \rangle} = 1$ for all $\lambda \in \Lambda$ even though $\alpha \neq 0$. 

For generic choices of $\xi$, we can evaluate the left hand side of \eqref{fake_brion_discrete} at $\alpha = \xi$ by summing up the right hand side term by term. However, for degenerate choices of $\xi$ this is no longer possible. Thus, given a fixed degenerate choice of $\xi$, it is desirable to re-express the right hand side of \eqref{fake_brion_discrete} as a sum of meromorphic functions in $\alpha$ such that each one is holomorphic at $\alpha = \xi$.

Since $S_{\Lambda}(\cdot, \cdot)$ is a valuation which sends polyhedra containing lines to zero, we may apply Theorem \ref{thm_polytope_decomp_intro} to obtain the following identity of meromorphic functions (again, with $\xi$ fixed and $\alpha$ varying):
\begin{gather}
    S_\Lambda(\mathfrak{p}; \alpha) = \sum_{\mathfrak{f} \in \{\mathfrak{p}\}_\xi} S_{\Lambda}(\mathfrak{f}^{\mathfrak{f}} \times \textnormal{LC}_{\mathfrak{f}}^{\mathfrak{p}}(\xi); \alpha). \label{eqn_intro_pre_brion}
\end{gather}
However, several new interlinked difficulties arise in comparison to the continuous setting. 
\begin{enumerate}
    \item[(1)] We could have that $e^{\langle \xi, \cdot \rangle} = 1$ at all lattice points on a face, even though $\langle \xi, \cdot \rangle$ is not constant on the face. In the most extreme example if we take $\xi \in 2 \pi i \cdot \Lambda^* \setminus \{0\}$, then $e^{\langle \xi, \lambda \rangle} = 1$ for all $\lambda \in \Lambda$, even though $\langle \xi, \cdot \rangle$ might not be constant on any face of $\mathfrak{p}$ other than the vertices. Then $S_{\Lambda}(\mathfrak{p}; \xi) = S_{\Lambda}(\mathfrak{p}; 0)$, but $\{\mathfrak{p}\}_\xi \neq \{\mathfrak{p}\}_0$. Thus $\{\mathfrak{p}\}_\xi$ is not quite the right set to study in the discrete setting.

    \item[(2)] Even if we find the right set to replace $\{\mathfrak{p}\}_\xi$, it is not clear that $S_{\Lambda}(\mathfrak{f}^{\mathfrak{f}} \times \textnormal{LC}_{\mathfrak{f}}^{\mathfrak{p}}(\xi); \alpha)$ is holomorphic at $\alpha = \xi$.

    \item[(3)] Unlike in \eqref{eqn_degen_brion_decomp}, we cannot split up the terms in \eqref{eqn_intro_pre_brion} into terms of the form ``$S_{\Lambda}(\mathfrak{f}^\mathfrak{f}; \alpha^{\mathfrak{f}}) \cdot S_{\Lambda}(\textnormal{LC}_{\mathfrak{f}}^{\mathfrak{p}}(\xi); \alpha^{\mathfrak{f}^\perp})$''. This notation does not even make sense because the spaces $\textnormal{lin}(\mathfrak{f})$ and $\textnormal{lin}(\mathfrak{f})^\perp$ contain (at least) two different natural lattices derived from $\Lambda$, namely the lattice obtained from intersecting $\Lambda$ with the subspace, and the lattice obtained from projection of $\Lambda$ onto the subspace.
\end{enumerate}

We shall overcome all of these issues. One of the main technical tools to do so, in particular to address Difficulty (2) above, is the local Euler-Maclaurin formula of Berline-Vergne \cite{berline_vergne}. Euler-Maclaurin formulas in general provide a precise relationship between integrating a function over a polytope and summing up the function at the lattice points in the polytope. The version of the Euler-Maclaurin formula in \cite{berline_vergne} builds on earlier work of Pukhlikov-Khovanskii \cite{pukhlikov_khovanskii}, Cappell-Shaneson \cite{cappell_shaneson}, and Brion-Vergne \cite{brion_vergne}. In particular Pukhlikov-Khovanskii \cite{pukhlikov_khovanskii} noticed that the Euler-Maclaurin formula for integral polytopes amounts to an explicit version of the Riemann-Roch theorem for the associated toric variety. In the case of a unimodular polytope associated to a smooth toric variety, the Euler-Maclaurin formula has a particularly simple form involving the so-called Todd operator $\textnormal{Td}(\frac{\partial}{\partial x})$, which is a differential operator whose symbol is $\frac{x}{1 - e^{-x}}$. This is also the exponential generating function for the Bernoulli numbers, which show up in the classical Euler-Maclaurin formula. See \cite{karshon_sternberg_weitsman} for a nice survey on the topic.

The approach of Berline-Vergne \cite{berline_vergne} is elementary and does not use the theory of toric varieties. Before describing it, we must first set some notation. Given a (not necessarily full-dimensional) polyhedron $\mathfrak{f}$, let $\Lambda^{\mathfrak{f}^\perp}$ denote the orthogonal projection of $\Lambda$ onto $\textnormal{lin}(\mathfrak{f})^\perp$, and let $\Lambda_\mathfrak{f}$ denote the intersection of $\Lambda$ with $\textnormal{lin}(\mathfrak{f})$. Let $I^{\Lambda_\mathfrak{f}}(\mathfrak{f}; \alpha)$ denote the integral of $e^{\langle \alpha, \cdot \rangle}$ over $\mathfrak{f}$ with the respect to the Lebesgue measure on $\mathfrak{f}^{\mathfrak{f}^\perp} + \textnormal{lin}(\mathfrak{f})$ normalized so that $\mathfrak{f}^{\mathfrak{f}^\perp} + \Lambda_{\mathfrak{f}}$ has covolume 1. Given the data of a vector space $W$, a lattice $\Gamma \subset W$, and a rational inner product $\langle \cdot, \cdot \rangle'$ on $W$, Berline-Vergne construct a map $\mu_{W}^{\Gamma, \langle \cdot, \cdot \rangle'}$ from rational cones in $W$ to meromorphic functions on $W^*_{\mathbb{C}}$. Furthermore $\mu_W^{\Gamma, \langle \cdot, \cdot \rangle'}(\gamma + \mathfrak{k}) = \mu_W^{\Gamma, \langle \cdot, \cdot \rangle'}(\mathfrak{k})$ for any rational cone $\mathfrak{k}$, and any $\gamma \in \Gamma$. We shall often suppress the data of the inner product in the superscript as it will be obvious from context what inner product we are using. These functions $\mu$ satisfy the following remarkable identity of meromorphic functions for any rational polyhedron $\mathfrak{q}$ in $V$:
\begin{gather}
    S_\Lambda(\mathfrak{q}; \alpha) = \sum_{\mathfrak{f} \in \textnormal{Face}(\mathfrak{q})} \mu_{\textnormal{lin}(\mathfrak{f})^\perp}^{\Lambda^{\mathfrak{f}^\perp}} (\mathfrak{t}_{\mathfrak{f}}^{\mathfrak{q}}; \alpha) \cdot I^{\Lambda_{\mathfrak{f}}}(\mathfrak{f}; \alpha). \label{euler_maclaurin}
\end{gather}
On each subspace $\textnormal{lin}(\mathfrak{f})^\perp$, we are using the restriction of the inner product on $V$. Each function $\mu_{\textnormal{lin}(\mathfrak{f})^\perp}^{\Lambda^{\mathfrak{f}^\perp}} (\mathfrak{t}_{\mathfrak{f}}^{\mathfrak{q}}; \alpha)$ is the extension of the function $\mu_{\textnormal{lin}(\mathfrak{f})^\perp}^{\Lambda^{\mathfrak{f}^\perp}} (\mathfrak{t}_{\mathfrak{f}}^{\mathfrak{q}}; \zeta)$ with $\zeta \in (\textnormal{lin}(\mathfrak{f})^\perp)^*_{\mathbb{C}}$ to $\alpha \in V^*_{\mathbb{C}}$ via pullback of the orthogonal projection $V^*_{\mathbb{C}} \to (\textnormal{lin}(\mathfrak{f})^\perp)^*_{\mathbb{C}}$.

We now address Difficulty (1) from above. Let $\xi \in V^*_{\mathbb{C}}$. As we discuss in Section \ref{subsec_xi_lambda}, in an elementary way we may find an element $\tilde{\xi} \in V^*_{\mathbb{C}}$ and a finite index sublattice $\tilde{\Lambda} \leq \Lambda$ satisfying the following properties. First off, $e^{\langle \xi, \lambda \rangle} = e^{\langle \tilde{\xi}, \lambda \rangle}$ for $\lambda \in \tilde{\Lambda}$. Furthermore, for any $\lambda \in \tilde{\Lambda}$ for which $e^{\langle \tilde{\xi}, \lambda \rangle} = 1$, we in fact that that $\langle \tilde{\xi}, \lambda \rangle = 0$. Finally, we have that
\begin{gather*}
    \{\mathfrak{p}\}_{\xi, \Lambda} := \{\mathfrak{f} \in \textnormal{Face}(\mathfrak{p}): e^{\langle \xi, \lambda \rangle} = 1 \textnormal{ on a finite index sublattice of } \textnormal{lin}(\mathfrak{f}) \cap \Lambda \} = \{\mathfrak{p}\}_{\tilde{\xi}}.
\end{gather*}
Using $\{\mathfrak{p}\}_{\xi, \Lambda}$ in place of $\{\mathfrak{p}\}_{\xi}$ will resolve Difficulty (1).

In general given a lattice $\Gamma$, we let $(V^*_{\mathbb{C}})^\Gamma$ denote those elements $\tau \in V^*_{\mathbb{C}}$ such that any time $e^{\langle \tau, \gamma \rangle} = 1$ with $\gamma \in \Gamma$, then in fact $\langle \tau, \gamma \rangle = 0$. This property will play an important role in analyzing holomorphicity of the functions $\mu$; see Proposition \ref{prop_mu_holomorphic}. Note that $\tilde{\xi} \in (V^*_{\mathbb{C}})^{\tilde{\Lambda}}$. Using the Berline-Vergne local Euler-Maclaurin formula, we may resolve Difficulty (2).
\begin{theorem}[see Theorem \ref{thm_levi_cone_euler_maclaurin} for a more refined statement]
    Let $\Lambda$ be a lattice, $\mathfrak{p}$ be a rational polytope, and $\xi \in (V_{\mathbb{C}}^*)^\Lambda$. Let $\mathfrak{f} \in \{\mathfrak{p}\}_{\xi}$ and $s \in \textnormal{lin}(\mathfrak{f})^\perp$ a rational point. The function $S_{s+ \Lambda_{\mathfrak{f}^\perp}}(\textnormal{LC}_{\mathfrak{f}}^{\mathfrak{p}}(\xi); \tau)$ with $\tau \in (\textnormal{lin}(\mathfrak{f})^\perp)^*_{\mathbb{C}}$ is holomorphic at $\tau = \xi$.
\end{theorem}

Given a rational point in $s \in V$, we shall use the notation
\begin{gather*}
    S_{s + \Lambda}(\mathfrak{q}; \alpha) := \sum_{\lambda \in \mathfrak{q} \cap (s + \Lambda)} e^{\langle \alpha, \lambda \rangle}.
\end{gather*}
In \eqref{eqn_reduce_to_adapted} of Section \ref{subsec_xi_lambda}, one finds a simple relationship between $S_{\Lambda}(\mathfrak{p}; \xi)$ and $S_{\gamma + \tilde{\Lambda}}(\mathfrak{p}; \tilde{\xi})$ for all $\gamma \in \Lambda/\tilde{\Lambda}$. Using this formula, we may thus always reduce to the case that $\xi \in (V^*_{\mathbb{C}})^\Lambda$. 

Finally, resolving Difficulty (3) is the content of Proposition \ref{prop_ortho_lattice}. See Section \ref{sec_decomposing} for the precise statement. As corollaries we obtain the following versions of the degeneration Brion's formula in the discrete setting. These are stated for $\xi \in (V^*_{\mathbb{C}})^\Lambda$. A version for arbitrary $\xi$ may be immediately obtained by combining any of the below formulas with \eqref{eqn_reduce_to_adapted}.

\begin{corollary} [Degenerate Brion's formula: discrete setting, version 1] \label{cor_discrete_1}
    Suppose $\mathfrak{p}$ is a rational polytope with respect to a lattice $\Lambda$. Let $\xi \in (V_{\mathbb{C}}^*)^{\Lambda}$. Then
    \begin{gather}
        S_{\Lambda}(\mathfrak{p}; \xi) = \sum_{\mathfrak{f} \in \{\mathfrak{p}\}_{\xi}} \sum_{[\omega] \in \Lambda/(\Lambda_{\mathfrak{f}} \oplus \Lambda_{\mathfrak{f}^\perp})} S_{[\omega^{\mathfrak{f}}] + \Lambda_\mathfrak{f}} (\mathfrak{f}^{\mathfrak{f}}; 0) \cdot S_{[\omega^{\mathfrak{f}^\perp}] + \Lambda_{\mathfrak{f}^\perp}} (\textnormal{LC}_{\mathfrak{f}}^{\mathfrak{p}}(\xi); \xi), \label{eqn_brion_1}
    \end{gather} 
    and each term is well-defined.
\end{corollary}
The finite sum over $\Lambda/(\Lambda_{\mathfrak{f}} \oplus \Lambda_{\mathfrak{f}^\perp})$ naturally arises in addressing Difficulty (3) from above. Note that $S_{[\omega^{\mathfrak{f}}] + \Gamma_{\mathfrak{f}}} (\mathfrak{f}^{\mathfrak{f}}; 0)$ simply counts the number of points of the shifted lattice $[\omega^{\mathfrak{f}}] + \Lambda_{\mathfrak{f}}$ inside of $\mathfrak{f}^{\mathfrak{f}}$. 

We can in turn apply the Euler-Maclaurin formula to each $S_{[\omega^{\mathfrak{f}}] + \Lambda_{\mathfrak{f}}} (\mathfrak{f}^{\mathfrak{f}}; 0)$. We shall use the notation $\mu^{s + \Lambda}_W(\mathfrak{q}; \cdot)$ as shorthand for $\mu^{\Lambda}_W(-s + \mathfrak{q}; \xi)$. We get:
\begin{corollary}[Degenerate Brion's formula: discrete setting, version 2] \label{thm_intro_degen_brion_2}
     Suppose $\mathfrak{p}$ is a rational polytope with respect to a lattice $\Lambda$. Let $\xi \in (V_{\mathbb{C}}^*)^{\Lambda}$. Then
     \begin{align}
         S_{\Lambda}(\mathfrak{p}; \xi) = & \sum_{\mathfrak{g} \in \{\mathfrak{p}\}_{\xi}} \textnormal{vol}^{\Lambda_{\mathfrak{g}}}(\mathfrak{g}) \notag \\
         \times & \Bigg( \sum_{\mathfrak{g} \subseteq \mathfrak{f} \in \{\mathfrak{p}\}_{\xi}} \sum_{[\omega] \in \Lambda/(\Lambda_{\mathfrak{f}} \oplus \Lambda_{\mathfrak{f}^\perp})} \mu^{([\omega^{\mathfrak{f}}] + \Lambda_{\mathfrak{f}})^{\mathfrak{g}^\perp}}_{\textnormal{lin}(\mathfrak{f}) \cap \textnormal{lin}(\mathfrak{g})^\perp} (\mathfrak{t}_{\mathfrak{g}}^{\mathfrak{f}}; 0) \cdot S_{[\omega^{\mathfrak{f}^\perp}] + \Lambda_{\mathfrak{f}^\perp}}(\textnormal{LC}_{\mathfrak{f}}^{\mathfrak{p}}(\xi); \xi) \Bigg), \label{eqn_intro_degen_brion_2}
     \end{align}
     and each term is well-defined.
\end{corollary}

Berline-Vergne prove that $\mu(\mathfrak{k}; \tau)$ is holomorphic at $\tau = 0$ with rational value for any rational cone $\mathfrak{k}$ in any rational inner product space. Given a lattice $\Gamma$ in a vector space, the notation $\textnormal{vol}^\Gamma$ refers to the usual volume form normalized so that $\Gamma$ has covolume 1. 

Notice that the innermost sum in \eqref{eqn_intro_degen_brion_2} can be interpreted as computing $S_{\Lambda^{\mathfrak{f}^\perp}}(\textnormal{LC}_{\mathfrak{f}}^{\mathfrak{p}}(\xi); \xi)$, but where we weight points in the coset $[\omega^{\mathfrak{f}^\perp}] + \Lambda_{\mathfrak{f}^\perp}$ by $\mu^{([\omega^{\mathfrak{f}}] + \Lambda_{\mathfrak{f}})^{\mathfrak{g}^\perp}}_{\textnormal{lin}(\mathfrak{f}) \cap \textnormal{lin}(\mathfrak{g})^\perp} (\mathfrak{t}_{\mathfrak{g}}^{\mathfrak{f}}; 0)$. When $\mathfrak{f} = \mathfrak{g}$, we have that $\mu^{([\omega^{\mathfrak{f}}] + \Lambda_{\mathfrak{f}})^{\mathfrak{g}^\perp}}_{\textnormal{lin}(\mathfrak{f}) \cap \textnormal{lin}(\mathfrak{g})^\perp} (\mathfrak{t}_{\mathfrak{g}}^{\mathfrak{f}}; 0) = 1$ for every $[\omega]$, so the innermost sum is exactly $S_{\Lambda^{\mathfrak{g}^\perp}}(\textnormal{LC}_{\mathfrak{g}}^{\mathfrak{p}}(\xi); \xi)$. If $\mathfrak{g}$ is $\xi$-maximal, then this is the only term in the coefficient of $\textnormal{vol}^{\Lambda_{\mathfrak{g}}}(\mathfrak{g})$. If all vertices are $\xi$-maximal, then the formula reduces to
\begin{gather*}
    S_{\Lambda}(\mathfrak{p}; \xi) = \sum_{\mathfrak{v} \in \textnormal{Vert}(\mathfrak{p})} S_{\Lambda}(\mathfrak{t}_{\mathfrak{v}}^{\mathfrak{p}}; \xi),
\end{gather*}
which is simply Brion's formula \eqref{fake_brion_discrete}. If $\xi = 0$, then this formula simply reduces to 
\begin{gather*}
    S_{\Gamma}(\mathfrak{p}; 0) = \sum_{\mathfrak{g} \in \textnormal{Face}(\mathfrak{p})} \textnormal{vol}^{\Gamma_{\mathfrak{g}}}(\mathfrak{g}) \cdot \mu^{\Lambda^{\mathfrak{g}^\perp}}_{\textnormal{lin}(\mathfrak{g})^{\perp}}(\mathfrak{t}_{\mathfrak{g}}^{\mathfrak{p}}; 0),
\end{gather*}
which is the formula obtained by Pommersheim-Thomas \cite{pommersheim_thomas} and rederived by Berline-Vergne \cite{berline_vergne}.

We also can apply the Euler-Maclaurin formula directly to $I(\mathfrak{p}; \xi)$. We obtain an alternate expression like in \eqref{eqn_intro_degen_brion_2} which is simpler but less directly geometric.
\begin{corollary}[Degenerate Brion's formula: discrete setting, version 3] \label{thm_intro_degen_brion_3}
    Suppose $\mathfrak{p}$ is a rational polytope with respect to a lattice $\Lambda$. Let $\xi \in (V_{\mathbb{C}}^*)^{\Lambda}$. Then
    \begin{align}
        S_{\Lambda}(\mathfrak{p}; \xi) = & \sum_{\mathfrak{g} \in \{\mathfrak{p}\}_{\xi}} \textnormal{vol}^{\Lambda_{\mathfrak{g}}}(\mathfrak{g}) \notag \\
        \times & \Bigg( \sum_{\mathfrak{f} \supseteq \mathfrak{g}, \  \mathfrak{f} = \mathfrak{g} \textnormal{ or } \mathfrak{f} \notin \{\mathfrak{p}\}_\xi} \mu^{\Lambda^{\mathfrak{f}^\perp}}_{\textnormal{lin}(\mathfrak{f})^\perp}(\mathfrak{t}_{\mathfrak{f}}^{\mathfrak{p}}; \xi) \cdot I^{(\Lambda^{\mathfrak{g}^\perp})_{\mathfrak{f}}}(\textnormal{LC}_{\mathfrak{g}}^{\mathfrak{f}}(\xi); \xi) \Bigg), \label{eqn_intro_degen_brion_3}
    \end{align}
    and each term is well-defined.
\end{corollary}
Notice that in \eqref{eqn_intro_degen_brion_3}, we sum over $\mathfrak{f} \supseteq \mathfrak{g}$ such that either $\mathfrak{f} = \mathfrak{g}$ or $\mathfrak{f} \notin \{\mathfrak{p}\}_{\xi}$. This is in contrast to \eqref{eqn_intro_degen_brion_2} where given $\mathfrak{g}$, we sum over $\mathfrak{f} \supseteq \mathfrak{g}$ such that $\mathfrak{f} \in \{\mathfrak{p}\}_{\xi}$. From \eqref{eqn_intro_degen_brion_3} we can also see that if $\mathfrak{g}$ is $\xi$-maximal, then, by the Euler-Maclaurin formula, the coefficient of $\textnormal{vol}^{\Lambda_{\mathfrak{g}}}(\mathfrak{g})$ is $S_{\Lambda^{\mathfrak{g}^\perp}}(\mathfrak{t}_{\mathfrak{g}}^{\mathfrak{p}}; \xi) = S_{\Lambda^{\mathfrak{g}^\perp}}(\textnormal{LC}_{\mathfrak{g}}^{\mathfrak{f}}(\xi); \xi)$. As before, if all vertices are $\xi$-maximal, then this formula simply reduces to Brion's formula \eqref{fake_brion_discrete}.

Suppose we take positive integer dilates of a fixed rational polytope $\mathfrak{p}$. We then get that
\begin{align}
    S_{\Lambda}(t \cdot \mathfrak{p}; \xi) = & \sum_{\mathfrak{g} \in \{\mathfrak{p}\}_\alpha} \textnormal{vol}^{\Lambda_{\mathfrak{g}}}(\mathfrak{g}) \cdot t^{\textnormal{dim}(\mathfrak{g})} e^{t \cdot \langle \xi, \mathfrak{g} \rangle} \notag \\
    \times & \Bigg( \sum_{\mathfrak{f} \supseteq \mathfrak{g}, \ \mathfrak{f} = \mathfrak{g} \textnormal{ or } \mathfrak{f} \notin \{\mathfrak{p}\}_\xi} \mu_{\textnormal{lin}(\mathfrak{f})^\perp}^{\Lambda^{\mathfrak{f}^\perp}}(t \cdot \mathfrak{f}^{\mathfrak{f}^\perp} + {}^0 \mathfrak{t}_{\mathfrak{f}}^{\mathfrak{p}}; 0) \cdot I^{(\Lambda^{\mathfrak{g}^\perp})_{\mathfrak{f}}} ({}^0 \textnormal{LC}_{\mathfrak{g}}^{\mathfrak{f}}(\xi); \xi) \Bigg). \label{brion_discrete_dilation}
\end{align}
The element $t \cdot \mathfrak{f}^{\mathfrak{f}^\perp}$ is periodic in $t$ inside of $\textnormal{lin}(\mathfrak{f})^\perp/(\Lambda^{\mathfrak{f}^\perp})$, and thus so is $\mu_{\textnormal{lin}(\mathfrak{f})^\perp}^{\Lambda^{\mathfrak{f}^\perp}}(t \cdot \mathfrak{f}^{\mathfrak{f}^\perp} + {}^0 \mathfrak{t}_{\mathfrak{f}}^{\mathfrak{p}}; 0)$. In case $\mathfrak{f}$ contains a lattice point, then its period is one. If all vertices are lattice points, then the period is one for all faces. Note that clearly $\textnormal{vol}^{\Lambda_{\mathfrak{g}}}(\mathfrak{g})$ and ${}^0 \textnormal{LC}_{\mathfrak{g}}^{\mathfrak{f}}(\xi)$ do not depend on $t$. Thus in the case of an integral polytope, \eqref{brion_discrete_dilation} is an expression involving a sum of terms of the form polynomial times exponential. In the case of a rational polytope, we instead have an expression as a sum of terms of the form quasi-polynomial times exponential. In case $\xi = 0$, we simply get the Ehrhart quasi-polynomial. Thus our degenerate Brion's formula simultaneously generalizes Brion's original formula and the Ehrhart quasi-polynomial. We point out that our formula is also of a nice form for understanding how $S_{\Lambda}(\mathfrak{p}(h); \xi)$ varies in $h$ over a family of polytopes with the same normal fan.

A much weaker version of the degenerate version of Brion's formula in the discrete case appeared in \cite{peterson}. There we proved the following: let $\mathfrak{p}$ be a polytope defined by $\bigcap_j \{\langle \eta_{\mathfrak{f}_j}^{\mathfrak{p}}, x \rangle \geq h^0_j\}$ and such that $\mathfrak{p}$ is integral with respect to a lattice $\Lambda$ (then $\eta_{\mathfrak{f}_j}^\mathfrak{p}$ is the outward normal to the facet $\mathfrak{f}_j$). Let $\mathfrak{p}(h)$ be the polytope defined by $\bigcap_j \{\langle \eta_{\mathfrak{f}_j}^{\mathfrak{p}}, x \rangle \geq h_j \}$ with $h = (h_1, \dots, h_k)$. Let $\mathcal{P}$ be the collection of integral polytopes of the form $\mathfrak{p}(h)$ which have the same normal fan as $\mathfrak{p}$. Let $\xi \in V^*$ be a real functional; this implies that $\xi \in (V^*_{\mathbb{C}})^\Lambda$. Then there exist polynomials $Q_\mathfrak{v}(h)$ in $h$ such that
\begin{gather}
    S_\Lambda(\mathfrak{p}(h); \xi) = \sum_{\mathfrak{v} \in \textnormal{Vert}(\mathfrak{p})} Q_\mathfrak{v}(h) e^{\langle \xi, \mathfrak{v}(h) \rangle}. \label{early_degen_brion}
\end{gather}
Furthermore $Q_{\mathfrak{v}}(h)$ has degree at most equal to the number of edges containing $\mathfrak{v}$ on which $\xi$ is constant. The proof involved an application of L'Hopital's rule to the generic case of Brion's formula. It required non-canonically choosing a direction $\zeta$ in which to perturb a degenerate choice of $\xi$ to make it generic, and the expression for $Q_\mathfrak{v}(h)$ was complicated and depended on $\zeta$. Furthermore, no geometric interpretation was given for $Q_{\mathfrak{v}}(h)$. Such a result as \eqref{early_degen_brion} follows immediately from Corollary \ref{thm_intro_degen_brion_2} or Corollary \ref{thm_intro_degen_brion_3}. Furthermore, Corollary \ref{thm_intro_degen_brion_2} and Corollary \ref{thm_intro_degen_brion_3} are much more geometric, precise, hold for general rational polytopes rather than merely integral ones, and (when combined with \eqref{eqn_reduce_to_adapted}) hold for any $\xi \in V^*_{\mathbb{C}}$. 

\subsection{A detailed example} \label{sec_example_intro}
We now discuss in detail an example illustrating the results of this paper. Let $\mathfrak{p}$ be the triangle in $\mathbb{R}^2$ with vertices given by $\mathfrak{v}_1 = (-1, 0)$, $\mathfrak{v}_2 = (1, 2)$, and $\mathfrak{v}_3 = (1, 0)$. Let $\mathfrak{e}$ be the edge between $\mathfrak{v}_1$ and $\mathfrak{v}_2$. Let $\alpha = (\alpha_1, \alpha_2)$. Let $\Lambda = \mathbb{Z}^2$. Brion's formula tells us that
\begin{align*}
    I(t \cdot \mathfrak{p}; \alpha) &= \frac{e^{-t \alpha_1}}{\alpha_1 (\alpha_1 + \alpha_2)} + \frac{e^{t(\alpha_1 + 2 \alpha_2)}}{(-\alpha_1 - \alpha_2)(-\alpha_2)} + \frac{e^{t \alpha_2}}{(-\alpha_1)(\alpha_2)}, \\
    S_{\Lambda}(t \cdot \mathfrak{p}; \alpha) &= \frac{e^{-t \alpha_1}}{(1 - e^{\alpha_1})(1 - e^{\alpha_1 + \alpha_2})} + \frac{e^{t(\alpha_1 + 2 \alpha_2)}}{(1 - e^{-\alpha_1 - \alpha_2})(1 - e^{-\alpha_2})} + \frac{e^{t \alpha_1}}{(1 - e^{-\alpha_1})(1 - e^{\alpha_2})}. 
\end{align*}
In both expressions, the terms in order correspond to the contributions from $\mathfrak{t}_{\mathfrak{v}_1}^{\mathfrak{p}}$, $\mathfrak{t}_{\mathfrak{v}_2}^{\mathfrak{p}}$, and $\mathfrak{t}_{\mathfrak{v}_3}^{\mathfrak{p}}$. Note that we are using the formulas \eqref{cone_integral} and \eqref{discrete_simplicial}.

Now suppose we set $\xi = (1, -1)$. Then $\{\mathfrak{p}\}_\xi = \{\mathfrak{e}, \mathfrak{v}_1, \mathfrak{v}_2, \mathfrak{v}_3\}$. See Figure \ref{fig_intro_example} for an illustration of the decomposition from Theorem \ref{thm_polytope_decomp_intro}. Concretely, we have that
\begin{align*}
    \textnormal{LC}_{\mathfrak{e}}^{\mathfrak{p}}(\xi) &= \mathfrak{t}_{\mathfrak{e}}^{\mathfrak{p}} = \mathds{1}\{(-0.5, 0.5) + t \cdot (1, -1) : t \in [0, \infty) \}, \\
    \textnormal{LC}_{\mathfrak{v}_1}^{\mathfrak{p}}(\xi) &= \mathfrak{t}_{\mathfrak{v}_1}^{\mathfrak{p}} - \mathfrak{t}_{\mathfrak{v}_1}^{\mathfrak{e}} \times \mathfrak{t}_{\mathfrak{e}}^{\mathfrak{p}} = - \mathds{1} \{(-1, 0) + t_1 \cdot (1,0) + t_2 \cdot (1, -1) : t_1 \in (0, \infty), t_2 \in [0, \infty) \}, \\
    \textnormal{LC}_{\mathfrak{v}_2}^{\mathfrak{p}}(\xi) &= \mathfrak{t}_{\mathfrak{v}_2}^{\mathfrak{p}} - \mathfrak{t}_{\mathfrak{v}_2}^{\mathfrak{e}} \times \mathfrak{t}_{\mathfrak{e}}^{\mathfrak{p}} = - \mathds{1} \{ (1, 2) + t_1 \cdot (0, -1) + t_2 \cdot (1, -1) : t_1 \in (0, \infty), t_2 \in [0, \infty) \}, \\
    \textnormal{LC}_{\mathfrak{v}_3}^{\mathfrak{p}}(\xi) &= \mathfrak{t}_{\mathfrak{v}_3}^{\mathfrak{p}} = \mathds{1} \{(1, 0) + t_1 \cdot (-1, 0) + t_2 \cdot (1, 0) : t_1, t_2 \in [0, \infty) \}.
\end{align*}

Furthermore, $\Lambda_{\mathfrak{e}}$ is generated by $(1, 1)$, the lattice $\Lambda_{\mathfrak{e}^\perp}$ is generated by $(1, -1)$, and $\Lambda_{\mathfrak{e}} \oplus \Lambda_{\mathfrak{e}^\perp}$ is index 2 inside of $\Lambda$. We can choose as our coset representatives $(0, 0)$ and $(1, 0)$. See Figure \ref{fig_lattices_example}. One can easily calculate:

\begin{minipage}{0.5\textwidth}
\begin{align*}
    S_{\Lambda_{\mathfrak{e}}}(t \cdot \mathfrak{e}^{\mathfrak{e}}; 0) &= 2t + 1, \\
    S_{\Lambda_{\mathfrak{e}}}(\textnormal{LC}_{\mathfrak{e}}^{\mathfrak{p}}(\xi); \xi) &= \frac{e^{-t}}{1 - e^2},
\end{align*}
\end{minipage}
\begin{minipage}{0.5\textwidth}
\begin{align*}
    S_{[(1, 0)] + \Lambda_{\mathfrak{e}}}(t \cdot \mathfrak{e}^{\mathfrak{e}}; 0) &= 2t, \\
    S_{[(1, 0)] + \Lambda_{\mathfrak{e}^\perp}}(\textnormal{LC}_{\mathfrak{e}}^{\mathfrak{p}}(\xi); \xi) &= \frac{e^{-t +1}}{1 - e^2}. 
\end{align*}
\end{minipage}\\

Using the formulas \eqref{degenerate_brion_cont} and \eqref{eqn_brion_1} we get
\begin{align}
    I(t \cdot \mathfrak{p}; \xi) &= \Big( 2 \sqrt{2} t \Big) \Big( \frac{\sqrt{2} e^{-t}}{2} \Big) - \frac{e^{-t}}{1 \cdot 2} - \frac{e^{-t}}{1 \cdot 2} + \frac{e^{t}}{(-1) \cdot (-1)}, \notag \\
    S_{\Lambda}(t \cdot \mathfrak{p}; \xi) &= \Bigg( (2t+1) \Big(\frac{e^{-t}}{1 - e^2} \Big) + (2 t) \Big( \frac{e^{-t+1}}{1 - e^2} \Big) \Bigg) \notag \\
    & - \Bigg( \frac{e^{-t}}{(1 - e^1)(1 - e^2)} - \frac{e^{-t}}{(1 - e^1)} \Bigg) - \Bigg( \frac{e^{-t}}{(1 - e^1)(1-e^2)} - \frac{e^{-t}}{(1-e^1)} \Bigg) \label{eqn_interesting_cones} \\
    & + \frac{e^t}{(1 - e^{-1})(1 - e^{-1})}. \notag 
\end{align}
We remark that $\textnormal{LC}_{\mathfrak{v}_1}^{\mathfrak{p}}(\xi)$ can be expressed as
\begin{gather*}
    - \mathds{1} \{ (-1, 0) + t_1 \cdot (1, 0) + t_2 \cdot (1, -1) : t_i \in [0, \infty) \} + \mathds{1} \{ (-1, 0) + t \cdot (1, 0) : t \in [0, \infty) \}.
\end{gather*}
A similar expression exists for $\textnormal{LC}_{\mathfrak{v}_2}^{\mathfrak{p}}$. When applying the valuation $I(\cdot; \cdot)$ to such as expression, the second term is zero because the corresponding cone is not full-dimensional. However, when applying $S_\Lambda(\cdot; \cdot)$ to it, both terms contribute non-trivially. This explains why each parenthesized term in \eqref{eqn_interesting_cones} consists of two individual terms.

\begin{figure} 
\begin{tikzpicture}[scale=0.75]
\fill[color = blue, opacity = 0.2] (-1, 0) -- (1,0) -- (1,2) -- cycle;
\draw[color = blue, line width = 1] (-1, 0) -- (1,0) -- (1, 2) -- cycle;
\node[xshift = -7mm] at (-1, 0) {$(-1, 0)$};
\node[yshift = -5mm] at (1, 0) {$(1, 0)$};
\node[xshift = -3mm, yshift = 3mm] at (0, 1) {$\mathfrak{e}$};
\node[yshift = 5mm] at (1, 2) {$(1, 2)$};

\node at (2.5, 1) {$\equiv$};

\begin{scope}[xshift = 5 cm]
    \fill[color = red, opacity = 0.2] (-1, 0) -- (1, 2) -- (3, 0) -- (1, -2) -- cycle; 
    \node[xshift = -7mm] at (-1, 0) {$(-1, 0)$};
    \node[yshift = 5mm] at (1, 2) {$(1, 2)$}; 
    \draw[color = red, line width = 1] (1, -2) -- (-1, 0) -- (1, 2) -- (3, 0);
\end{scope}

\node at (8.5, 1) {$-$};
\begin{scope}[xshift = 11cm]
    \fill[color = red, opacity = 0.2] (1, -2) -- (-1, 0) -- (1, 0) -- cycle;
    \draw[color = red, line width = 1] (-1, 0) -- (1, -2);
    \draw[color = red, line width = 1, dashed] (-1, 0) -- (1, 0);
    \draw[color = red] (-1, 0) circle (2 pt);
    \node[xshift = -7mm] at (-1, 0) {$(-1, 0)$};
\end{scope}

\node at (12.5, 1) {$-$};

\begin{scope}[xshift = 12cm]
    \fill[color = red, opacity = 0.2] (1, 0) -- (1, 2) -- (3, 0) -- cycle;
    \draw[color = red, line width = 1] (1, 2) -- (3, 0);
    \draw[color = red, line width = 1, dashed] (1, 2) -- (1, 0);
    \node[yshift = 5mm] at (1, 2) {$(1, 2)$};
    \draw[color = red] (1, 2) circle (2 pt);
\end{scope}

\node at (15.5, 1) {$+$};

\begin{scope}[xshift = 17cm]
    \fill[color = red, opacity = 0.2] (-1, 0) -- (1, 0) -- (1, 2) -- cycle;
    \draw[color = red, line width = 1] (-1, 0) -- (1, 0);
    \draw[color = red, line width = 1] (1, 0) -- (1, 2);
    \node[yshift = -5mm] at (1, 0) {$(1, 0)$};
\end{scope}
\end{tikzpicture}
\caption{Illustration of the decomposition of $\mathfrak{p}$ as in Theorem \ref{thm_polytope_decomp_intro} with respect to $\xi = (1, -1)$.}
\label{fig_intro_example}
\end{figure}

\begin{figure}
\centering
\begin{tikzpicture}[scale=1.5]

\node at (0,-1) [circle,fill,inner sep=2.0pt, color = brown]{};
\node at (1,-1) [circle,fill,inner sep=2.0pt, color = blue]{};
\node at (2,-1) [circle,fill,inner sep=2.0pt, color = brown]{};
\node at (3,-1) [circle,fill,inner sep=2.0pt, color = blue]{};
\node at (4,-1) [circle,fill,inner sep=2.0pt, color = brown]{};
\node at (0,0) [circle,fill,inner sep=2.0pt, color = blue]{};
\node at (1,0) [circle,fill,inner sep=2.0pt, color = brown]{};
\node at (2,0) [circle,fill,inner sep=2.0pt, color = blue]{};
\node at (3,0) [circle,fill,inner sep=2.0pt, color = brown]{};
\node at (4,0) [circle,fill,inner sep=2.0pt, color = blue]{};
\node at (0,1) [circle,fill,inner sep=2.0pt, color = brown]{};
\node at (1,1) [circle,fill,inner sep=2.0pt, color = blue]{};
\node at (2,1) [circle,fill,inner sep=2.0pt, color = brown]{};
\node at (3,1) [circle,fill,inner sep=2.0pt, color = blue]{};
\node at (4,1) [circle,fill,inner sep=2.0pt, color = brown]{};
\node at (0,2) [circle,fill,inner sep=2.0pt, color = blue]{};
\node at (1,2) [circle,fill,inner sep=2.0pt, color = brown]{};
\node at (2,2) [circle,fill,inner sep=2.0pt, color = blue]{};
\node at (3,2) [circle,fill,inner sep=2.0pt, color = brown]{};
\node at (4,2) [circle,fill,inner sep=2.0pt, color = blue]{};
\node at (0,3) [circle,fill,inner sep=2.0pt, color = brown]{};
\node at (1,3) [circle,fill,inner sep=2.0pt, color = blue]{};
\node at (2,3) [circle,fill,inner sep=2.0pt, color = brown]{};
\node at (3,3) [circle,fill,inner sep=2.0pt, color = blue]{};
\node at (4,3) [circle,fill,inner sep=2.0pt, color = brown]{};

\begin{scope}[xshift = 1cm, yshift = 1cm]
\fill[color = red, opacity = 0.2] (-1, 0) -- (1, 2) -- (4, -1) -- (2, -3) -- cycle; 
\node[xshift = -7mm] at (-1, 0) {$(-1, 0)$};
\node[yshift = 5mm] at (1, 2) {$(1, 2)$}; 
\draw[color = red, line width = 1] (2, -3) -- (-1, 0) -- (1, 2) -- (4, -1);
\node[xshift = -3mm, yshift = 3mm] at (0, 1) {$\mathfrak{e}$};
\end{scope}
\end{tikzpicture}
\caption{The red polyhedron corresponds to $\mathfrak{e}^\mathfrak{e} \times \textnormal{LC}_{\mathfrak{e}}^{\mathfrak{p}}(\xi)$. The lattice $\Lambda$ corresponds to the union of the blue and brown points. The blue points correspond to the sublattice $\Lambda_{\mathfrak{e}} \oplus \Lambda_{\mathfrak{e}^\perp}$, and the brown points to the shifted sublattice $(1, 0) + \Lambda_{\mathfrak{e}} \oplus \Lambda_{\mathfrak{e}^\perp}$. See also Figure \ref{fig_lattice}.}
\label{fig_lattices_example}
\end{figure}

\subsection{A view towards applications}
As in \cite{peterson}, our original motivation for seeking such formulas arose from harmonic analysis on symmetric spaces and affine buildings (such as Bruhat-Tits buildings). For example, suppose we have a symmetric space of non-compact type $G/K$ where $G$ is a semisimple Lie group and $K$ is a maximal compact subgroup. By the Cartan decomposition, a $K$-invariant function $f$ on $G/K$ is determined by its values on $\mathfrak{a}^+$, a corresponding Weyl chamber chosen in a way respecting the Cartan involution giving rise to $K$. We can thus instead integrate $f$ on $\mathfrak{a}^+$ with respect to a certain Jacobian factor which has an explicit expression as a finite sum of exponential functions \cite{helgason}:
\begin{gather}
    c \cdot \prod_{\alpha \in \Phi^+} (e^{\langle \alpha, x \rangle} - e^{\langle -\alpha, x \rangle})^{m_\alpha}, \label{symmetric_space_jacobian}
\end{gather}
where $\Phi^+ \subset \mathfrak{a}^*$ is the set of positive relative roots with respect to $\mathfrak{a}$, $m_\alpha$ is the multiplicity of the relative root $\alpha$, and $c$ is an explicit constant. Notice that the dominating term in \eqref{symmetric_space_jacobian} is $e^{\langle 2 \rho, x \rangle}$ where $2 \rho = \sum_{\alpha \in \Phi^+} m_\alpha \cdot \alpha$.

At times it is advantageous to consider polyhedral Finsler metrics on $G/K$ in place of the Riemannian metric (see, e.g., \cite{kapovich_leeb, horofunction_compactification, brumley_matz, lutsko_weich_wolf, davalo_riestenberg}); such metrics arise from polytopes inside of $\mathfrak{a}^+$. Then computing the volume of a ball of radius $t$ in the symmetric space in such a metric reduces to integrating a finite sum of exponential functions over the $t$-th dilate of a fixed polytope. One would like to use Brion's formula to compute this, but if the exponent is constant on some face of the polytope, then one cannot use it. However, one may in such an instance instead use the degenerate Brion's formula as in \eqref{brion_cont_dilation}.

Something similar happens in affine buildings. Suppose, for simplicity, $G$ is a split adjoint semisimple algebraic group over a non-archimedean local field $F$ with maximal torus $T$. Let $q$ be the order of the residue field. Let $K$ be a special maximal compact subgroup. Then $G/K$ is the collection of special vertices of the associated Bruhat-Tits building. If we have a $K$-invariant function on $G/K$, then by the Cartan decomposition, it is determined by its values on the special vertices inside a sector based at the unique vertex stabilized by $K$. These vertices are in turn naturally parametrized by the intersection of the coweight lattice with the Weyl chamber $\mathfrak{a}^+$ inside of the vector space $\mathfrak{a} := X_\bullet(T) \otimes_{\mathbb{Z}} \mathbb{R}$ where $X_\bullet(T)$ is the cocharacter lattice $\textnormal{Hom}_{\mathbb{Z}}(F^\times, T)$; such lattice points are called \textit{dominant coweights}. The cardinality of the preimage of a dominant coweight $\lambda$ under Cartan projection has an expression of the form 
\begin{gather}
    c(\lambda, q^{-1}) \cdot q^{\langle 2 \rho, \lambda \rangle}, \label{building_jacobian}
\end{gather}
where $c(\lambda, q^{-1})$ is constant on the interior of each face of the Weyl chamber \cite{macdonald_spherical}. One may interpret this coefficient as the ``Jacobian factor'' that arises when integrating a bi-$K$-invariant function on $G$. Like in the symmetric spaces setting, at times it is advantageous to work with polyhedral metrics on the building using polytopes in $\mathfrak{a}^+$ (see, e.g., \cite{landvogt, BT_theory_berkovich, suh, gao_2022, gao_thesis, peterson, gielen_2025}). To compute the size of a ball of radius $t \in \mathbb{N}$ in this metric, one must integrate \eqref{building_jacobian} over the lattice points in the defining polytope. One could use Brion's formula, but, as in the symmetric spaces case, at times one must instead use degenerate Brion's formula as in \eqref{brion_discrete_dilation}. We plan to further discuss applications of our formulas to computing volume growth on symmetric spaces and affine buildings in future work.

In the aforementioned work \cite{peterson}, we proved a rudimentary version of degenerate Brion's formula \eqref{early_degen_brion}. There we applied Brion's formula and its degenerate version in four distinct ways in service of proving a ``quantum ergodicity'' result on Bruhat-Tits buildings associated to $\textnormal{PGL}(3)$; more specifically the main result of \cite{peterson} is a delocalization result for joint eigenfunctions of the spherical Hecke algebra on compact quotients of the underlying Bruhat-Tits building. One way in which we used Brion's formula was to compute the volume of balls with respect to a polyhedral metric as alluded to in the previous paragraph. A more sophisticated application has to do with computing the eigenvalue of Hecke operators associated to polyhedral balls. Roughly speaking, this computation reduces to integrating the spherical function associated to a given choice of Satake parameters over lattice points lying in dilates of the polytope defining the metric. For $\textnormal{PGL}(n)$, the dominant coweights are naturally parametrized by partitions $\lambda = (\lambda_1, \dots, \lambda_{n-1}, 0)$ with $\lambda_1 \geq \dots \geq \lambda_{n-1} \geq 0$. The spherical functions can be expressed in terms of the Hall-Littlewood polynomials \cite{macdonald_symmetric}:
\begin{gather*}
    P_\lambda(x_1, \dots, x_n; t) := \frac{d(t)}{c(\lambda, t)} \sum_{\sigma \in S_n} \sigma.\Big(x_1^{\lambda_1} \dots x_n^{\lambda_n} \prod_{i < j} \frac{x_i - t x_j}{x_i - x_j} \Big),
\end{gather*}
where $d(t)$ is a constant depending only on $t$, $\sigma$ is an element of the symmetric group $S_n$ permuting the $x_i$'s, and $c(\lambda, t)$ is the same as in \eqref{building_jacobian}. To compute the eigenvalue of the Hecke operator associated to a polytopal ball, one must integrate the following expression over dominant coweights $\lambda$ lying in the defining polytope:
\begin{gather}
    d(q^{-1}) \cdot q^{\langle \rho, \lambda \rangle} \sum_{\sigma \in S_n} \sigma.\Big(q^{\langle s, \lambda \rangle} \prod_{i < j} \frac{q^{s_i} - q^{-1} q^{s_j}}{q^{s_i} - q^{s_j}}\Big), \label{spherical_fn}
\end{gather}
where $q^s = (q^{s_1}, \dots, q^{s_n})$ are the Satake parameters. When $q^{s}$ is fixed, \eqref{spherical_fn} varies in $\lambda$ as a finite sum of exponential functions, and thus one may use Brion's formula. This type of computation explains why it is desirable to work with $\xi \in V^*_{\mathbb{C}}$ rather than restrict to real functionals, as the Satake parameters (and in the setting of symmetric spaces, the spectral parameters), naturally vary over all of $V^*_{\mathbb{C}}$. The Satake parameters corresponding to purely imaginary functionals, i.e. the tempered ones, are in fact often the most important ones to study.

Finally, the most sophisticated application in \cite{peterson} which required the full strength of degenerate Brion's formula involves computing the size of the intersection of two polyhedral balls with different centers. The strategy is roughly as follows: first we parametrize points in the intersection with respect to the relative position of such points relative to the centers of the balls. The possible relative positions are then parametrized by lattice points in some polytope, and the number of points associated to a given relative position is exponential in these coordinates. Furthermore, as we change the radii of the balls or the location of the centers, the parametrizing polytope simply has its supporting hyperplanes shifted. Thus in particular the normal fan stays the same, which allows for a good uniform bound as we vary the radii and centers of the polytopal balls in the building. See \cite{thesis} for a visualization of the technique for the building of $\textnormal{PGL}(2)$, which is the infinite regular tree.

Brion's formula is already known to have applications in representation theory (see, e.g., \cite{finis_lapid, makhlin, feigin_makhlin, asgari_kiumars}). We believe our degenerate version of Brion's formula may have further applications to representation theory as well as to analysis on symmetric spaces and affine buildings. In particular, our degenerate Brion's formula may be seen in certain situations as a substitute for Laplace's method and the method of stationary phase in settings in which such methods are not applicable because of the non-differentiable nature of polytopes, or in the discrete case because we are taking sums of exponential functions rather than integrating. 

Since its inception, Brion's formula has been deeply linked to toric geometry and symplectic geometry, and we believe that our degenerate Brion's formula should have such a link as well. We briefly describe a result in equivariant cohomology which in some sense may also be interpreted as a ``degenerate Brion's formula''. The main source for this discussion is Chapter 7 of \cite{BGV}, particularly Theorem 7.13 (the Atiyah-Bott-Berline-Vergne localization theorem). Suppose $T$ is a torus with Lie algebra $\mathfrak{t}$. Let $V = \mathfrak{t}^*$. Let $M$ be a smooth toric variety with associated polytope $\mathfrak{p} \subset V$. Let $\omega$ be an equivariantly closed differential form on $M$. Let $\xi \in \mathfrak{t}$, and let $M_0$ denote the fixed points of the one parameter subgroup generated by $\xi$. Let $\mathcal{N}$ be the normal bundle to $M_0$ inside of $M$. Then for $\alpha \in \mathfrak{t}$ sufficiently close to $\xi$ we have that 
\begin{gather}
    \int_M \omega(\alpha) = \int_{M_0} (2 \pi)^{\textnormal{rank}(\mathcal{N})/2} \frac{\omega(\alpha)}{\chi_{\mathfrak{t}}(\mathcal{N})(\alpha)}, \label{eqn_localization}
\end{gather}
where $\chi_{\mathfrak{t}}(\mathcal{N})(\alpha)$ is the equivariant Euler class of $\mathcal{N}$. If we let $\omega$ be the exponential of the equivariant symplectic form on $M$ then, upon passing to action-angle coordinates, the LHS of \eqref{eqn_localization} turns into $I(\mathfrak{p}; \alpha)$. In case $\xi$ is generic, we have that $M_0$ consists of finitely many points (exactly the points mapping to the vertices of $\mathfrak{p}$ under the moment map), and the RHS of \eqref{eqn_localization} turns into the RHS of Brion's formula \eqref{fake_brion} (or really terms like in \eqref{cone_integral}). In this sense \eqref{eqn_localization} is a ``degenerate'' generalization of Brion's formula, as in general $\chi_\mathfrak{t}(\mathcal{N})(\alpha)$ is invertible near $\alpha = \xi$ so the RHS of \eqref{eqn_localization} makes sense. The locus $M_0$ corresponds to the preimage under the moment map of the union of the $\xi$-maximal faces, to the RHS of \eqref{eqn_localization} is really a sum over the $\xi$-maximal faces of some expression involving only the local geometry near each $\xi$-maximal face. The above formula suggests that the version of degenerate Brion's formula in our paper (Corollary \ref{thm_degen_brion_cont_intro}) may provide an explicit formula for the inverse of the equivariant Euler class $\chi_t(\mathcal{N})(\alpha)$, at least after an inner product is chosen. It would be very interesting to further explore the exact connection between these formulas, though we do not pursue such questions in the present work.

More generally our results contribute towards the analysis of multivariate meromorphic functions with hyperplane singularities which, in addition to the setting of this paper (and the associated connections to toric geometry and symplectic geometry), show up in renormalization in quantum field theory and the study of multiple zeta functions. See e.g. the works of Guo-Paycha-Zhang \cite{GPZ_zeta, GPZ_euler_maclaurin, GPZ_laurent, GPZ_locality}.

\section{Background and notation}
\subsection{Inner product spaces} Let $V$ be a real $n$-dimensional vector space with an inner product $\langle \cdot, \cdot \rangle$. This allows us to identify $V$ with $V^*$. We shall use $\langle \cdot, \cdot \rangle$ to denote both the inner product and the pairing between $V$ and $V^*$. We can extend the inner product on $V^*$ to a complex bilinear pairing on $V^*_{\mathbb{C}}$ which we still denote by $\langle \cdot, \cdot \rangle$. We can also define a natural Hermitian inner product on $V^*_{\mathbb{C}}$ defined by $\langle \alpha, \beta \rangle_H := \langle \bar{\alpha}, \beta \rangle$. This also defines a norm on $V^*_{\mathbb{C}}$ by $\|\alpha\|^2 := \langle \alpha, \alpha \rangle_H = \langle \bar{\alpha}, \alpha \rangle$. If $W$ is a subspace of $V$, let $P_W$ denote orthogonal projection onto $W$.

\subsection{Polyhedra, polytopes, and cones} \label{sec_polytopes} Recall that a \textit{polyhedron} $\mathfrak{q}$ in $V$ is the intersection of finitely many closed half-spaces in $V$, and a \textit{polytope} is a bounded polyhedron. We shall use ``mathfrak'' letters to denote polyhedra. If $\mathfrak{q}$ is a $k$-dimensional polyhedron, we shall at times write $\mathfrak{q}^{(k)}$ to emphasize its dimension. A codimension one face of a polyhedron is called a \textit{facet}. A \textnormal{cone} $\mathfrak{k}$ is a polyhedron of the form $\{x + \sum_{i} a_i v_i : a_i \geq 0, v_i \in V\}$. We let ${}^0\mathfrak{k}$ denote the same cone shifted to be based at the origin, i.e. ${}^0 \mathfrak{k} := \{\sum_i a_i v_i : a_i \geq 0, v_i \in V\}$. If $\mathfrak{k}$ contains no lines, then we call it a \textit{pointed cone}, and it has a unique vertex. A cone is called \textit{simplicial} if there are exactly $n$ defining vectors $v_i$, and they are all linearly independent. In that case we call $\{v_i\}$ a \textit{conical basis}. Such cones are necessarily pointed.

Let $\mathfrak{f}$ be a face of a polyhedron $\mathfrak{q}$. Let $\textnormal{aff}(\mathfrak{f})$ be the affine subspace generated by $\mathfrak{f}$, and let $\textnormal{lin}(\mathfrak{f})$ be the linear subspace parallel to $\textnormal{aff}(\mathfrak{f})$. Let $x$ be any point in the relative interior of $\mathfrak{f}$. The \textit{tangent cone }of $\mathfrak{f}$ in $\mathfrak{q}$ is defined as
\begin{gather*}
    \mathfrak{s}_{\mathfrak{f}}^{\mathfrak{q}} := \{x + w : x + \varepsilon \cdot w \in \mathfrak{p} \textnormal{ for $\varepsilon > 0$ small enough} \}.
\end{gather*}
The definition is independent of the choice of $x$. Intuitively, this is the cone generated by all directions we can move in from $x$ and stay inside of $\mathfrak{p}$. Note that unless $\mathfrak{f}$ is a vertex, this cone is not pointed as it contains $\textnormal{aff}(\mathfrak{f})$. 

Suppose $\mathfrak{f}$ and $\mathfrak{g}$ are polyhedra in $V$. We shall often use the notation $\mathfrak{f}^{\mathfrak{g}} := P_{\textnormal{lin}(\mathfrak{g})}(\mathfrak{f})$ and $\mathfrak{f}^{\mathfrak{g}^\perp} := P_{\textnormal{lin}(\mathfrak{g})^\perp}(\mathfrak{f})$. One of the most important definitions in this paper is that of the \textit{transverse cone} of $\mathfrak{f}$ in $\mathfrak{q}$ defined as
\begin{gather*}
    \mathfrak{t}_{\mathfrak{f}}^{\mathfrak{q}} := P_{\textnormal{lin}(\mathfrak{f})^\perp} (\mathfrak{s}_{\mathfrak{f}}^{\mathfrak{q}}) = (\mathfrak{s}_{\mathfrak{f}}^{\mathfrak{q}})^{\mathfrak{f}^\perp}. 
\end{gather*}
This cone is always pointed with vertex at the point $\mathfrak{f}^{\mathfrak{f}^\perp}$. If $\mathfrak{v}$ is a vertex, then $\mathfrak{s}_{\mathfrak{v}}^{\mathfrak{q}} = \mathfrak{t}_{\mathfrak{v}}^{\mathfrak{q}}$. A polytope $\mathfrak{p}$ is called \textit{simple} if $\mathfrak{t}_{\mathfrak{v}}^{\mathfrak{p}}$ is a simplicial cone for all vertices $\mathfrak{v}$.

Suppose $\mathfrak{p} \subset V$ is an $n$-dimensional polytope. For each facet $\mathfrak{f}_i$ of $\mathfrak{p}$, we have an associated supporting half-space $\langle - \eta_{\mathfrak{f}_i}^{\mathfrak{p}}, x \rangle + h_i^0 \geq 0$ where $\eta_{\mathfrak{f}_i}^{\mathfrak{p}}$ is the outward pointing unit normal. Suppose $\mathfrak{p}$ has $\ell$ facets. Let $h^0 = (h_1^0, \dots, h_\ell^0) \in \mathbb{R}^\ell$, and let $h = (h_1, \dots, h_\ell)$ denote an arbitrary vector in $\mathbb{R}^\ell$. Let $\mathfrak{p}(h)$ be defined as
\begin{gather*}
    \mathfrak{p}(h) := \{x \in V : \langle -\eta_{\mathfrak{f}_j}^\mathfrak{p}, x \rangle + h_j \geq 0 \textnormal{ for all $j$}\}.
\end{gather*}
We clearly have $\mathfrak{p} = \mathfrak{p}(h^0)$.

Recall that a \textit{polyhedral fan} is a finite collection of pointed cones such that the intersection of any two of the constituent cones is also in the set, and every face of each of the constituent cones is also in the set. The \textit{normal fan} of $\mathfrak{p}$, denoted $\mathcal{N}(\mathfrak{p})$, is the polyhedral fan in $V^*$ whose constituent cones are defined as follows. For each face $\mathfrak{f} \subseteq \mathfrak{p}$, we define $N_\mathfrak{f}$, which is the cone of linear functionals on $V$ whose restriction to $\mathfrak{p}$ is maximized on $\mathfrak{f}$:
\begin{gather*}
    N_\mathfrak{f} := \bigg\{ \xi \in V^* : \mathfrak{f} \subseteq \{x \in \mathfrak{p} : \langle \xi, x \rangle = \max_{y \in \mathfrak{p}} \langle \xi, y \rangle \} \bigg\}.
\end{gather*}
There is a natural inclusion-reversing poset isomorphism between the poset of faces of $\mathfrak{p}$ and the poset of cones in the normal fan $\mathcal{N}(\mathfrak{p})$. 

Let $N(h^0) \subset \mathbb{R}^\ell$ denote those $h$'s such that $\mathcal{N}(\mathfrak{p}(h)) = \mathcal{N}(\mathfrak{p}(h^0))$. Clearly $N(h^0)$ is invariant under scaling by $\mathbb{R}_+^*$ which corresponds to dilating the polytope. Given $h \in N(h^0)$, we can naturally identify each face $\mathfrak{f}$ of $\mathfrak{p} = \mathfrak{p}(h^0)$ with a face $\mathfrak{f}(h)$ of $\mathfrak{p}(h)$. Note that if $h \in N(h^0)$, then
\begin{gather*}
    \textnormal{lin}(\mathfrak{f}) = \{ x : \langle \xi, x \rangle = 0 \textnormal{ for all $\xi \in N_{\mathfrak{f}}$} \} = \textnormal{lin}(\mathfrak{f}(h)).
\end{gather*}
Therefore for all $h \in N(h^0)$, the face $\mathfrak{f}(h)$ is parallel to the face $\mathfrak{f}$. 

\subsection{The $\xi$-decomposition}

Let $\xi \in V^*_{\mathbb{C}}$. A face $\mathfrak{f} \subset \mathfrak{p}$ is called \textit{$\xi$-constant} if $\xi$ is constant on $\mathfrak{f}$; in that case, the value $\langle \xi, \mathfrak{f} \rangle$ is well-defined. The face $\mathfrak{f}$ is called \textit{$\xi$-maximal} if it is $\xi$-constant and there does not exist a bigger face $\mathfrak{h} \supset \mathfrak{f}$ which is also $\xi$-constant. Note that two distinct $\xi$-maximal faces may intersect non-trivially. We call a face \textit{$\xi$-good} if it is $\xi$-constant and is contained in a unique $\xi$-maximal face. Let $\{\mathfrak{p}\}_\xi$ denote the set of all $\xi$-constant faces; we call this the \textit{$\xi$-decomposition} of $\mathfrak{p}$. Given a fixed $\xi^0 \in V^*_{\mathbb{C}}$, we let $L(\xi^0)$ denote those $\xi \in V^*_{\mathbb{C}}$ such that $\{\mathfrak{p}\}_\xi = \{\mathfrak{p}\}_{\xi^0}$.

We now make several observations. If $\xi \in V^*$ is a real functional, then the union of all faces on which $\langle \xi, \cdot \rangle$ achieves its maximum forms a $\xi$-maximal face which is also $\xi$-good. The analogous statement is true with maximum replaced with minimum. If $\mathfrak{p}$ is a simple polytope, then every $\xi$-constant face is also $\xi$-good. A face $\mathfrak{f}$ is $\xi$-constant if and only if $\xi$ lies in the complex linear span of $N_{\mathfrak{f}}$. The $\xi$-decomposition is the same for all $\mathfrak{p}(h)$ with $h \in N(h^0)$. The set $L(\xi^0)$ is invariant under multiplication by $\mathbb{C}^*$.

\subsection{Flags}
A \textit{flag} in $\mathfrak{p}$ is a chain of faces $\mathfrak{h}_1 \subset \dots \subset \mathfrak{h}_\ell$. We say that the flag \textit{starts at} $\mathfrak{h}_1$ and \textit{ends at} $\mathfrak{h}_\ell$. Suppose $\mathfrak{h}_j$ is $i_j$-dimensional. We say that the flag is \textit{saturated} if $i_{j+1} = i_{j}+1$ for all $1 \leq j \leq \ell-1$. We let $\textnormal{sFl}_{\mathfrak{f}}^{\mathfrak{p}}$ denote the set of all saturated flags that start at $\mathfrak{f}$ and end at $\mathfrak{p}$. We say that a flag that starts at $\mathfrak{f}$ is \textit{$\xi$-adapted} if every face in the flag other than $\mathfrak{f}$ is not $\xi$-constant. We say that a flag is \textit{$\xi$-maladapted} if every face in the flag is $\xi$-constant. Let $\textnormal{sFl}_{\mathfrak{f}}^{\mathfrak{p}}(\xi)$ denote the set of all saturated $\xi$-adapted flags from $\mathfrak{f}$ to $\mathfrak{p}$. If $\xi = 0$, then if $\mathfrak{f} \subseteq \mathfrak{p}$ the set $\textnormal{sFl}_{\mathfrak{f}}^{\mathfrak{p}}(\xi)$ is empty unless $\mathfrak{f} = \mathfrak{p}$, in which case its only element is the flag $\mathfrak{p}$. Let $\textnormal{mFl}_{\mathfrak{f}}^{\mathfrak{p}}(\xi)$ denote the set of all $\xi$-maladapted flags that start at $\mathfrak{f}$. This set is empty if $\mathfrak{f}$ is not $\xi$-constant.

\subsection{Brion's formula, continuous version}
The material in this section may be found in \cite{barvinok_book_2}.

We have a natural Lebesgue measure on $V$ coming from the inner product, i.e. we define any parallelogram whose associated Gram matrix has determinant $1$ to have Lebesgue measure equal to 1. We shall denote this measure by $\textnormal{vol}$, or at times by $\textnormal{vol}^{V}$ to emphasize that it is the Lebesgue measure on $V$ (rather than on a subspace). Given a linear subspace $W$, we also have a natural Lebesgue measure coming from the restriction of the inner product to $W$. Similarly we have a Lebesgue measure on any affine subspace parallel to $W$. We shall denote this by $\textnormal{vol}^W$. If $W = 0$, then we set $\textnormal{vol}^{W}(W) = 1$. If $\mathfrak{f}$ is a face of a polytope $\mathfrak{q}$, then we use the notation $\textnormal{vol}^{\mathfrak{f}}$ both to denote the Lebesgue measure on $\textnormal{aff}(\mathfrak{f})$ and on $\textnormal{lin}(\mathfrak{f})$ (it will be clear from context which is meant). Later on when we also choose a lattice in $V$, there will be other natural normalizations of the Lebesgue measure on subspaces of $V$ and the relationship between them will be important. 

Let $\mathfrak{k}$ be an $n$-dimensional pointed cone with vertex $\mathfrak{v}$. Let $({}^0 \mathfrak{k})^*$ be the dual cone:
\begin{gather*}
    (^0\mathfrak{k})^* = \{\eta \in V^* : \langle \eta, x \rangle \geq 0 \textnormal{ for all $x \in {}^0\mathfrak{k}$} \}.
\end{gather*}
For $\alpha$ in the interior of $-(^0\mathfrak{k})^*$, we define
\begin{gather*}
    I(\mathfrak{k}; \alpha) := \int_{\mathfrak{k}} e^{\langle \alpha, x \rangle} d \textnormal{vol}(x). 
\end{gather*}
The integral converges absolutely for such $\alpha$. Notice that $I(\mathfrak{k}; \alpha) = e^{\langle \alpha, \mathfrak{v} \rangle} I({}^0\mathfrak{k}; \alpha)$.

Suppose $\mathfrak{k}$ is a simplicial cone with vertex at the origin and with conical basis $v_1, \dots, v_n$. Let $\square(v_1, \dots, v_n)$ denote the semi-open parallelpiped formed by $v_1, \dots, v_n$:
\begin{gather*}
    \square(v_1, \dots, v_n) := \Big\{\sum_j a_j v_j : a_j \in [0, 1) \Big\}.
\end{gather*}
Then for $\alpha \in \textnormal{Int}(-\mathfrak{k}^*)$ we have
\begin{gather}
    I(\mathfrak{k}; \alpha) = \textnormal{vol}(\square(v_1, \dots, v_n)) \prod_{j = 1}^n -\langle \alpha, v_j \rangle^{-1}. \label{cone_integral}
\end{gather}
This function has a natural meromorphic continuation to $\alpha \in V^*_{\mathbb{C}}$ with singularities on $\bigcup_j \{\langle \alpha, v_j \rangle = 0 \}$. More generally we have the following:
\begin{proposition}[\cite{barvinok}] \label{barvinok_homogeneous}
Let $\mathfrak{k}$ be a pointed cone. Then the function $I(\mathfrak{k}; \alpha)$ has a unique meromorphic continuation on $V^*_{\mathbb{C}}$ which is rational and $(-n)$-homogeneous in $\alpha$. Let $v_1, \dots, v_r$ denote a complete set of representatives of vectors on the extremal rays of ${}^0\mathfrak{k}$. Then the set of singular points of the analytic continuation of $I(\mathfrak{k}; \alpha)$ is the union of the hyperplanes $\{\alpha \in V^*_{\mathbb{C}} : \langle \alpha, v_j \rangle = 0 \}$. 
\end{proposition}

If $\mathfrak{q}$ is a polyhedron, then it may be expressed as $\mathfrak{p} + \mathfrak{k}$ where $\mathfrak{p}$ is a polytope and $\mathfrak{k}$ is a cone. Furthermore the choice of $\mathfrak{k}$ is unique and is known as the \textit{recession cone}. The locus on which $\int_{\mathfrak{q}} e^{\langle \alpha, x \rangle} d \textnormal{vol}(x)$ converges absolutely is the interior of $-({}^0 \mathfrak{k})^*$, and thus if $\mathfrak{k}$ is pointed we can define $I(\mathfrak{q}; \alpha)$ through meromorphic continuation as before. The resulting function is meromorphic with singularities on the union of hyperplanes orthogonal to the extremal rays of ${}^0\mathfrak{k}$. If $\mathfrak{k}$ is not pointed, or equivalently if $\mathfrak{q}$ contains a line, then we set $I(\mathfrak{q}; \alpha) = 0$. If $\mathfrak{p}$ is a polytope, then clearly $I(\mathfrak{p}; \alpha)$ is an entire function. 

Let $[\mathfrak{q}]$ denote the indicator function of the polyhedron $\mathfrak{q}$. Any $\mathbb{R}$-linear combination of such indicator functions is called a \textit{virtual polyhedron}. An $\mathbb{R}$-linear combination of indicator functions of pointed cones with the same vertex is called a \textit{virtual cone}. Let $\mathcal{P}(V)$ denote the vector space of virtual polyhedra in $V$. Let $\mathcal{P}_0(V)$ denote the subspace generated by polyhedra containing lines. Given $f, g \in \mathcal{P}(V)$, we write $f \equiv g$ to mean that $f-g \in \mathcal{P}_0(V)$. A linear map from $\mathcal{P}(V)$ to another vector space is called a \textit{valuation}. The map $I(\cdot; \alpha)$ extends to a valuation on $\mathcal{P}(V)$ mapping to the space of meromorphic functions on $V_{\mathbb{C}}^{*}$. This follows almost immediately from the linearity of integration, except that we have \textit{defined} $I(\cdot; \alpha)$ to be zero for elements in $\mathcal{P}_0(V)$. One can show that if we write $\mathfrak{q} \in \mathcal{P}_0(V)$ as a sum of polyhedra not containing lines, then the sum of the associated meromorphic functions is indeed zero. Heuristically this can be thought of as a generalization of the following: the Fourier transform of a line $\mathbb{R} \cdot x$ is the ``Dirac delta'' supported on the subspace $\{\langle \alpha, x \rangle = 0: \alpha \in i V^*\}$. This ``generalized function'' is identically zero outside of this locus. If a meromorphic function is identically zero on some open subset of $i V^*$, then it must be identically zero. 

\begin{example} 
For example, the Fourier transform of $f(x) = 1$ on $\mathbb{R}$ is the Dirac delta at $0 \in i\mathbb{R}^*$. On the other hand $1 = \mathds{1}_{[0, \infty)} + \mathds{1}_{(-\infty, 0]} - \mathds{1}_{\{0\}}$. We have from \eqref{cone_integral}:
\begin{align*}
    I([0, \infty); \alpha) &= -\frac{1}{\alpha} \\
    I((-\infty, 0]; \alpha) &= \frac{1}{\alpha} \\
    I(\{0\}; \alpha) &= 0.
\end{align*}
\end{example}

We shall at times make use of the following proposition whose proof is immediate.
\begin{proposition} \label{prop_multiplicative}
    Suppose $V$ splits as an orthogonal direct sum: $V = W_1 \oplus W_2$. Suppose $\mathfrak{p}_i$ is a polyhedron in $W_i$. Suppose $\alpha \in V^*_{\mathbb{C}}$ splits as $\alpha = \alpha_1 + \alpha_2$ with respect to this orthogonal decomposition. Then we have the identity of meromorphic functions:
    \begin{gather*} 
    I(\mathfrak{p}_1 \times \mathfrak{p}_2; \alpha) = I(\mathfrak{p}_1; \alpha_1) \cdot I(\mathfrak{p}_2; \alpha_2).
    \end{gather*}
\end{proposition}

We may now state Brion's formula: continuous version.
\begin{theorem}[Brion's formula: continuous version \cite{brion}] \label{thm_brion}
    Let $\mathfrak{p} \subset V$ be an $n$-dimensional polytope. We have the following equality of meromorphic functions in $\alpha \in V^*_\mathbb{C}$:
    \begin{gather*}
        I(\mathfrak{p}; \alpha) = \sum_{\mathfrak{v} \in \textnormal{Vert}(\mathfrak{p})} I(\mathfrak{t}_{\mathfrak{v}}^{\mathfrak{p}}; \alpha) = \sum_{\mathfrak{v} \in \textnormal{Vert}(\mathfrak{p})} e^{\langle \alpha, \mathfrak{v} \rangle} I({}^0 \mathfrak{t}_{\mathfrak{v}}^{\mathfrak{p}}; \alpha).
    \end{gather*}
\end{theorem}
\noindent This theorem was first proven by Brion using toric varieties \cite{brion}. Barvinok \cite{barvinok} later gave a more elementary proof.

We shall also make use of the following theorem which may be seen as a generalization of the Brianchon-Gram theorem.
\begin{theorem} \label{brianchon_gram}
    Let $\mathfrak{q}$ be a polyhedron. Then, modulo $\mathcal{P}_0(V)$ we have
    \begin{gather*}
        [\mathfrak{q}] \equiv \sum_{\mathfrak{v} \in \textnormal{Vert}(\mathfrak{q})} [\mathfrak{t}_{\mathfrak{v}}^{\mathfrak{q}}].
    \end{gather*}
\end{theorem}
\noindent Notice that the valuation property of $I(\cdot; \xi)$, and the fact that $\mathcal{P}_0(V)$ is in the kernel, immediately imply Theorem \ref{thm_brion}.

\subsection{Rational inner product spaces}
Now suppose $V$ contains a lattice $\Lambda$. The points of $\Lambda$ are called \textit{integral}, and the rational points in $V$, denoted $V_{\mathbb{Q}}$, are defined to be elements of the form $\sum_i a_i \lambda_i$ with $a_i \in \mathbb{Q}$ and $\lambda_i \in \Lambda$. A basis of $V$ is called \textit{integral} if it is also a basis of $\Lambda$. The dual lattice $\Lambda^*$ is the lattice in $V^*$ defined by $\{\xi : \langle \xi, \lambda \rangle \in \mathbb{Z} \textnormal{ for all } \lambda \in \Lambda\}$.

A (real) subspace $W$ is called \textit{rational} if $W \cap \Lambda$ is a lattice in $W$. We furthermore assume that $\langle \cdot, \cdot \rangle$ takes rational values on $\Lambda$. This ensures that if $W$ is rational, then $W^\perp$ is also rational. An affine subspace is called rational if it is of the form $x + W$ where $x \in V_{\mathbb{Q}}$ and $W$ is a rational (linear) subspace. Note that rational affine subspaces might not contain any elements of $\Lambda$. 

Given a lattice $\Gamma$ in a subspace $W$, we define $\textnormal{gr}(\Gamma)$ to be the square root of the determinant of the Gram matrix associated to any basis of $\Gamma$. This is equal to the covolume of $\Gamma$ with respect to the restriction of the inner product to $W$, or phrased another way, $\textnormal{vol}^W(W/\Gamma) = \textnormal{gr}(\Gamma)$. We define $\textnormal{vol}^\Gamma$ to be the Lebesgue measure with respect to which $\Gamma$ has covolume 1. If $\mathfrak{q}$ is a polyhedron in an affine space parallel to $W$, we shall at times use the notation $I^\Gamma(\mathfrak{q}; \xi)$ to mean that we take the same integral as in the definition of $I(\mathfrak{q}; \xi)$, but we integrate with respect to $\textnormal{vol}^{\Gamma}$. We have the simple relation:
\begin{gather}
    I(\mathfrak{q}; \xi) = \textnormal{gr}(\Gamma) \cdot I^\Gamma(\mathfrak{q}; \xi). \label{eqn_renormalize_integral}
\end{gather}

Suppose $W$ is a rational subspace. Let $\Lambda_W := W \cap \Lambda$. Let $\Lambda^W := P_{W}(\Lambda)$. Because the inner product is rational, $\Lambda^W$ is a lattice which contains $\Lambda_W$ as a finite index sublattice. Thus $W$ carries three natural Lebesgue measures: $\textnormal{vol}^W, \textnormal{vol}^{\Lambda_W}$, and $\textnormal{vol}^{\Lambda^W}$. We shall say more about the relationship between these in Section \ref{sec_decomposing}.

\subsection{The $(\xi, \Lambda)$-decomposition} \label{subsec_xi_lambda}
Suppose $\mathfrak{p}$ is a rational polytope in $V$ with respect to $\Lambda$. Let $\xi \in V^*_{\mathbb{C}}$. Let $\bar{\Lambda}$ denote the sublattice of $\Lambda$ on which $e^{\langle \xi, \lambda \rangle} = 1$. Let $\tilde{V}$ denote the real subspace of $V$ spanned by $\bar{\Lambda}$. Let $\tilde{\xi}$ be the functional defined to be 0 on $\tilde{V}$ and equal to $\xi$ on $\tilde{V}^\perp$. Let $\tilde{\Lambda}$ be defined to be the sublattice of $\Lambda$ on which $e^{\langle \xi, \lambda \rangle} = e^{\langle \tilde{\xi}, \lambda \rangle}$.

Note that $\bar{\Lambda}$ is a finite index sublattice of $\Lambda_{\tilde{V}}$, and $\Lambda_{\tilde{V}}$ is a finite index sublattice of $\Lambda^{\tilde{V}}$. We have a map $\beta_1: \Lambda \to \Lambda^{\tilde{V}}$ given by orthogonal projection. We can then compose this with the map $\beta_2: \Lambda^{\tilde{V}} \to \Lambda^{\tilde{V}}/\bar{\Lambda}$. The kernel of $\beta_2 \circ \beta_1$ is exactly $\tilde{\Lambda}$, and thus $\tilde{\Lambda}$ is a finite index sublattice of $\Lambda$, and $\Lambda/\tilde{\Lambda} \simeq \Lambda^{\tilde{V}}/\bar{\Lambda}$. Let $\gamma_1, \dots, \gamma_k$ be a set of representatives for $\Lambda^{\tilde{V}}/\bar{\Lambda}$. Given $\lambda \in \Lambda$, we can write $\lambda = \lambda_1 + \lambda_2 + \gamma_j$ where $\lambda_1 \in \tilde{V}^\perp$ and $\lambda_2 \in \bar{\Lambda}$. Then
\begin{gather*}
    e^{\langle \xi, \lambda \rangle} = e^{\langle \xi, \lambda_1 + \lambda_2 + \gamma_j \rangle} = e^{\langle \tilde{\xi}, \lambda_1 + \lambda_2 + \gamma_j \rangle} \cdot e^{\langle \xi, \gamma_j \rangle} = e^{\langle \xi, \gamma_j \rangle} \cdot e^{\langle \tilde{\xi}, \lambda \rangle}.
\end{gather*}
Obviously the value of $e^{\langle \xi, \gamma_j \rangle}$ does not depend on the choice of representative $\gamma_j$. This implies the general formula:
\begin{gather}
    S_{\Lambda}(\mathfrak{p}; \xi) = \sum_{[\gamma] \in \Lambda/\tilde{\Lambda}} e^{\langle \xi, \gamma^{\tilde{V}} \rangle} S_{[\gamma] + \tilde{\Lambda}}(\mathfrak{p}; \tilde{\xi}). \label{eqn_reduce_to_adapted}
\end{gather}

The functional $\tilde{\xi}$ has the following crucial property: if $e^{\langle \tilde{\xi}, \lambda \rangle} = 1$ with $\lambda \in \tilde{\Lambda}$, then $\langle \xi, \lambda \rangle = 0$. This property will be important when analyzing the holomorphicity of the functions $\mu$ appearing in the Euler-Maclaurin formula. In general given a lattice $\Gamma$, define the set $(V^*_{\mathbb{C}})^\Gamma$ as those elements $\tau \in V^*_{\mathbb{C}}$ such that if $e^{\langle \tau, \gamma \rangle} = 1$ with $\gamma \in \Gamma$, then $\langle \tau, \gamma \rangle = 0$. Thus $\tilde{\xi} \in (V^*_{\mathbb{C}})^{\tilde{\Lambda}}$. If $\xi \in (V^*_{\mathbb{C}})^\Lambda$, then $\xi = \tilde{\xi}$. Note that any purely real functional is necessarily in $(V^*_{\mathbb{C}})^\Gamma$ for any lattice $\Gamma$. On the other hand the only purely imaginary functional in $(V^*_{\mathbb{C}})^\Gamma$ which is of the form $2 \pi i$ times a rational functional is the zero functional.

We define $\{\mathfrak{p}\}_{\xi, \Lambda} := \{\mathfrak{p}\}_{\tilde{\xi}}$ and call it the $(\xi, \Lambda)$-\textit{decomposition} of $\mathfrak{p}$. A face $\mathfrak{f} \in \{\mathfrak{p}\}_{\tilde{\xi}}$ if and only if there exists a finite index sublattice of $\Lambda^{\mathfrak{f}}$ on which $e^{\langle \xi, \lambda \rangle} = 1$. 

\subsection{Brion's formula, discrete version}
The material for this section can be found in \cite{barvinok_book_2}. Suppose $\mathfrak{k}$ is a pointed rational cone with vertex $\mathfrak{v}$. Then for $\alpha$ in the interior of $-({}^0 \mathfrak{k})^*$ we define
\begin{gather*}
    S_\Lambda(\mathfrak{k}; \alpha) := \sum_{\lambda \in \mathfrak{k} \cap \Lambda} e^{\langle \alpha, \lambda \rangle}.
\end{gather*}
For such $\alpha$ this expression converges absolutely.

Suppose $\mathfrak{k}$ is simplicial with conical basis $w_1, \dots w_n$. Because $\mathfrak{k}$ is rational, we know that $\Lambda \cap \mathbb{R} \cdot w_j$ is itself a lattice. Let $v_j$ be the unique primitive vector for this lattice which lies in $\mathbb{R}_+ \cdot w_j$. For $\alpha$ in the interior of $-({}^0 \mathfrak{k})^*$ we have
\begin{gather}
    S_\Lambda(\mathfrak{k}; \alpha) = \Big( \sum_{\lambda \in (\mathfrak{v}+\square(v_1, \dots, v_n)) \cap \Lambda} e^{\langle \alpha, \lambda \rangle} \Big) \prod_j \frac{1}{1 - e^{\langle \alpha, v_j \rangle}}. \label{discrete_simplicial}
\end{gather}
We call $\mathfrak{k}$ \textit{unimodular} if $\{v_j\}$ form an integral basis. This is equivalent to $\mathfrak{v} + \square(v_1, \dots, v_n) \cap \Lambda$ only containing a single lattice point.

The function \eqref{discrete_simplicial} has an obvious meromorphic continuation with singularities whenever $e^{\langle \alpha, v_j \rangle} = 1$. This condition may alternatively be expressed as $\langle \alpha, v_j \rangle = 2 \pi i k$ for some $k \in \mathbb{Z}$, or $\langle \alpha + \lambda^*, v_j \rangle = 0$ for some $\lambda^* \in \Lambda^*$.
More generally we have:
\begin{proposition} [\cite{barvinok}]
    Let $\mathfrak{k}$ be a pointed rational cone. Then $S_\Lambda(\mathfrak{k}; \alpha)$ has a meromorphic continuation to $V_{\mathbb{C}}^*$. More specifically: let $v_1, \dots, v_r$ denote the primitive lattice points along each extremal ray of ${}^0\mathfrak{k}$. Then there exists an expression
    \begin{gather*}
        S_\Lambda(\mathfrak{k}; \alpha) = \frac{R^{\mathfrak{k}}(\alpha)}{\prod_j 1 - e^{\langle \alpha, v_j \rangle}},
    \end{gather*}
    where $R^{\mathfrak{k}}(\alpha)$ is of the form
    \begin{gather*}
        R^{\mathfrak{k}}(\alpha) = \sum_{\lambda \in \Lambda} c_\lambda^{\mathfrak{k}} \cdot e^{\langle \alpha, \lambda \rangle}
    \end{gather*}
    with $c_\lambda^{\mathfrak{k}} \in \mathbb{Z}$ and $c_\lambda^{\mathfrak{k}} = 0$ for all by finitely many $\lambda$. The singular locus of $S_\Lambda(\mathfrak{k}; \alpha)$ is the union of the hyperplanes $\{\alpha \in V_{\mathbb{C}}^* : \langle \alpha, v_j \rangle = 2 \pi i k\}$ with $k \in \mathbb{Z}$.
\end{proposition}

Like before, if $\mathfrak{q} = \mathfrak{p} + \mathfrak{k}$ is a rational polyhedron not containing a straight line and with recession cone $\mathfrak{k}$, then the locus on which $\sum_{\lambda \in \mathfrak{q} \cap \Lambda} e^{\langle \alpha, \lambda \rangle}$ converges absolutely is the interior of $-({}^0\mathfrak{k}^*)$, and thus we can define $S_\Lambda(\mathfrak{q}; \alpha)$ through meromorphic continuation as before. The resulting function is meromorphic with singularities on the union of hyperplanes $\{\langle \alpha, v_j \rangle = 2 \pi i k \}$ where $v_j$ is an extremal ray of the cone $\mathfrak{k}$ and $k \in \mathbb{Z}$. Otherwise we set $S_{\Lambda}(\mathfrak{q}; \cdot)$ to be identically zero. Notice that we always have $S_\Lambda(\mathfrak{q}; \alpha + \lambda^*) = S_{\Lambda}(\mathfrak{q}; \alpha)$ for $\lambda^* \in \Lambda^*$. Thus we may think of $S_{\Lambda}(\mathfrak{q}; \alpha)$ as a meromorphic function on $V^*_{\mathbb{C}}/\Lambda^*$, which is the space of characters of $\Lambda$.

Let $\mathcal{P}(V_{\mathbb{Q}})$ be the rational subspace of virtual polyhedra generated by rational linear combinations of rational polyehdra. Let $\mathcal{P}_{0}(V_{\mathbb{Q}})$ be the rational subspace generated by rational linear combinations of rational polyhedra containing lines. Like before, we have that $S_\Lambda(\cdot; \alpha)$ is a valuation, i.e. a linear map from $\mathcal{P}(V_{\mathbb{Q}})$ to meromorphic functions on $V^*_{\mathbb{C}}/\Lambda^*$. Also like before, this follows almost immediately from the linearity of summation, except that we have defined $S_{\Lambda}(\cdot; \alpha)$ to be zero for elements in $\mathcal{P}_0(V_{\mathbb{Q}})$. Heuristically, this can be thought of as a generalization of the following. Given any function on $\Lambda$, we can take its ``Fourier transform'' by integrating against unitary characters of $\Lambda$, i.e. integrating against $e^{\langle \alpha, \lambda \rangle}$ with $\alpha \in (i V^*)/(2 \pi i \Lambda^*)$. Let $\lambda$ be any primitive vector in $\Lambda$. Then the Fourier transform of the indicator function of $\mathbb{Z} \cdot \lambda$ is the ``Dirac delta'' supported on $\{\alpha : e^{\langle \alpha, \lambda \rangle} = 1\} \subset (i V^*)/(2 \pi i \Lambda^*)$. This ``generalized function'' is identically zero outside of this locus. If a meromorphic function on $V^*_{\mathbb{C}}/\Lambda^*$ is identically zero on some open subset of $(i V^*)/(2 \pi i \Lambda^*)$, then it must be identically zero. 

\begin{example} 
For example, the Fourier transform of $f(x) = 1$ on $\mathbb{Z}$ is the Dirac delta at $0 \in (i \mathbb{R})/(2 \pi i \mathbb{Z})$. On the other hand $f(x) = \mathds{1}_{[0, \infty) \cap \mathbb{Z}} + \mathds{1}_{(-\infty, 0] \cap \mathbb{Z}} - \mathds{1}_{\{0\}}$. We have from \eqref{discrete_simplicial}:
\begin{align*}
    S_{\mathbb{Z}}([0, \infty); \alpha) &= \frac{1}{1 - e^{\alpha}} \\
    S_{\mathbb{Z}}((-\infty, 0]; \alpha) &= \frac{1}{1 - e^{-\alpha}} \\
    S_{\mathbb{Z}}(\{0\}; \alpha) &= 1 \\
    \frac{1}{1 - e^{\alpha}} + \frac{1}{1 - e^{-\alpha}} - 1 &= 0.
\end{align*}
\end{example}

\begin{theorem}[Brion's formula: discrete version \cite{brion}] \label{thm_brion_discrete}
    Let $\mathfrak{p} \subset V$ be a rational polytope. We have the following equality of meromorphic functions in $\alpha \in V_{\mathbb{C}}^*$:
    \begin{gather*}
        S_\Lambda(\mathfrak{p}; \alpha) = \sum_{\mathfrak{v} \in \textnormal{Vert}(\mathfrak{p})} S_\Lambda(\mathfrak{t}_{\mathfrak{v}}^{\mathfrak{p}}; \alpha).
    \end{gather*}
\end{theorem}
\noindent Like in the continuous case, this follows immediately from combining the valuation property of $S_{\Lambda}(\cdot; \alpha)$, the fact that $\mathcal{P}_0(V_{\mathbb{Q}})$ is in the kernel, and Theorem \ref{brianchon_gram}.

\section{Degenerate Brion's formula: continuous setting} 

\subsection{Proto-degenerate Brion's formula}
Barvinok's proof of Brion's theorem \cite{barvinok} makes use of Stokes' theorem on polytopes. We shall also make use of an adapted version of Stokes' theorem for polytopes. A similar idea appears in \cite{barvinok3} and \cite{diaz_le_robins}.

 \begin{proposition} [\cite{barvinok}] \label{prop_stokes}
     Suppose $\xi \in V^*_{\mathbb{C}}$ with $\xi \neq 0$. Let $\{\mathfrak{f}_{j}\}$ be the facets of $\mathfrak{p}$. Then
     \begin{gather*}
         \int_{\mathfrak{p}} e^{\langle \xi, x \rangle} d \textnormal{vol}(x) = \frac{1}{\|\xi\|^2} \sum_j \langle \xi, \eta_{\mathfrak{f}_j}^{\mathfrak{p}} \rangle_H \int_{\mathfrak{f}_j} e^{\langle \xi, x \rangle} d \textnormal{vol}^{\mathfrak{f}_j}(x).
     \end{gather*}
 \end{proposition}
 \begin{proof}
    Barvinok \cite{barvinok} shows that for any $\lambda \in V \simeq V^*$ such that $\langle \xi, \lambda \rangle \neq 0$, we have
    \begin{gather*}
        \int_{\mathfrak{p}} e^{\langle \xi, x \rangle} d \textnormal{vol}(x) = \frac{1}{\langle \lambda, \xi \rangle} \sum_j \langle \lambda, \eta_{\mathfrak{f}_j}^{\mathfrak{p}} \rangle \int_{\mathfrak{f}_j} e^{\langle \xi, x \rangle} d\textnormal{vol}^{\mathfrak{f}_j}(x).
    \end{gather*}
    The right hand side is constantly equal to the left hand side as $\lambda$ varies. Furthermore, the right hand side is clearly an analytic function in $\lambda$, hence the identity must also hold true for $\lambda \in V^*_{\mathbb{C}}$. In particular, we can set $\lambda = \bar{\xi}$ which yields the result.
 \end{proof} 

Given $\xi \in V^*_\mathbb{C}$, let $\xi^{\mathfrak{g}}$ denote the restriction of $\xi$ to $\textnormal{lin}(\mathfrak{g})$. Thus $\xi^{\mathfrak{g}} \in \textnormal{lin}(\mathfrak{g})^*_{\mathbb{C}}$. The restriction of the inner product to $\textnormal{lin}(\mathfrak{g})$ induces an inner product on $\textnormal{lin}(\mathfrak{g})^*$. Thus the norm of $\xi^{\mathfrak{g}}$ is well-defined and is non-zero so long as $\xi$ is not constantly zero on $\textnormal{lin}(\mathfrak{g})$. 

By repeated application of Stokes' theorem above, we obtain:
\begin{proposition}[Proto-degenerate Brion's formula; see also \cite{diaz_le_robins}] \label{proto_brion}
    Let $\mathfrak{p}$ be an $n$-dimensional polytope, and let $\xi \in V^*_{\mathbb{C}}$. Then
    \begin{gather*}
        I(\mathfrak{p}; \xi) = \sum_{\mathfrak{g}^{(k)} \in \{\mathfrak{p}\}_\xi} e^{\langle \xi, \mathfrak{g} \rangle} \textnormal{vol}^{\mathfrak{g}}(\mathfrak{g}) \sum_{\mathfrak{h}_k \subset \dots \subset \mathfrak{h}_n \in \textnormal{sFl}^\mathfrak{p}_\mathfrak{g}(\xi)} \prod_{j = k+1}^n \frac{\langle \xi, \eta_{\mathfrak{h}_{j-1}}^{\mathfrak{h}_j} \rangle_H}{\|\xi^{\mathfrak{h}_j}\|^2}.
    \end{gather*}
\end{proposition}
\begin{proof}
    In case $\xi = 0$, then the contribution from the $\xi$-constant face $\mathfrak{p}$ gives the term $\textnormal{vol}(\mathfrak{p})$, and for any lower-dimensional face $\mathfrak{g}$ the set $\textnormal{sFl}_{\mathfrak{g}}^{\mathfrak{p}}(\xi)$ is empty and thus has trivial contribution. Thus the formula holds for $\xi = 0$.

    If $\xi \neq 0$, then we may apply Stokes' theorem. After the first application, we are left with terms that look like
    \begin{gather*}
        \frac{\langle \xi, \eta_{\mathfrak{f}}^{\mathfrak{p}} \rangle_H}{\|\xi\|^2} \int_{\mathfrak{f}^{(n-1)}} e^{\langle \xi, x \rangle} d\textnormal{vol}^{\mathfrak{f}},
    \end{gather*}
    where $\mathfrak{f}$ is a facet. If $\mathfrak{f}$ is $\xi$-constant, then this expression is simply equal to $e^{\langle \xi, \mathfrak{f} \rangle} \textnormal{vol}^{\mathfrak{f}}(\mathfrak{f}) \frac{\langle \xi, \eta_{\mathfrak{f}}^{\mathfrak{p}} \rangle_H}{\|\xi\|^2}$, and clearly $\mathfrak{f} \subset \mathfrak{p}$ is the only element in $\textnormal{sFl}_{\mathfrak{f}}^{\mathfrak{p}}(\xi)$.

    Else suppose that $\mathfrak{f}$ is not $\xi$-constant. Let $\{\mathfrak{g}_j\}$ be the codimension one faces of $\mathfrak{f}$. Let $y$ be any point in $\textnormal{aff}(\mathfrak{f})$. Then $\mathfrak{f} - y \subset \textnormal{lin}(\mathfrak{f})$. Therefore we can apply Stokes' theorem again:
    \begin{align*}
        \int_{\mathfrak{f}} e^{\langle \xi, x \rangle} d\textnormal{vol}^{\mathfrak{f}} = e^{\langle \xi, y \rangle} \int_{\mathfrak{f} - y} e^{\langle \xi^{\mathfrak{f}}, x \rangle} d \textnormal{vol}^{\mathfrak{f}} &= e^{\langle \xi, y \rangle} \frac{1}{\|\xi^{\mathfrak{f}}\|^2} \sum_{j} \langle \xi^{\mathfrak{f}}, \eta_{\mathfrak{g}_j}^{\mathfrak{f}} \rangle_H \int_{\mathfrak{g}_j - y} e^{\langle \xi, x \rangle} d \textnormal{vol}^{\mathfrak{g}_j} \\
        &= \frac{1}{\|\xi^{\mathfrak{f}}\|^2} \sum_{j} \langle \xi^{\mathfrak{f}}, \eta_{\mathfrak{g}_j}^{\mathfrak{f}} \rangle_H \int_{\mathfrak{g}_j} e^{\langle \xi, x \rangle} d \textnormal{vol}^{\mathfrak{g}_j}.
    \end{align*}

We may inductively continue this procedure on each $\mathfrak{g}_j$ until we reach $\xi$-constant faces. We can write $\xi = \xi^{\mathfrak{f}_j} + \xi^{\mathfrak{f}_j^\perp}$ with $\xi^{\mathfrak{f}_j^\perp} \in (\textnormal{lin}(\mathfrak{f}_j)^\perp)^*_{\mathbb{C}}$; we can naturally identify $(\textnormal{lin}(\mathfrak{f_j})^\perp)^*_{\mathbb{C}}$ with those elements in $V^*_{\mathbb{C}}$ which are zero on $\textnormal{lin}(\mathfrak{f}_j)$. We clearly have $\eta_{\mathfrak{f}_{j-1}}^{\mathfrak{f}_j} \in \textnormal{lin}(\mathfrak{f}_j)$, so we get that
\begin{gather*}
    \langle \xi, \eta_{\mathfrak{f}_{j-1}}^{\mathfrak{f}_j} \rangle_H = \langle \xi^{\mathfrak{f}_j} + \xi^{\mathfrak{f}_j^\perp}, \eta_{\mathfrak{f}_{j-1}}^{\mathfrak{f}_j} \rangle_H = \langle \xi^{\mathfrak{f}_j}, \eta_{\mathfrak{f}_{j-1}}^{\mathfrak{f}_j} \rangle_H.
\end{gather*}
Thus, for each $\xi$-constant face $\mathfrak{g}$, we exactly get a contribution like on the right hand side of the formula.

\end{proof}

Suppose $\mathfrak{g}$ is a $\xi$-constant face of $\mathfrak{p}$. We define
\begin{gather}
    {}^0 \textnormal{DBI}^\mathfrak{p}_\mathfrak{g}(\xi) := \sum_{\mathfrak{h}_k \subset \dots \subset \mathfrak{h}_n \in \textnormal{sFl}^\mathfrak{p}_\mathfrak{g}(\xi)} \prod_{j = k+1}^n \frac{\langle \xi, \eta_{\mathfrak{h}_{j-1}}^{\mathfrak{h}_j} \rangle_H}{\|\xi^{\mathfrak{h}_j}\|^2}. \label{dbi}
\end{gather}
We now wish to give a geometric interpretation to ${}^0\textnormal{DBI}^\mathfrak{p}_\mathfrak{g}(\xi)$.

Though Barvinok does not emphasize the following formula, in the course of his proof of Brion's formula he proves the following:
\begin{proposition} [\cite{barvinok}] \label{barvinok_prop}
Suppose $\mathfrak{v}$ is a vertex of $\mathfrak{p}$ such that $\mathfrak{v}$ is $\xi$-maximal. Then
\begin{gather*}
    {}^0\textnormal{DBI}^\mathfrak{p}_\mathfrak{v}(\xi) = \sum_{\mathfrak{h}_0 \subset \dots \subset \mathfrak{h}_n \in \textnormal{sFl}^\mathfrak{p}_\mathfrak{v}} \prod_{j = 1}^n \frac{\langle \xi, \eta_{\mathfrak{h}_{j-1}}^{\mathfrak{h}_j} \rangle_H }{\|\xi^{\mathfrak{h}_j}\|^2} = I({}^0 \mathfrak{t}_{\mathfrak{v}}^{\mathfrak{p}}; \xi).
\end{gather*}
\end{proposition}
\noindent Note that if $\mathfrak{v}$ is $\xi$-maximal, then $\textnormal{sFl}_{\mathfrak{v}}^{\mathfrak{p}}(\xi) = \textnormal{sFl}_{\mathfrak{v}}^{\mathfrak{p}}$. Thus if every vertex is $\xi$-maximal, then Proposition \ref{proto_brion} immediately turns into Brion's formula, i.e. Theorem \ref{thm_brion}. 

As discussed in Ziegler \cite{ziegler}, the normal fan $\mathcal{N}(\mathfrak{p}^{\mathfrak{g}^\perp})$ is the intersection $\mathcal{N}(\mathfrak{p}) \cap (\textnormal{lin}(\mathfrak{g})^\perp)^*$, and the map $P_{\textnormal{lin}(\mathfrak{g})^\perp}$ induces an isomorphism of the poset of faces containing $\mathfrak{g}$ and the poset of faces containing the vertex $\mathfrak{g}^{\mathfrak{g}^\perp}$ in the polytope $\mathfrak{p}^{\mathfrak{g}^\perp}$. This is in turn isomorphic to the poset of faces containing $\mathfrak{g}^{\mathfrak{g}^\perp}$ in the transverse cone $\mathfrak{t}_{\mathfrak{g}}^{\mathfrak{p}}$. Thus if $\mathfrak{f}$ is a face of $\mathfrak{p}$ containing $\mathfrak{g}$, then there is an associated face of $\mathfrak{t}_{\mathfrak{g}}^{\mathfrak{p}}$ (which must itself be a pointed cone) and which can be naturally identified with $\mathfrak{t}_{\mathfrak{g}}^{\mathfrak{f}} \subset \textnormal{lin}(\mathfrak{f})/\textnormal{lin}(\mathfrak{g}) \subset V/\textnormal{lin}(\mathfrak{g}) \simeq \textnormal{lin}(\mathfrak{g})^\perp$. If we have $\mathfrak{g} \subseteq \mathfrak{f} \subset \mathfrak{f}'$ with $\mathfrak{f}$ a facet of $\mathfrak{f}'$, then we also have $\mathfrak{t}_{\mathfrak{g}}^{\mathfrak{f}}$ as a facet of $\mathfrak{t}_{\mathfrak{g}}^{\mathfrak{f}'}$. We then naturally have $\eta_{\mathfrak{t}_{\mathfrak{g}}^{\mathfrak{f}}}^{\mathfrak{t}_{\mathfrak{g}}^{\mathfrak{f}'}} = P_{\textnormal{lin}(\mathfrak{g})^\perp} \eta_{\mathfrak{f}}^{\mathfrak{f}'} =: (\eta_{\mathfrak{f}}^{\mathfrak{f}'})^{\mathfrak{g}^\perp}$.

\begin{proposition} \label{prop_maximal}
    Suppose $\mathfrak{g}$ is a $\xi$-maximal face of $\mathfrak{p}$. Then
    \begin{gather*}
        {}^0 \textnormal{DBI}^\mathfrak{p}_\mathfrak{g}(\xi) = I({}^0 \mathfrak{t}_{\mathfrak{g}}^{\mathfrak{p}}; \xi).
    \end{gather*}
\end{proposition}

\begin{proof}
    First off note that because $\mathfrak{g}$ is $\xi$-maximal, there is no bigger face containing it which is also $\xi$-constant. Therefore $\textnormal{sFl}_{\mathfrak{g}}^{\mathfrak{p}}(\xi) = \textnormal{sFl}_{\mathfrak{g}}^{\mathfrak{p}}$. In light of Proposition \ref{barvinok_prop}, it therefore suffices to show that if $\mathfrak{g} \subseteq \mathfrak{f} \subset \mathfrak{f}'$ are faces such that $\mathfrak{f}$ is a facet of $\mathfrak{f}'$, then 
    \begin{gather*}
        \frac{\langle \xi, \eta_{\mathfrak{f}}^{\mathfrak{f}'} \rangle_H}{\|\xi^{\mathfrak{f}'}\|^2} = \frac{\langle \xi^{\mathfrak{g}^\perp}, \eta_{\mathfrak{t}_{\mathfrak{g}}^{\mathfrak{f}}}^{\mathfrak{t}_{\mathfrak{g}}^{\mathfrak{f}'}} \rangle_H}{\|(\xi^{\mathfrak{g}^\perp})^{\mathfrak{t}_{\mathfrak{g}}^{\mathfrak{f}'}}\|^2}.
    \end{gather*}

We have the orthogonal direct sum $V = \textnormal{lin}(\mathfrak{g}) \oplus (\textnormal{lin}(\mathfrak{f}') \cap \textnormal{lin}(\mathfrak{g})^\perp) \oplus \textnormal{lin}(\mathfrak{f}')^\perp$. We can write $\xi = \xi^{\mathfrak{g}} + \xi^{\mathfrak{f}' \cap \mathfrak{g}^\perp} + \xi^{\mathfrak{f}'}$ with respect to this decomposition. Then $\xi^{\mathfrak{f}'} = \xi^{\mathfrak{g}} + \xi^{\mathfrak{f}' \cap \mathfrak{g}^\perp}$. Since $\mathfrak{g}$ is $\xi$-constant, we have $\xi^{\mathfrak{g}} = 0$. Therefore, $\xi^{\mathfrak{f}'} = \xi^{\mathfrak{f}' \cap \mathfrak{g}^\perp}$. 

Similarly, we have $\xi^{\mathfrak{g}^\perp} = \xi^{\mathfrak{f}' \cap \mathfrak{g}^\perp} + \xi^{\mathfrak{f}'} = \xi$. We have the orthogonal decomposition $\textnormal{lin}(\mathfrak{g})^\perp = (\textnormal{lin}(\mathfrak{f}') \cap \textnormal{lin}(\mathfrak{g})^\perp) \oplus \textnormal{lin}(\mathfrak{f}')^\perp$, and we have $(\xi^{\mathfrak{g}^\perp})^{\mathfrak{t}_{\mathfrak{g}}^{\mathfrak{f}'}} = \xi^{\mathfrak{f}' \cap \mathfrak{g}^\perp}$. This is all to say that we naturally have $\xi^{\mathfrak{f}'} = (\xi^{\mathfrak{g}^\perp})^{\mathfrak{t}_{\mathfrak{g}}^{\mathfrak{f}'}}$ with both sides thought of as living inside of $(\textnormal{lin}(\mathfrak{g})^\perp)^*$. 

Lastly, because $P_{\textnormal{lin}(\mathfrak{g})^\perp}$ is an orthogonal projection and thus self-adjoint, we have
\begin{gather*}
    \langle \xi, \eta_{\mathfrak{f}}^{\mathfrak{f}'} \rangle_H = \langle \xi^{\mathfrak{g}^\perp}, \eta_{\mathfrak{f}}^{\mathfrak{f}'} \rangle_H = \langle \xi,  (\eta_{\mathfrak{f}}^{\mathfrak{f}'})^{\mathfrak{g}^\perp} \rangle_H = \langle \xi, \eta_{\mathfrak{t}_{\mathfrak{g}}^{\mathfrak{f}}}^{\mathfrak{t}_{\mathfrak{g}}^{\mathfrak{f}'}} \rangle_H = \langle \xi^{\mathfrak{g}^\perp}, \eta_{\mathfrak{t}_{\mathfrak{g}}^{\mathfrak{f}}}^{\mathfrak{t}_{\mathfrak{g}}^{\mathfrak{f}'}} \rangle_H.
\end{gather*}
\end{proof}

This gives a geometric interpretation to ${}^0\textnormal{DBI}^\mathfrak{p}_\mathfrak{g}(\xi)$ for $\mathfrak{g}$ which are $\xi$-maximal. However, we would like to give an analogous geometric interpretation for the other $\xi$-constant faces. Notice that the expression \eqref{dbi} is $\mathbb{C}$-homogeneous of degree $-n+k$ in $\xi$ which is consistent with being $I(\mathfrak{k}; \xi)$ for some $(n-k)$-dimensional cone $\mathfrak{k}$ based at the origin as in Proposition \ref{barvinok_homogeneous}. We shall indeed show that ${}^0\textnormal{DBI}^\mathfrak{p}_\mathfrak{g}(\xi) = I(\mathfrak{k}; \xi)$ where $\mathfrak{k}$ is an $(n-k)$-dimensional virtual cone in $\textnormal{lin}(\mathfrak{g})^\perp$. 

\subsection{The alternating Levi cone}
The collection of all $\xi$-constant faces of $\mathfrak{p}$ forms a poset with respect to inclusion. The maximal elements in this poset are the $\xi$-maximal faces. Suppose $\mathfrak{g}$ is $\xi$-constant. Let $\mathcal{C}(V)$ denote the space of virtual cones based at the origin in $V$. We define ${}^0\textnormal{LC}^\mathfrak{p}_\mathfrak{g}(\xi) \in \mathcal{C}(\textnormal{lin}(\mathfrak{g})^\perp)$ recursively as follows. If $\mathfrak{g}$ is $\xi$-maximal, then 
\begin{gather}
    {}^0\textnormal{LC}^\mathfrak{p}_\mathfrak{g}(\xi) := [{}^0\mathfrak{t}_{\mathfrak{g}}^{\mathfrak{p}}]. \label{eqn_xi_maximal_case}
\end{gather}
Otherwise we define
\begin{gather}
    {}^0\textnormal{LC}^\mathfrak{p}_\mathfrak{g}(\xi) := [{}^0\mathfrak{t}_{\mathfrak{g}}^{\mathfrak{p}}] - \sum_{\mathfrak{f} \in \{\mathfrak{p}\}_\xi \textnormal{ with } \mathfrak{g} \subset \mathfrak{f}} [{}^0\mathfrak{t}_{\mathfrak{g}}^{\mathfrak{f}}] \times {}^0\textnormal{LC}^\mathfrak{p}_\mathfrak{f}(\xi). \label{virtual_cone}
\end{gather}
Notice that if $\mathfrak{f} \subset \mathfrak{g}$, we have the orthogonal decomposition $\textnormal{lin}(\mathfrak{g})^\perp = (\textnormal{lin}(\mathfrak{f}) \cap \textnormal{lin}(\mathfrak{g})^\perp) \oplus \textnormal{lin}(\mathfrak{f})^\perp$. Thus inductively the expression in \eqref{virtual_cone} is naturally in $\mathcal{C}(\textnormal{lin}(\mathfrak{g})^\perp)$. We shall repeatedly use the following simple observation:
\begin{proposition} \label{prop_triv_or_one}
    Suppose $\xi = 0$. Then ${}^0 \textnormal{LC}_{\mathfrak{p}}^{\mathfrak{p}}(0) = \{0\}$ which is a cone inside a zero-dimensional space, and $I({}^0\textnormal{LC}_{\mathfrak{p}}^{\mathfrak{p}}(0); 0) = 1$. If $\mathfrak{f} \subset \mathfrak{p}$ then ${}^0 \textnormal{LC}_{\mathfrak{f}}^{\mathfrak{p}}(0) = \emptyset$ and $I({}^0 \textnormal{LC}_{\mathfrak{f}}^{\mathfrak{p}}(0); 0) = 0$. 
\end{proposition}

We call ${}^0 \textnormal{LC}_{\mathfrak{g}}^{\mathfrak{p}}(\xi)$ the \textit{shifted alternating Levi cone}. When we unwind the recursive definition, we get the following non-recursive definition:
\begin{gather}
    {}^0\textnormal{LC}^\mathfrak{p}_\mathfrak{g}(\xi) = \sum_{\mathfrak{g} = \mathfrak{h}_0 \subset \dots \subset \mathfrak{h}_{\ell} \in \textnormal{mFl}^\mathfrak{p}_\mathfrak{g}(\xi)} (-1)^\ell [{}^0\mathfrak{t}_{\mathfrak{g}}^{\mathfrak{h}_1} \times {}^0\mathfrak{t}_{\mathfrak{h}_1}^{\mathfrak{h}_2} \times \dots \times {}^0\mathfrak{t}_{\mathfrak{h}_{\ell-1}}^{\mathfrak{h}_\ell} \times {}^0\mathfrak{t}_{\mathfrak{h}_\ell}^{\mathfrak{p}}]. \label{levi_cone2}
\end{gather}
Notice that we naturally have that if $V = W_1 \oplus W_2$ and if $\mathfrak{p}_1 \subset W_1$ and $\mathfrak{p}_2 \subset W_2$ are polyhedra, then $[\mathfrak{p}_1 \times \mathfrak{p}_2] = [\mathfrak{p}_1] \times [\mathfrak{p}_2]$. Also notice that given a element in $\mathfrak{g} = \mathfrak{h}_0 \subset \dots \subset \mathfrak{h}_\ell \in \textnormal{mFl}^\mathfrak{p}_\mathfrak{g}(\xi)$, we get a natural orthogonal decomposition
\begin{gather*}
    \textnormal{lin}(\mathfrak{g})^\perp = (\textnormal{lin}(\mathfrak{g})^\perp \cap \textnormal{lin}(\mathfrak{h}_1)) \oplus (\textnormal{lin}(\mathfrak{h}_1)^\perp \cap \textnormal{lin}(\mathfrak{h}_2)) \oplus \dots \oplus (\textnormal{lin}(\mathfrak{h}_{\ell-1})^\perp \cap \textnormal{lin}(\mathfrak{h}_\ell)) \oplus \textnormal{lin}(\mathfrak{h}_\ell)^\perp.
\end{gather*}
We clearly have that ${}^0 \mathfrak{t}_{\mathfrak{h}_j}^{\mathfrak{h}_{j+1}} \subset (\textnormal{lin}(\mathfrak{h}_j)^\perp \cap \textnormal{lin}(\mathfrak{h}_{j+1}))$.

Let
\begin{gather*}
    \textnormal{LC}^\mathfrak{p}_\mathfrak{g}(\xi) := \mathfrak{g}^{\mathfrak{g}^\perp} + {}^0 \textnormal{LC}^\mathfrak{p}_\mathfrak{g}(\xi).
\end{gather*}
In words, we have shifted ${}^0\textnormal{LC}^\mathfrak{p}_\mathfrak{g}(\xi)$ to be based at $\mathfrak{g}^{\mathfrak{g}^\perp}$ which is the orthogonal projection of $\mathfrak{g}$ onto the space $\textnormal{lin}(\mathfrak{g})^\perp$. We call this the \textit{alternating Levi cone}.

\begin{example}
As an example, suppose $\mathfrak{g}$ is of codimension 2 in the $\xi$-maximal face $\mathfrak{f}$, and suppose $\mathfrak{f}$ is the only $\xi$-maximal face containing $\mathfrak{g}$. Let $\mathfrak{h}_1$ and $\mathfrak{h}_2$ be the intermediate faces between $\mathfrak{g}$ and $\mathfrak{f}$ (there are necessarily exactly two such faces). We have
\begin{align*}
\textnormal{LC}_{\mathfrak{g}}^{\mathfrak{p}}(\xi) = [\mathfrak{t}_{\mathfrak{g}}^{\mathfrak{p}}] - [\mathfrak{t}_{\mathfrak{g}}^{\mathfrak{h}_1} \times \mathfrak{t}_{\mathfrak{h}_1}^{\mathfrak{p}}] + [\mathfrak{t}_{\mathfrak{g}}^{\mathfrak{h}_1} \times \mathfrak{t}_{\mathfrak{h}_1}^{\mathfrak{f}} \times \mathfrak{t}_{\mathfrak{f}}^{\mathfrak{p}}] -[\mathfrak{t}_{\mathfrak{g}}^{\mathfrak{h}_2} \times \mathfrak{t}_{\mathfrak{h}_2}^{\mathfrak{p}}] + [\mathfrak{t}_{\mathfrak{g}}^{\mathfrak{h}_2} \times \mathfrak{t}_{\mathfrak{h}_2}^{\mathfrak{f}} \times \mathfrak{t}_{\mathfrak{f}}^{\mathfrak{p}}] - [\mathfrak{t}_{\mathfrak{g}}^{\mathfrak{f}} \times \mathfrak{t}_{\mathfrak{f}}^{\mathfrak{p}}].
\end{align*}
\end{example}

We shall see that the alternating Levi cone provides a geometric interpretation for the quantity ${}^0 \textnormal{DBI}_{\mathfrak{g}}^{\mathfrak{p}}(\xi)$. Before discussing that we shall first discuss what we believe to be a new polytope decomposition involving the alternating Levi cones.

\subsection{Polyhedral decomposition of a polytope with respect to $\xi$} \label{subsec_polyhedral_decomp}

\begin{proof}[Proof of Theorem \ref{thm_polytope_decomp_intro}]
    Theorem \ref{brianchon_gram} tells us that
    \begin{gather}
        \mathfrak{g}^{\mathfrak{g}} \times \textnormal{LC}_{\mathfrak{g}}^{\mathfrak{p}}(\xi) \equiv \sum_{\mathfrak{v} \in \textnormal{Vert}(\mathfrak{g})} \mathfrak{t}_{\mathfrak{v}^{\mathfrak{g}}}^{\mathfrak{g}^{\mathfrak{g}}} \times \textnormal{LC}_{\mathfrak{g}}^{\mathfrak{p}}(\xi) = \sum_{\mathfrak{v} \in \textnormal{Vert}(\mathfrak{g})} \mathfrak{v} + {}^0 \mathfrak{t}_{\mathfrak{v}}^{\mathfrak{g}} \times {}^0 \textnormal{LC}_{\mathfrak{g}}^{\mathfrak{p}}(\xi). \label{vertex_expansion}
    \end{gather}
    Let $\textnormal{Face}_{\mathfrak{v}}^{\mathfrak{p}}(\xi)$ denote the set of $\xi$-constant faces of $\mathfrak{p}$ which contain $\mathfrak{v}$. Suppose we expand out all of the terms in \eqref{polytope_decomp} as in \eqref{vertex_expansion} and group them according to the vertices of $\mathfrak{p}$. Then for a fixed vertex $\mathfrak{v}$ we get as contribution the following virtual cone (shifted to be based at $\mathfrak{v}$):
    \begin{align}
        \sum_{\mathfrak{g} \in \textnormal{Face}_{\mathfrak{v}}^{\mathfrak{p}}(\xi)} & \sum_{(\mathfrak{g} = \mathfrak{h}_0 \subset \dots \subset \mathfrak{h}_\ell) \in \textnormal{mFl}_{\mathfrak{g}}^{\mathfrak{p}}(\xi)} (-1)^\ell \ \ {}^0\mathfrak{t}_{\mathfrak{v}}^{\mathfrak{g}} \times {}^0\mathfrak{t}_{\mathfrak{g}}^{\mathfrak{h}_1} \times \dots \times {}^0\mathfrak{t}_{\mathfrak{h}_\ell}^{\mathfrak{p}} \label{cancel1} \\
        &= \sum_{(\mathfrak{v} = \mathfrak{h}_0 \subset \dots \subset \mathfrak{h}_\ell) \in \textnormal{mFl}_{\mathfrak{v}}^{\mathfrak{p}}(\xi)} (-1)^\ell \ \ {}^0\mathfrak{t}_{\mathfrak{v}}^{\mathfrak{h}_1} \times {}^0\mathfrak{t}_{\mathfrak{h}_1}^{\mathfrak{h}_2} \times \dots {}^0\mathfrak{t}_{\mathfrak{h}_\ell}^{\mathfrak{p}} \label{cancel2} \\
        &+ \sum_{\mathfrak{g} \in \textnormal{Face}_{\mathfrak{v}}^{\mathfrak{p}}(\xi) \setminus \{\mathfrak{v}\}} \sum_{(\mathfrak{g} = \mathfrak{h}_1 \subset \dots \mathfrak{h}_\ell) \in \textnormal{mFl}_{\mathfrak{g}}^{\mathfrak{p}}(\xi)} (-1)^{\ell - 1} \ \ {}^0\mathfrak{t}_{\mathfrak{v}}^{\mathfrak{h}_1} \times {}^0\mathfrak{t}_{\mathfrak{h}_1}^{\mathfrak{h}_2} \times \dots {}^0\mathfrak{t}_{\mathfrak{h}_\ell}^{\mathfrak{p}} \label{cancel3} \\
        &= {}^0\mathfrak{t}_{\mathfrak{v}}^{\mathfrak{p}}. \label{cancel4}
    \end{align}
To go from \eqref{cancel1} to \eqref{cancel2} and \eqref{cancel3}, we have broken up the outer sum in \eqref{cancel1} into terms corresponding to $\mathfrak{v}$, and the terms coming from the remaining faces in $\textnormal{Face}_{\mathfrak{v}}^{\mathfrak{p}}(\xi)$. Every term in \eqref{cancel2} with $\ell \geq 1$ can be grouped with a unique term in \eqref{cancel3} which cancels it out. We are therefore left with a single term corresponding to the unique element in $\textnormal{mFl}_{\mathfrak{v}}^{\mathfrak{p}}(\xi)$ with $\ell = 0$, i.e. the flag $\mathfrak{v}$ which gives the cone ${}^0 \mathfrak{t}_{\mathfrak{v}}^{\mathfrak{p}}$. 

Finally, since $\mathfrak{t}_{\mathfrak{v}}^{\mathfrak{p}} = \mathfrak{v} + {}^0 \mathfrak{t}_{\mathfrak{v}}^{\mathfrak{p}}$, we obtain the theorem by again appealing to Theorem \ref{brianchon_gram} which tells us that $\mathfrak{p}$ is equal to the sum of the $\mathfrak{t}_{\mathfrak{v}}^{\mathfrak{p}}$ modulo indicator functions of polyhedra with lines.
\end{proof}

\subsection{An example}

In Section \ref{sec_example_intro} we gave an example of the decomposition in Theorem \ref{thm_polytope_decomp_intro}. In that example the $\xi$-maximal edge $\mathfrak{e}$ had the property that the tangent cone at both vertices of $\mathfrak{e}$ was acute. This resulted in the corresponding alternating Levi cones of the vertices being ``negative''. We now also consider a two-dimensional example with a $\xi$-maximal edge $\mathfrak{e}$ for which the cones at each vertex are obtuse, and thus the alternating Levi cones are ``positive''. 

Let $\mathfrak{p}$ be the polygon with vertices at $(0, 0), (0, 2), (1, 3), (3, 1), (2, -1)$. Let $\mathfrak{e}$ denote the edge from $(0, 0)$ to $(0, 2)$. Let $\xi = (1, 0)$. We then have that $\{\mathfrak{p} \}_\xi = \{\mathfrak{e}, (0, 0), (0, 2), (1, 3),$ $(3, 1), (2, -1)\}$. For the vertices that are $\xi$-maximal, i.e. $(1, 3), (3, 1), (2, -1)$, we have $\textnormal{LC}^\mathfrak{p}_\mathfrak{v}(\xi) = \mathfrak{t}_{\mathfrak{v}}^\mathfrak{p}$ like in the usual Brion's formula. We also have
\begin{align*}
    \textnormal{LC}^\mathfrak{p}_ \mathfrak{e}(\xi) &= \mathfrak{t}_{\mathfrak{e}}^{\mathfrak{p}} = \mathds{1}\{t \cdot (1, 0): t \in [0, \infty)\} \\
    \textnormal{LC}^\mathfrak{p}_{(0, 0)}(\xi) &= \mathfrak{t}_{(0, 0)}^{\mathfrak{p}} - \mathfrak{t}_{(0, 0)}^{\mathfrak{e}} \times \mathfrak{t}_{\mathfrak{e}}^{\mathfrak{p}} = \mathds{1}\{t_1 \cdot (2, -1) + t_2 \cdot (1, 0): t_1 \in [0, \infty), t_2 \in (0, \infty)\} \\
    \textnormal{LC}^\mathfrak{p}_{(0, 2)}(\xi) &= \mathfrak{t}_{(0, 2)}^{\mathfrak{p}} - \mathfrak{t}_{(0, 2)}^{\mathfrak{e}} \times \mathfrak{t}_{\mathfrak{e}}^{\mathfrak{p}} = \mathds{1}\{(0, 2) + t_1 \cdot (1, 1) + t_2 \cdot (1, 0): t_1 \in [0, \infty), t_2 \in (0, \infty)\}.
\end{align*}
The decomposition in Theorem \ref{thm_polytope_decomp_intro} for this polygon is illustrated in Figure \ref{fig_poly1}.

\begin{figure} 
\begin{tikzpicture}[scale=0.75]
\fill[color = blue, opacity = 0.2] (0, 0) -- (0, 2) -- (1, 3) -- (3, 1) -- (2, -1) -- cycle;
\draw[color = blue, line width = 1] (0, 0) -- (0, 2) -- (1, 3) -- (3, 1) -- (2, -1) -- cycle;
\node[xshift = -5mm] at (0, 0) {$(0, 0)$};
\node[xshift = -5mm] at (0, 2) {$(0, 2)$};
\node[yshift = 3mm] at (1, 3) {$(1, 3)$};
\node[xshift = 5mm] at (3, 1) {$(3, 1)$};
\node[yshift = -3mm] at (2, -1) {$(2, -1)$};
\node[xshift = -2mm] at (0, 1) {$\mathfrak{e}$};

\node at (5, 1) {$\equiv$};
\begin{scope}[xshift = 6.5 cm]
    \fill[color = red, opacity = 0.2] (0, 0) -- (0, 2) -- (3, 2) -- (3, 0) -- cycle;
    \node[xshift = -5mm] at (0, 0) {$(0, 0)$};
    \node[xshift = -5mm] at (0, 2) {$(0, 2)$};
    \node[xshift = -2mm] at (0, 1) {$\mathfrak{e}$};
    \draw[color = red, line width = 1] (3, 2) -- (0, 2) -- (0, 0) -- (3, 0);
\end{scope}

\node at (10, 1) {$+$};

\begin{scope}[xshift = 11.5 cm]
    \fill[color = red, opacity = 0.2] (0, 0) -- (3, 0) -- (3, -1.5) -- cycle;
    \node[xshift = -5mm] at (0, 0) {$(0, 0)$};
    \draw[color = red, line width = 1] (0, 0) -- (3, -1.5);
    \draw[color = red] (0, 0) circle (2 pt);
    \draw[color = red, line width = 1, dashed] (0, 0) -- (3, 0);
\end{scope}

\node at (15, 1) {$+$};

\begin{scope}[xshift = 16 cm]
    \fill[color = red, opacity = 0.2] (0, 2) -- (3, 5) -- (3, 2) -- cycle;
    \node[xshift = -5mm] at (0, 2) {$(0, 2)$};
    \draw[color = red, line width = 1] (0, 2) -- (3, 5);
    \draw[color = red] (0, 2) circle (2 pt);
    \draw[color = red, line width = 1, dashed] (0, 2) -- (3, 2);
\end{scope}

\node at (5, -4) {$+$};

\begin{scope}[xshift = 6.5 cm, yshift = -5 cm]
    \fill[color = red, opacity = 0.2] (1, 3) -- (3, 1) -- (-1, 1) -- cycle;
    \node[yshift = 3mm] at (1, 3) {$(1, 3)$};
    \draw[color = red, line width = 1] (1, 3) -- (3, 1);
    \draw[color = red, line width = 1] (1, 3) -- (-1, 1);
\end{scope}

\node at (10, -4) {$+$};

\begin{scope}[yshift = -5cm, xshift = 9 cm]
    \fill[color = red, opacity = 0.2] (3, 1) -- (2, -1) -- (2, 2) -- cycle;
    \node[xshift = 5mm] at (3, 1) {$(3, 1)$};
    \draw[color = red, line width = 1] (3, 1) -- (2, -1);
    \draw[color = red, line width = 1] (3, 1) -- (2, 2);
\end{scope}

\node at (14, -4) {$+$};

\begin{scope}[yshift = -5cm, xshift = 16.3cm]
    \fill[color = red, opacity = 0.2] (2, -1) -- (3, 1) -- (-2, 1) -- cycle;
    \node[yshift = -3mm] at (2, -1) {$(2, -1)$};
    \draw[color = red, line width = 1] (2, -1) -- (3, 1);
    \draw[color = red, line width = 1] (2, -1) -- (-2, 1);
\end{scope}

\end{tikzpicture}
\caption{Illustration of the polyhedral decomposition of Theorem \ref{thm_polytope_decomp_intro} for $\mathfrak{p}$.}
\label{fig_poly1}
\end{figure}

\begin{remark}
We see that in general in two dimensions, $\textnormal{LC}^\mathfrak{p}_\mathfrak{v}(\xi)$ is a cone with one of its rays missing and possibly weighted by -1. Thus this virtual cone is (up to lower dimensional cones and negative signs) an actual cone. More generally, if a $\xi$-constant face $\mathfrak{f}$ has codimension one, then either it is $\xi$-maximal, in which $\textnormal{LC}_{\mathfrak{f}}^{\mathfrak{p}}(\xi)$ is an actual cone by \eqref{eqn_xi_maximal_case}, or else $\mathfrak{p}$ is itself $\xi$-maximal, in which case $\textnormal{LC}_{\mathfrak{f}}^{\mathfrak{p}}(\xi) = \emptyset$ by Proposition \ref{prop_triv_or_one}. 

If $\mathfrak{f}$ is codimension two, then in case it is $\xi$-maximal $\textnormal{LC}_{\mathfrak{f}}^{\mathfrak{p}}(\xi)$ is clearly an actual cone, and in case $\mathfrak{p}$ is $\xi$-constant, then $\textnormal{LC}_{\mathfrak{f}}^{\mathfrak{p}}(\xi) = \emptyset$. Otherwise, it is necessarily contained in a unique $\xi$-maximal face which is codimension one, call it $\mathfrak{h}$. To see this, note that if $\xi$ were constant on two different facets of $\mathfrak{p}$ which intersected non-trivially, then necessarily $\xi = 0$. Thus the computation of $\textnormal{LC}_{\mathfrak{f}}^{\mathfrak{p}}(\xi)$ (which lives in the two-dimensional space orthogonal to $\mathfrak{f}$) reduces to a situation like in the above example or the example in Section \ref{sec_example_intro}. Thus we see that for codimension at most two, the alternating Levi cone is (again, up to lower dimensional cones and negative signs) an actual cone. It is worth investigating if this continues to be true for higher codimensions.
\end{remark}

\subsection{Geometric interpretation of ${}^0 \textnormal{DBI}_{\mathfrak{g}}^{\mathfrak{p}}(\xi)$}
We can group all elements in $\textnormal{sFl}_{\mathfrak{g}}^{\mathfrak{p}}$ according to the first $\xi$-constant face that occurs in the chain. We then naturally get
\begin{align}
    \textnormal{sFl}_{\mathfrak{g}}^{\mathfrak{p}} &= \bigsqcup_{\mathfrak{f} \in \{\mathfrak{p}\}_\xi \textnormal{ with } \mathfrak{g} \subseteq \mathfrak{f}} \textnormal{sFl}_{\mathfrak{g}}^{\mathfrak{f}} \cdot \textnormal{sFl}_{\mathfrak{f}}^{\mathfrak{p}}(\xi) \notag \\
    &= \textnormal{sFl}_{\mathfrak{g}}^{\mathfrak{p}}(\xi) \sqcup \bigsqcup_{\mathfrak{f} \in \{\mathfrak{p}\}_\xi \textnormal{ with } \mathfrak{g} \subset \mathfrak{f}} \textnormal{sFl}_{\mathfrak{g}}^{\mathfrak{f}} \cdot \textnormal{sFl}_{\mathfrak{f}}^{\mathfrak{p}}(\xi), \label{recursive_flags}
\end{align}
where the product ``$\cdot$'' means we take all the ways of concatenating a flag in $\textnormal{sFl}_{\mathfrak{g}}^{\mathfrak{f}}$ with a flag in $\textnormal{sFl}_{\mathfrak{f}}^{\mathfrak{p}}(\xi)$.  In comparing \eqref{virtual_cone} with \eqref{recursive_flags}, the definition of ${}^0\textnormal{LC}^\mathfrak{p}_\mathfrak{g}(\xi)$ naturally suggests the following:
\begin{theorem} \label{thm_holomorphic}
    Suppose $\mathfrak{g}$ is $\xi$-constant. Then
    \begin{gather*}
        {}^0\textnormal{DBI}^\mathfrak{p}_\mathfrak{g}(\xi) = I({}^0\textnormal{LC}^\mathfrak{p}_\mathfrak{g}(\xi); \xi).
    \end{gather*}
    By this we mean that the function $I({}^0\textnormal{LC}^\mathfrak{p}_\mathfrak{g}(\xi); \tau)$ with $\tau \in (\textnormal{lin}(\mathfrak{g})^\perp)^*_{\mathbb{C}}$ is holomorphic at $\tau = \xi$ with value equal to ${}^0\textnormal{DBI}^\mathfrak{p}_\mathfrak{g}(\xi)$. 
\end{theorem}

We point out that if $\mathfrak{g}$ is not $\xi$-maximal, then for every single constituent cone $\mathfrak{k}$ in the definition of ${}^0 \textnormal{LC}^\mathfrak{p}_\mathfrak{g}(\xi)$ in \eqref{levi_cone2}, we have that $I(\mathfrak{k}; \tau)$ is \textit{not} holomorphic at $\tau = \xi$. We also point out that if $\mathfrak{g}$ is $\xi$-maximal, then Theorem \ref{thm_holomorphic} simply reduces to Proposition \ref{prop_maximal}.

\begin{example} \label{ex_calc}
    Suppose $\mathfrak{g}$ is of codimension 1 in the $\xi$-maximal face $\mathfrak{f}$, and suppose $\mathfrak{f}$ is the only $\xi$-maximal face containing $\mathfrak{g}$. Then ${}^0 \textnormal{LC}_{\mathfrak{g}}^{\mathfrak{p}}(\xi) = [{}^0 \mathfrak{t}_{\mathfrak{g}}^{\mathfrak{p}}] - [{}^0\mathfrak{t}_{\mathfrak{g}}^{\mathfrak{f}} \times {}^0 \mathfrak{t}_{\mathfrak{f}}^{\mathfrak{p}}]$. We wish to show that Theorem \ref{thm_holomorphic} holds in this case. Let $\tau \in (\textnormal{lin}(\mathfrak{g})^\perp)^*_{\mathbb{C}}$ be generic with respect to ${}^0 \textnormal{LC}_{\mathfrak{g}}^{\mathfrak{p}}(\xi)$. We have that
    \begin{align}
        I({}^0 \textnormal{LC}_{\mathfrak{g}}^{\mathfrak{p}}(\xi); \tau) &= \sum_{(\mathfrak{g} = \mathfrak{h}_0 \subset \mathfrak{h}_1 \neq \mathfrak{f} \subset \dots \subset \mathfrak{h}_\ell) \in \textnormal{sFl}_{\mathfrak{g}}^{\mathfrak{p}}} \prod_{j =1}^\ell \frac{\langle \tau, \eta_{\mathfrak{h}_{j-1} }^{\mathfrak{h}_j} \rangle}{\|\tau^{\mathfrak{h}_j}\|^2} + \sum_{(\mathfrak{g} = \mathfrak{h}_0 \subset \mathfrak{f} = \mathfrak{h}_1 \subset \dots \subset \mathfrak{h}_\ell) \in \textnormal{sFl}_{\mathfrak{g}}^{\mathfrak{p}}} \frac{\langle \tau, \eta_{\mathfrak{g}}^{\mathfrak{f}} \rangle}{\|\tau^\mathfrak{f}\|^2} \prod_{j = 2}^\ell \frac{\langle \tau, \eta_{\mathfrak{h}_{j-1} }^{\mathfrak{h}_j} \rangle}{\|\tau^{\mathfrak{h}_j}\|^2} \nonumber \\
        & - \frac{\langle \tau, \eta_{\mathfrak{g}}^{\mathfrak{f}} \rangle}{\|\tau^{\mathfrak{f}}\|^2} \sum_{(\mathfrak{f} = \mathfrak{h}_1 \subset \dots \subset \mathfrak{h}_\ell) \in \textnormal{sFl}_{\mathfrak{f}}^{\mathfrak{p}}} \prod_{j = 2}^\ell \frac{\langle \tau, \eta_{\mathfrak{h}_{j-1}}^{\mathfrak{h}_j} \rangle}{\|\tau^{\mathfrak{t}_{\mathfrak{f}}^{\mathfrak{h}_j}}\|^2}. \label{eqn_nicest}
    \end{align}

    Note that we have $\|\tau^{\mathfrak{h}_j}\|^2 = \|\tau^{\mathfrak{f}}\|^2 + \|\tau^{\mathfrak{t}_{\mathfrak{f}}^{\mathfrak{h}_j}}\|^2$. We therefore have
    \begin{align}
        & \sum_{(\mathfrak{g} = \mathfrak{h}_0 \subset \mathfrak{f} = \mathfrak{h}_1 \subset \dots \subset \mathfrak{h}_\ell) \in \textnormal{sFl}_{\mathfrak{g}}^{\mathfrak{p}}} \frac{\langle \tau, \eta_{\mathfrak{g}}^{\mathfrak{f}} \rangle}{\|\tau^\mathfrak{f}\|^2} \prod_{j = 2}^\ell \frac{\langle \tau, \eta_{\mathfrak{h}_{j-1} }^{\mathfrak{h}_j} \rangle}{\|\tau^\mathfrak{h}_j\|^2} - \frac{\langle \tau, \eta_{\mathfrak{g}}^{\mathfrak{f}} \rangle}{\|\tau^{\mathfrak{f}}\|^2} \sum_{(\mathfrak{f} = \mathfrak{h}_1 \subset \dots \subset \mathfrak{h}_\ell) \in \textnormal{sFl}_{\mathfrak{f}}^{\mathfrak{p}}} \prod_{j = 2}^\ell \frac{\langle \tau, \eta_{\mathfrak{h}_{j-1}}^{\mathfrak{h}_j} \rangle}{\|\tau^{\mathfrak{t}_{\mathfrak{f}}^{\mathfrak{h}_j}}\|^2} \nonumber \\
        &= \frac{\langle \tau, \eta_{\mathfrak{g}}^{\mathfrak{f}} \rangle}{\|\tau^{\mathfrak{f}}\|^2} \Big( \sum_{(\mathfrak{f} = \mathfrak{h}_1 \subset \dots \subset \mathfrak{h}_\ell \in \textnormal{sFl}_{\mathfrak{f}}^{\mathfrak{g}})} \prod_{j = 2}^\ell \langle \tau, \eta_{\mathfrak{h}_{j-1}}^{\mathfrak{h}_j} \rangle \big( \prod_{j = 2}^{\ell} \frac{1}{\|\tau^{\mathfrak{f}}\|^2 + \|\tau^{\mathfrak{t}_{\mathfrak{f}}^{\mathfrak{h}_j}}\|^2} - \prod_{j = 2}^\ell \frac{1}{\|\tau^{\mathfrak{t}_{\mathfrak{f}}^{\mathfrak{h}_j}}\|^2} \big) \Big). \label{eqn_nicer}
    \end{align}
    We also have that
    \begin{align}
        \frac{1}{\|\tau^{\mathfrak{f}}\|^2}\Big( \prod_{j = 2}^{\ell} \frac{1}{\|\tau^{\mathfrak{f}}\|^2 + \|\tau^{\mathfrak{t}_{\mathfrak{f}}^{\mathfrak{h}_j}}\|^2} - \prod_{j = 2}^\ell \frac{1}{\|\tau^{\mathfrak{t}_{\mathfrak{f}}^{\mathfrak{h}_j}}\|^2} \Big) &= \frac{1}{\|\tau^{\mathfrak{f}}\|^2} \Big(\frac{ \prod_{j = 2}^\ell \|\tau^{\mathfrak{t}_{\mathfrak{f}}^{\mathfrak{h}_j}}\|^2 - \prod_{j = 2}^\ell (\|\tau^{\mathfrak{f}}\|^2 + \|\tau^{\mathfrak{t}_{\mathfrak{f}}^{\mathfrak{h}_j}}\|^2)}{\prod_{j = 2}^\ell (\|\tau^{\mathfrak{t}_{\mathfrak{f}}^{\mathfrak{h}_j}}\|^2)\cdot (\|\tau^{\mathfrak{f}}\|^2 + \|\tau^{\mathfrak{t}_{\mathfrak{f}}^{\mathfrak{h}_j}}\|^2)} \Big) \nonumber \\
        &= \frac{1}{\|\tau^{\mathfrak{f}}\|^2} \Big( \frac{\|\tau^\mathfrak{f}\|^2 \cdot R}{\prod_{j = 2}^\ell (\|\tau^{\mathfrak{t}_{\mathfrak{f}}^{\mathfrak{h}_j}}\|^2)\cdot (\|\tau^{\mathfrak{f}}\|^2 + \|\tau^{\mathfrak{t}_{\mathfrak{f}}^{\mathfrak{h}_j}}\|^2)} \Big) \label{eqn_messy}
    \end{align}
    where $R$ is some polynomial in $\|\tau^{\mathfrak{f}}\|^2$ and $\|\tau^{\mathfrak{t}_{\mathfrak{h}_{j-1}}^{\mathfrak{h}_{j}}}\|^2$. Thus the expression in \eqref{eqn_messy} is well-defined even when $\tau^{\mathfrak{f}} = 0$ and no other term in the denominator vanishes, for example when $\tau = \xi$. This in turn implies that \eqref{eqn_nicer} is holomorphic in a neighborhood of $\xi$ and is zero at $\tau = \xi$. This then finally implies via \eqref{eqn_nicest} that $I({}^0 \textnormal{LC}_{\mathfrak{g}}^{\mathfrak{p}}(\xi); \xi) = {}^0 \textnormal{DBI}_\mathfrak{g}^{\mathfrak{p}}(\xi)$.
\end{example}

\subsection{An algebraic interlude}
In preparation for proving Theorem \ref{thm_holomorphic}, we make an algebraic interlude which captures and generalizes the essential algebraic ingredients used in Example \ref{ex_calc}. Let $a = (a_1, \dots, a_k)$ and $x = (x_1, \dots, x_m)$ be variables with $m \geq 1$. Let $\varepsilon = (\varepsilon_1, \varepsilon_2, \dots, \varepsilon_k)$ be a binary sequence. Let 
\begin{align*}
    \delta^{\varepsilon} &:= (\delta_1^\varepsilon, \dots, \delta_k^\varepsilon, \delta_{k+1}^\varepsilon, \dots, \delta_{k+m}^\varepsilon) = (1, \varepsilon_1, \dots, \varepsilon_{k-1}, \varepsilon_k, 0, \dots, 0) \\
    b &:= (b_1, \dots, b_{k+m}) = (a_1, \dots, a_k, x_1, \dots x_m).
\end{align*}
In words, $\delta^\varepsilon$ is the binary sequence $\varepsilon$ with a 1 prepended, and 0's appended so that the total length is $m$, and $b$ is simply the concatenation of $a$ and $x$. We define
\begin{gather*}
    L(\delta^\varepsilon, j) := \sup \{ i \leq j : \delta_i^\varepsilon = 1\}.
\end{gather*}
In words, $L(\delta^\varepsilon, j)$ is the biggest index less than or equal to $j$ such that $\delta^\varepsilon$ has a 1. We now define
\begin{align}
    M(\varepsilon; b; j) &:= \sum_{\ell = L(\delta^\varepsilon, j)}^{j} b_\ell \notag \\
    P(\varepsilon; b) &:= \prod_{j = 1}^{k+m} \frac{1}{M(\varepsilon; b; j)}. \label{p_fn}
\end{align}
Notice that $M(\varepsilon; b; j)$ is only a function of the $a_i$'s if $1 \leq j \leq k$.

\begin{example} 
These definitions are best digested with examples. Suppose $a = (a_1, \dots, a_4)$, $x = (x_1, x_2)$, $\varepsilon = (0, 0, 1, 0)$, and $\delta^\varepsilon = (1, 0, 0, 1, 0, 0)$. Then
\begin{align}
    M(\varepsilon; b; 1) &= a_1 \notag \\
    M(\varepsilon; b; 2) &= a_1 + a_2 \notag \\
    M(\varepsilon; b; 3) &= a_1 + a_2 + a_3 \notag \\
    M(\varepsilon; b; 4) &= a_4 \notag \\
    M(\varepsilon; b; 5) &= a_4 + x_1 \notag \\
    M(\varepsilon; b; 6) &= a_4 + x_1 + x_2 \notag \\
    P(\varepsilon; b) &= \frac{1}{a_1} \cdot \frac{1}{a_1 + a_2} \cdot \frac{1}{a_1 + a_2 + a_3} \cdot \frac{1}{a_4} \cdot \frac{1}{a_4 + x_1} \cdot \frac{1}{a_4 + x_1 + x_2}. \label{algebraic_example}
\end{align}
\end{example}

\begin{lemma} \label{lemma_alg1}
    The function $Q(c, y_1, \dots, y_n)$ defined by
    \begin{gather*}
        Q(c, y_1, \dots, y_n):= \frac{1}{c} \Big( \prod_{\ell = 1}^n \frac{1}{y_\ell + c} - \prod_{\ell = 1}^n \frac{1}{y_\ell} \Big)
    \end{gather*}
    has an expression as a rational function $f(c, y)/g(c, y)$ such that $g(c, y)$ does not contain a factor of $c$.
\end{lemma}
\begin{proof}
    When we bring to a common denominator we get
    \begin{gather}
        \frac{1}{c} \frac{\prod_\ell y_\ell - \prod_\ell (y_\ell + c)}{\prod_\ell \big(y_\ell \cdot (y_\ell + c) \big)}. \label{common_denom}
    \end{gather}
    Notice then that
    \begin{gather*}
        \prod_{\ell = 1}^n (y_\ell + c) = (\prod_{\ell = 1}^n y_\ell) + c \cdot R(c, y_1, \dots, y_n),
    \end{gather*}
    where $R(c, y_1, \dots, y_n)$ is a polynomial. The expression in \eqref{common_denom} thus simplifies to
    \begin{gather*}
        \frac{-R(c, y_1, \dots, y_n)}{\prod y_\ell \cdot (y_\ell + c)}.
    \end{gather*}
\end{proof}
\noindent This is the same calculation that occurred in \eqref{eqn_messy}.

\begin{lemma} \label{lemma_alg2}
    Let $\varepsilon$ and $\varepsilon'$ be two binary sequences that differ only at position $r$. Then $P(\varepsilon; b) - P(\varepsilon'; b)$ has an expression as a rational function $f(x, a)/g(x, a)$ such that $g(x, a)$ does not contain the factor $M(\varepsilon; b; r) = M(\varepsilon'; b; r)$. 
\end{lemma}

\begin{proof}
    Without loss of generality, suppose $\varepsilon$ has a 0 at position $r$, and $\varepsilon'$ has a 1 at position $r$. Let $s$ be the first index in $\delta^\varepsilon$ or in $\delta^{\varepsilon'}$ with $s > r+1$ such that $\delta_s^\varepsilon = 1 = \delta_s^{\varepsilon'}$; if $\delta_s^\varepsilon = 0$ for all $s > r$, then set $s = k+m + 1$. Then for all $1 \leq t \leq r$ we have $M(\varepsilon; b; t) = M(\varepsilon'; b; t)$. Also for all $s \leq t \leq k+m$ we have $M(\varepsilon; b; t) = M(\varepsilon'; b; t)$. On the other hand, for all $r+1 \leq t < s$, we have $M(\varepsilon; b; t) - M(\varepsilon'; b; t) = M(\varepsilon; b; r)$. We can thus write
    \begin{align*}
        P(\varepsilon; b) - P(\varepsilon'; b) &= \prod_{t = 1}^{r-1} \frac{1}{M(\varepsilon; b; t)} \cdot \prod_{t = s}^{k+m} \frac{1}{M(\varepsilon; b; t)} \\
        & \cdot \frac{1}{M(\varepsilon; b; r)} \Big(\prod_{t= r+1}^{s-1} \frac{1}{M(\varepsilon'; b; t) + M(\varepsilon; b; r)} - \prod_{t = r+1}^{s-1} \frac{1}{M(\varepsilon'; b; t)} \Big).
    \end{align*}
The result now follows from Lemma \ref{lemma_alg1}.
\end{proof}

\begin{example} 
Continuing the example in \eqref{algebraic_example} from before, suppose now $\varepsilon' = (1, 0, 1, 0)$, so then $\delta^{\varepsilon'} = (1, 1, 0, 1, 0, 0)$. We then have
\begin{align*}
    M(\varepsilon'; b; 1) &= a_1 \\
    M(\varepsilon'; b; 2) &= a_2 \\
    M(\varepsilon'; b; 3) &= a_2 + a_3 \\
    M(\varepsilon'; b; 4) &= a_4 \\
    M(\varepsilon'; b; 5) &= a_4 + x_1 \\
    M(\varepsilon'; b; 6) &= a_4 + x_1 + x_2.
\end{align*}
In terms of the proof of Lemma \ref{lemma_alg2}, we have that $r = 1$ and $s = 4$.
\end{example}

Given a binary sequence $\varepsilon$, let $\#1(\varepsilon)$ denote the number of ones occuring in $\varepsilon$. Let $\mathcal{B}_k$ denote the set of all binary sequences of length $k$.

\begin{lemma} \label{lemma_alg3}
    Define the function $Q(a, x)$ by
    \begin{gather*}
        Q(a, x) := \sum_{\varepsilon \in \mathcal{B}_k} (-1)^{\# 1(\varepsilon)} P(\varepsilon; b).
    \end{gather*}
    Then $Q(a, x)$ has an expression as a rational function $f(x, a)/g(x, a)$ such that $g(x, a)$ contains no factors of the form $M(\varepsilon'; b; t)$ for any $\varepsilon'$ and for any $1 \leq t \leq k$.
\end{lemma}

\begin{proof}
    This follows almost immediately from Lemma \ref{lemma_alg2} and from the observation that all of the linear polynomials $M(\varepsilon; b; t)$ are irreducible and relatively prime to each other. Since polynomial rings are unique factorization domains, it suffices to show that for every factor $M(\varepsilon; b; t)$ that could show up in the denominator for $Q(a, x)$, we can find an expression for $Q(a, x)$ that does not have that factor in the denominator, but we can clearly do this using Lemma \ref{lemma_alg2}.

    More specifically, suppose we fix a factor $a_j + a_{j+1} + \dots + a_{\ell}$ that could show up as $M(\varepsilon; b; t)$ for some $\varepsilon$ and for some $1 \leq t \leq k$. This factor exactly shows up as $M(\varepsilon; b; \ell)$ for those $\varepsilon$ whose associated $\delta^\varepsilon$ satisfies $\delta_j^\varepsilon = 1$ and $\delta_{j+1}^\varepsilon = \dots = \delta_\ell^\varepsilon = 0$; let $\mathcal{B}_k(j, \ell)$ denote this set of binary sequences $\varepsilon$. Given such a $\varepsilon$, let $\varepsilon^\vee$ be the binary sequence which is identical to $\varepsilon$ except that $\delta^{\varepsilon^\vee}_{\ell+1} = \varepsilon^\vee_{\ell}$ has its bit flipped. We can thus pair up every element in $\mathcal{B}_k(j, \ell)$ with another element such that the binary sequences only differ in position $\ell$. Obviously we have $(-1)^{\#1(\varepsilon)} = (-1) \cdot (-1)^{\#1(\varepsilon^\vee)}$. Let $\varepsilon^1, \dots \varepsilon^{2^{k-(\ell-j)-1}}$ be a complete set of representatives for this pairing into two element subsets $\{\varepsilon, \varepsilon^\vee\}$. We thus get
    \begin{gather*}
        \sum_{\varepsilon \in \mathcal{B}_k} (-1)^{\#1(\varepsilon)} P(\varepsilon; b) = \sum_{\varepsilon \in \mathcal{B}_k \setminus \mathcal{B}_k(j, \ell)} (-1)^{\#1(\varepsilon)} P(\varepsilon; b) + \sum_{i = 1}^{2^{k-(\ell-j)-1}} (-1)^{\#1(\varepsilon)}(P(\varepsilon^i; b) - P((\varepsilon^i)^\vee; b)).
    \end{gather*}
    The first term can be brought to an expression $f(x, a)/g(x, a)$ such that $g(x, a)$ does not contain $a_j + \dots + a_\ell$. By Lemma \ref{lemma_alg2}, every term in the second sum can as well. Therefore, the left hand side can as well.  
\end{proof}

Lemma \ref{lemma_alg3} implies that the function $Q(a, x)$ is holomorphic at any point such that $a_1 = \dots = a_k = 0$ and $x_1, \dots, x_m \neq 0$.

\subsection{Proof of Theorem \ref{thm_holomorphic}}

\begin{proof}[Proof of Theorem \ref{thm_holomorphic}]
Let $\mathfrak{g}$ be $\xi$-constant. Let $\mathfrak{H}$ denote a $\xi$-maladapted flag $\mathfrak{g} = \mathfrak{h}_0 \subset \dots \subset \mathfrak{h}_\ell$. Let $\textnormal{eFl}_{\mathfrak{g}}^{\mathfrak{p}}(\mathfrak{H})$ denote the set of all extensions of this flag to an element in $\textnormal{sFl}_{\mathfrak{g}}^{\mathfrak{p}}$. Let
\begin{gather*}
    {}^0\mathfrak{t}_{\mathfrak{H}} := {}^0\mathfrak{t}_{\mathfrak{g}}^{\mathfrak{h}_1} \times {}^0\mathfrak{t}_{\mathfrak{h}_1}^{\mathfrak{h}_2} \times \dots \times {}^0\mathfrak{t}_{\mathfrak{h}_{\ell-1}}^{\mathfrak{h}_\ell} \times {}^0\mathfrak{t}_{\mathfrak{h}_\ell}^{\mathfrak{p}}.
\end{gather*}
This is a cone based at the origin. 

Let $\tau \in (\textnormal{lin}(\mathfrak{g})^\perp)^*_{\mathbb{C}}$ be generic with respect to ${}^0\mathfrak{t}_{\mathfrak{H}}$ (i.e. not constant on any extremal ray). By Proposition \ref{prop_multiplicative}, we know that if we write $\tau = \tau_0 + \tau_1 + \dots + \tau_{\ell}$ where $\tau_j$ is the restriction of $\tau$ to $\textnormal{lin}({}^0 \mathfrak{t}_{\mathfrak{h}_{j}}^{\mathfrak{h}_{j+1}})$, we have that
\begin{gather*}
    I({}^0\mathfrak{t}_{\mathfrak{H}}; \tau) = I({}^0 \mathfrak{t}_{\mathfrak{g}}^{\mathfrak{h}_1}; \tau_0) \cdots I({}^0 \mathfrak{t}_{\mathfrak{h}_\ell}^{\mathfrak{p}}; \tau_{\ell}).
\end{gather*}
Let $d_j = \textnormal{dim}(\mathfrak{h}_j) - \textnormal{dim}(\mathfrak{g})$ and $d_{\ell+1} = \textnormal{dim}(\mathfrak{p}) - \textnormal{dim}(\mathfrak{g})$. Therefore, using Proposition \ref{barvinok_prop}, we get that
\begin{align}
    & I({}^0\mathfrak{t}_{\mathfrak{H}}; \tau) \notag \\
    &= \Big(\sum_{(\mathfrak{g} = \mathfrak{q}_0 \subset \mathfrak{q}_1 \subset \dots \subset \mathfrak{q}_{d_1} = \mathfrak{h}_1) \in \textnormal{sFl}_{\mathfrak{g}}^{\mathfrak{h}_1}}
    \prod_{j = 1}^{d_1} \frac{\langle \tau, \eta_{\mathfrak{q}_{j-1}}^{\mathfrak{q}_{j}} \rangle_H}{ \|\tau^{\mathfrak{t}_{\mathfrak{g}}^{\mathfrak{q}_{j}}}\|^2} \Big) \dots \Big( \sum_{(\mathfrak{h}_\ell = \mathfrak{q}_{d_\ell} \subset \dots \subset \mathfrak{q}_{d_{\ell+1}} = \mathfrak{p}) \in \textnormal{sFl}_{\mathfrak{h}_\ell}^{\mathfrak{p}}} \prod_{j = d_\ell+1}^{d_{\ell+1}} \frac{\langle \tau, \eta_{\mathfrak{q}_{j-1}}^{\mathfrak{q}_{j}} \rangle_H}{\|\tau^{\mathfrak{t}_{\mathfrak{h}_\ell}^{\mathfrak{q}_j}}\|^2} \Big). \label{prod_flags}
\end{align}
We have used that if $\mathfrak{h}_t \subseteq \mathfrak{q}_1 \subset \mathfrak{q}_2 \subseteq \mathfrak{h}_{t+1}$, then $\langle \tau_t, \eta_{\mathfrak{q}_1}^{\mathfrak{q}_2} \rangle_H = \langle \tau, \eta_{\mathfrak{q}_1}^{\mathfrak{q}_2} \rangle_H$ and $\tau^{\mathfrak{t}_{\mathfrak{h}_t}^{\mathfrak{q}_2}} = (\tau_t)^{\mathfrak{t}_{\mathfrak{h}_t}^{\mathfrak{q}_2}}$. When we expand out \eqref{prod_flags}, the individual terms are parametrized by elements of $\textnormal{eFl}(\mathfrak{H})$. 

We ultimately are interested in the expression
\begin{gather}
    I({}^0 \textnormal{LC}^\mathfrak{p}_\mathfrak{g}(\xi); \tau) = \sum_{\mathfrak{H} = (\mathfrak{g} = \mathfrak{h}_0 \subset \dots \subset \mathfrak{h}_\ell) \in \textnormal{mFl}_{\mathfrak{g}}^{\mathfrak{p}}(\xi)} (-1)^\ell I({}^0\mathfrak{t}_{\mathfrak{H}}; \tau). \label{ult_interest}
\end{gather}
If we replace each summand in \eqref{ult_interest} with the expanded out version of \eqref{prod_flags}, we obtain a double sum where the outer sum is over $\xi$-maladapted flags starting at $\mathfrak{g}$, and an inner sum over refinements of this flag to a saturated flag from $\mathfrak{g}$ to $\mathfrak{p}$. We shall ``change the order of summations'', i.e. have the outer sum be over elements in $\textnormal{sFl}_{\mathfrak{g}}^{\mathfrak{p}}$, and the inner sum over $\xi$-maladapted flags that are refined by this flag.

Let $d = \textnormal{dim}(\mathfrak{p}) - \textnormal{dim}(\mathfrak{g})$. Fix an element $\mathfrak{M} \in \textnormal{sFl}_{\mathfrak{g}}^{\mathfrak{p}}$ with $\mathfrak{M} = (\mathfrak{g} = \mathfrak{m}_0 \subset \mathfrak{m}_1 \subset \dots \subset \mathfrak{m}_d = \mathfrak{p}$. Let $k({\mathfrak{M}})$ be the index such that $\mathfrak{m}_{k({\mathfrak{M}})}$ is $\xi$-constant, but no higher element in the flag is $\xi$-constant. Notice that $k({\mathfrak{M}}) = 0$ if and only if the flag is $\xi$-adapted. Let $m(\mathfrak{M}) = d - k(\mathfrak{M})$. 

We now describe a binary encoding to $\xi$-maladapted flags which are refined by $\mathfrak{M}$. They are parametrized by binary sequences $\varepsilon = (\varepsilon_1, \dots, \varepsilon_{k(\mathfrak{M})})$ of length $k(\mathfrak{M})$. Let $\ell(\varepsilon)$ be the number of ones in $\varepsilon$, and let $i_1, i_2, \dots, i_{\ell(\varepsilon)}$ be the indices of the 1s in $\varepsilon$. We define
\begin{align*}
    D(\mathfrak{g}, \mathfrak{M}, \varepsilon; \tau) :&= \Bigg(\frac{\langle \tau, \eta_{\mathfrak{g}}^{\mathfrak{m}_1} \rangle_H}{\|\tau^{{}^0\mathfrak{t}_{\mathfrak{g}}^{\mathfrak{m}_1}}\|^2} \frac{\langle \tau, \eta_{\mathfrak{m}_1}^{\mathfrak{m}_2} \rangle_H}{\|\tau^{{}^0\mathfrak{t}_{\mathfrak{g}}^{\mathfrak{m}_2}}\|^2} \dots \frac{\langle \tau, \eta_{\mathfrak{m}_{i_1 - 1}}^{\mathfrak{m}_{i_1}} \rangle_H}{\|\tau^{{}^0\mathfrak{t}_{\mathfrak{g}}^{\mathfrak{m}_{i_1}}}\|^2} \Bigg) \\
    &\cdot \Bigg(\frac{\langle \tau, \eta_{\mathfrak{m}_{i_1}}^{\mathfrak{m}_{i_1+1}} \rangle_H}{\|\tau^{{}^0\mathfrak{t}_{\mathfrak{m}_{i_1}}^{\mathfrak{m}_{i_1+1}}}\|^2} \dots \frac{\langle \tau, \eta_{\mathfrak{m}_{i_2 - 1}}^{\mathfrak{m}_{i_2}} \rangle_H}{\|\tau^{{}^0\mathfrak{t}_{\mathfrak{m}_{i_1}}^{\mathfrak{m}_{i_2}}}\|^2} \Bigg) \dots \Bigg(\frac{\langle \tau, \eta_{\mathfrak{m}_{i_\ell}}^{\mathfrak{m}_{i_\ell+1}} \rangle_H}{\|\tau^{{}^0\mathfrak{t}_{\mathfrak{m}_{i_\ell}}^{\mathfrak{m}_{i_\ell+1}}}\|^2} \dots \frac{\langle \tau, \eta_{\mathfrak{m}_{d-1}}^{\mathfrak{p}} \rangle_H}{\|\tau^{{}^0\mathfrak{t}_{\mathfrak{m}_{i_\ell}}^{\mathfrak{p}}}\|^2} \Bigg).
\end{align*}
Let $N(\mathfrak{g}, \mathfrak{M}; \tau)$ be the numerator:
\begin{gather*}
    N(\mathfrak{g}, \mathfrak{M}; \tau) := \prod_{j = 1}^d \langle \tau, \eta_{\mathfrak{m}_{j-1}}^{\mathfrak{m}_j} \rangle_H.
\end{gather*}Notice that in general we have
\begin{gather*} 
    \|\tau_{\mathfrak{t}_{\mathfrak{m}_i}^{\mathfrak{m}_j}}\|^2 = \sum_{r = i}^{j-1} \|\tau_{\mathfrak{t}_{\mathfrak{m}_r}^{\mathfrak{m}_{r+1}}}\|^2,
\end{gather*}
which comes from expressing the norm for $\tau_{\mathfrak{t}_{\mathfrak{m}_i}^{\mathfrak{m}_j}}$ with respect to its constituents in the orthonormal basis $\{\eta_{\mathfrak{m}_i}^{\mathfrak{m}_{i+1}}, \dots, \eta_{\mathfrak{m}_{j-1}}^{\mathfrak{m}_j}\}$ for $\textnormal{lin}(\mathfrak{m}_j)/\textnormal{lin}(\mathfrak{m}_i) \simeq \big(\textnormal{lin}(\mathfrak{m}_i)^\perp \cap \textnormal{lin}(\mathfrak{m}_j) \big)$. Let $a(\mathfrak{M}) = (a_1, \dots, a_{k(\mathfrak{M})})$, and $x(\mathfrak{M}) = (x_1, \dots, x_{m(\mathfrak{M})})$, and $b(\mathfrak{M}) = (a_1, \dots a_{k(\mathfrak{M})}, x_1, \dots, x_{m(\mathfrak{M})})$. Using the notation from \eqref{p_fn}, we have that 
\begin{gather*}
    D(\mathfrak{g}, \mathfrak{M}, \varepsilon; \tau) = N(\mathfrak{g}, \mathfrak{M}; \tau) \cdot P(\varepsilon; b(\mathfrak{M})),
\end{gather*}
where we plug in at
\begin{align*}
    a(\mathfrak{M}) &= (\|\tau^{{}^0\mathfrak{t}_{\mathfrak{m}_0}^{\mathfrak{m}_1}}\|^2, \dots, \|\tau^{{}^0\mathfrak{t}_{\mathfrak{m}_{k-1}}^{\mathfrak{m}_k}}\|^2), \\
    x(\mathfrak{M}) &= (\|\tau^{{}^0\mathfrak{t}_{\mathfrak{m}_{k(\mathfrak{M})}}^{\mathfrak{m}_{k(\mathfrak{M})+1}}}\|^2, \dots, \|\tau^{{}^0\mathfrak{t}_{\mathfrak{m}_{d-1}}^{\mathfrak{m}_{d}}}\|^2)).
\end{align*}
We may thus use Lemma \ref{lemma_alg3} to conclude that
\begin{gather*}
    \sum_{\varepsilon \in \mathcal{B}_{k(\mathfrak{M})}} (-1)^{\#1(\varepsilon)} D(\mathfrak{g}, \mathfrak{M}, \varepsilon; \tau)
\end{gather*}
is analytic at any $\tau$ such that $\|\tau_{\mathfrak{t}_{\mathfrak{m}_0}^{\mathfrak{m}_1}}\|^2 = \dots = \|\tau_{\mathfrak{t}_{\mathfrak{m}_{k-1}}^{\mathfrak{m}_k}}\|^2 = 0$ and $\|\tau_{\mathfrak{t}_{\mathfrak{m}_{j-1}}^{\mathfrak{m}_j}}\|^2 \neq 0$ for all $j > k$. For example at $\tau = \xi$. Furthermore, $N(\mathfrak{g}, \mathfrak{M}; \xi) = 0$ unless $\mathfrak{M}$ is $\xi$-adapted. 

Thus, after switching the order of summation in \eqref{ult_interest}, we get that
\begin{gather*}
    I({}^0 \textnormal{LC}^\mathfrak{p}_\mathfrak{g}(\xi); \tau) = \sum_{\mathfrak{M} \in \textnormal{sFl}_{\mathfrak{g}}^{\mathfrak{p}}} \ \sum_{\varepsilon \in \mathcal{B}_{k(\mathfrak{M})}} (-1)^{\#1(\varepsilon)} D(\mathfrak{g}, \mathfrak{M}, \varepsilon; \tau).
\end{gather*}
This function is analytic at $\tau = \xi$ and its value there is exactly 
\begin{gather*}
    \sum_{\mathfrak{M} \in \textnormal{sFl}^\mathfrak{p}_\mathfrak{g}(\xi)} D(\mathfrak{g}, \mathfrak{M}, (); \xi) = {}^0 \textnormal{DBI}^\mathfrak{p}_\mathfrak{g}(\xi).
\end{gather*}
By the notation ``$()$'', we mean the empty binary sequence.
\end{proof}

\section{Local Euler-Maclaurin formula of Berline-Vergne}
\subsection{The classical Euler-Maclaurin formula}
Classically the Euler-Maclaurin formula relates the integral of a smooth function $f(x)$ on an interval $[a, b]$ with $a, b \in \mathbb{Z}$ to the sum of the values of $f(x)$ over the integer points in $[a, b]$. Let $p \in \mathbb{N}$. Then
\begin{gather}
    \sum_{n = a}^b f(n) - \int_{a}^b f(x) dx = \sum_{k = 1}^p \frac{B_k}{k!} \Big(f^{(k-1)}(b) - f^{(k-1)}(a) \Big) + R_p. \label{classical_EM}
\end{gather}
Here $R_p$ is a remainder term that can be explicitly bounded in terms of the size of $f^{(p)}$ on $[a, b]$, and $B_k$ are the Bernoulli numbers defined by the relation:
\begin{gather}
    \frac{\alpha}{1-e^{-\alpha}} = \sum_{k = 0}^\infty \frac{B_k \alpha^k}{k!}. \label{bernoulli}
\end{gather}
For certain functions $f(\alpha)$ such as polynomials and exponential functions, in letting $p \to \infty$ in \eqref{classical_EM}, we get an exact equality of the two sides (i.e. no remainder). 

\begin{example} 
We can of course interpret $[a, b]$ as a lattice polytope in $\mathbb{R}$. Then if $f(x) = e^{\alpha \cdot x}$, the Euler-Maclaurin formula can be interpreted as giving the relationship between $I([a, b]; \alpha)$ and $S_{\mathbb{Z}}([a, b]; \alpha)$. The exact Euler-Maclaurin formula is also valid for rays in $\mathbb{R}$. For example, if we take the infinite interval $(-\infty, 0]$, then we get for $\alpha > 0$ and sufficiently small:
\begin{align*}
    \sum_{n = -\infty}^0 e^{\alpha \cdot n} - \int_{-\infty}^0 e^{\alpha \cdot x} dx  &= \frac{1}{1 - e^{-\alpha}} - \frac{1}{\alpha} \\
    &= \frac{1}{\alpha} (\frac{\alpha}{1 - e^{-\alpha}} - 1) \\
    &= \sum_{k = 1}^\infty \frac{B_k}{k!} \alpha^{k-1}.
\end{align*}
\end{example}

\subsection{Euler-Maclaurin formula for rational polytopes}
In the more general setting of integral (or rational) polytopes, this formula has been vastly generalized to give precise relationships between $I(\mathfrak{p}; \alpha)$ and $S_\Lambda(\mathfrak{p}; \alpha)$. Much of the development traces back to the work of Pukhlikov-Khovanskii \cite{pukhlikov_khovanskii} who connected the Euler-Maclaurin formula to an explicit version of the Riemann-Roch theorem for the toric variety associated to a simple unimodular lattice polytope. They re-express \eqref{classical_EM} as
\begin{gather*}
    S_{\mathbb{Z}}([a, b]; \alpha) = \textnormal{Td}(\frac{\partial}{\partial h_1}) \textnormal{Td}(\frac{\partial}{\partial h_2}) \bigg|_{h_1 = h_2 = 0} I([a - h_1, b + h_2]; \alpha),
\end{gather*}
where $\textnormal{Td}(\frac{\partial}{\partial h})$ is the so-called \textit{Todd operator} which is the differential operator whose symbol is given by \eqref{bernoulli}. This operator is closely related to the \textit{Todd classes} that show up in the Riemann-Roch formula.

Recalling the notation in Section \ref{sec_polytopes}, suppose $\mathfrak{p} = \mathfrak{p}(h^0)$ is a unimodular simple lattice polytope. Pukhlikov-Khovanskii \cite{pukhlikov_khovanskii} prove that
\begin{gather}
    S_{\mathbb{Z}^n}(\mathfrak{p}(h^0); \alpha) = \textnormal{Td}(\frac{\partial}{\partial h_1}) \dots \textnormal{Td}(\frac{\partial}{\partial h_k}) \bigg|_{h = h^0} I(\mathfrak{p}(h); \alpha). \label{PK_EM}
\end{gather}
Because Theorem \ref{thm_degen_brion_cont_intro} is of a nice form for understanding how $I(\mathfrak{p}(h); \alpha)$ varies in $h$, one could use \eqref{PK_EM} to obtain a discrete version of the degenerate Brion's formula, at least for unimodular lattice polytopes. However we ultimately instead use the Euler-Maclaurin formula of Berline-Vergne as it works in greater generality.

The exact Euler-Maclaurin formula for polytopes was extended by Cappell-Shaneson \cite{cappell_shaneson} and Brion-Vergne \cite{brion_vergne} to any rational polytope. In Berline-Vergne \cite{berline_vergne}, the authors take an alternate approach which is elementary, in the sense that it does not use the theory of toric varieties, and which results in a \textit{local} version of the Euler-Maclaurin formula. Roughly speaking, this means that they give a way to express $S_{\Lambda}(\mathfrak{p}; \alpha)$ as a weighted sum of the values of $I(\mathfrak{f}; \alpha)$ where $\mathfrak{f} \in \textnormal{Face}(\mathfrak{p})$, and furthermore, the weight given to $I(\mathfrak{f}; \alpha)$ only depends on the local geometry of $\mathfrak{f}$ in $\mathfrak{p}$ (and on $\alpha$), i.e. on the transverse cone $\mathfrak{t}_{\mathfrak{f}}^{\mathfrak{p}}$. This work builds on earlier work of Pommersheim-Thomas \cite{pommersheim_thomas} who gave such an expression in case $\alpha = 0$, i.e. a way to express the number of lattice points in $\mathfrak{p}$ as a weighted sum of the volumes of the faces of $\mathfrak{p}$.

\subsection{The local Euler-Maclaurin formula of Berline-Vergne} \label{sec_BV_euler_maclaurin}
The work of Berline-Vergne \cite{berline_vergne} requires a non-canonical choice of a rational inner product. Approaches to the problem without choosing an inner product can be found in \cite{garoufalidis_pommersheim} and \cite{fischer_pommersheim}. However, we have been generously and fruitfully utilizing the choice of an inner product throughout this work, and shall continue to do so. The main result of Berline-Vergne is the following:
\begin{theorem} [\cite{berline_vergne}] \label{thm_berline_vergne}
    There exists a unique family of maps $\mu_{W}^{\Gamma, \langle \cdot, \cdot \rangle'}$ indexed by rational inner product spaces, i.e. a vector space $W$ with a choice of lattice $\Gamma$ and a choice of rational inner product $\langle \cdot, \cdot \rangle'$, mapping rational cones in $W$ to meromorphic functions on $W_{\mathbb{C}}^*$ satisfying the following properties:
    \begin{enumerate}
    \item
        For any rational polyhedron $\mathfrak{q} \subset W$, we have the identity of meromorphic functions
        \begin{gather*}
            S_{\Gamma}(\mathfrak{q}; \alpha) = \sum_{\mathfrak{f} \in \textnormal{Face}(\mathfrak{q})} \mu_{\mathfrak{f}^\perp}^{\Gamma^{\mathfrak{f}^\perp}, \langle \cdot, \cdot \rangle'} (\mathfrak{t}_{\mathfrak{f}}^{\mathfrak{q}}; \alpha) \cdot I^{\Gamma_{\mathfrak{f}}}(\mathfrak{f}; \alpha).
        \end{gather*}
        Here we have extended $\mu_{\mathfrak{f}^\perp}^{\Gamma^{\mathfrak{f}^\perp}, \langle \cdot, \cdot \rangle} (\mathfrak{t}_{\mathfrak{f}}^{\mathfrak{q}}; \zeta)$ with $\zeta \in (\textnormal{lin}(\mathfrak{f})^\perp)^*_{\mathbb{C}}$ to $\alpha \in W^*_{\mathbb{C}}$ by pullback of the orthogonal projection $W^*_{\mathbb{C}} \to (\textnormal{lin}(\mathfrak{f})^\perp)^*_{\mathbb{C}}$, and we take the restriction of the inner product to $\mathfrak{f}^\perp$.
        \item For any rational cone $\mathfrak{k}$ and any element $\gamma \in \Gamma$, we have that
        \begin{gather*}
            \mu^{\Gamma, \langle \cdot, \cdot \rangle'}_W(\gamma + \mathfrak{k}; \alpha) = \mu^{\Gamma, \langle \cdot, \cdot \rangle'}_W (\mathfrak{k}; \alpha).
        \end{gather*}
        \item For any rational cone $\mathfrak{k}$, we have that $\mu_{W}^{\Gamma, \langle \cdot, \cdot \rangle'}(\mathfrak{k}; \alpha)$ is holomorphic on a neighborhood of $\alpha = 0$.
        \item If the rational cone $\mathfrak{k}$ contains a line, then $\mu_{W}^{\Gamma, \langle \cdot, \cdot \rangle'}(\mathfrak{k}; \alpha) = 0$.
        \item If $W = 0$ (so then $\Gamma = 0$ necessarily), then $\mu_{W}^{\Gamma, \langle \cdot, \cdot \rangle'}(W; \alpha) = 1$. 
        \item The map $\mu_{W}^{\Gamma, \langle \cdot, \cdot \rangle}(\mathfrak{k}; \alpha)$ is a valuation on all rational cones $\mathfrak{k}$ with any particular fixed vertex.
    \end{enumerate}
\end{theorem}

Throughout this paper we have chosen a fixed inner product and, when dealing with exponential sums, we always assume it is rational. As such we shall often omit the data of the inner product in the superscript for $\mu$. Recall that $\mathfrak{f}^\perp$ is shorthand for the space $\textnormal{lin}(\mathfrak{f})^\perp$.

Property (1) in Theorem \ref{thm_berline_vergne} is obviously the most important. However before applying the theorem, we wish to elaborate further on Property (3), i.e. the holomorphicity of $\mu$. In the case where $W = \mathbb{R}$ and $\Gamma = \mathbb{Z}$, then the only rational cones are of the form $[s, \infty)$ or $(-\infty, s]$ with $s$ rational. As computed in \cite{berline_vergne},
\begin{gather}
    \mu_{\mathbb{R}}^{\mathbb{Z}}([s, \infty); \alpha) = \frac{e^{[[s]] \alpha}}{1 - e^{\alpha}} + \frac{1}{\alpha}, \label{1d_example}
\end{gather}
where $[[s]]$ is the unique representative of $-s + \mathbb{Z}$ in the interval $[0, 1)$. This function has potential singularities on $2 \pi i \mathbb{Z}$. On the other hand, it is actually holomorphic at $\alpha = 0$, and the value there is $\frac{1}{2} - [[s]]$. However, \eqref{1d_example} is genuinely singular at any other point in $2 \pi i \mathbb{Z}$. The point 0 is clearly in $(V_{\mathbb{C}}^*)^{\mathbb{Z}}$, but no other point in $2 \pi i \mathbb{Z}$ is. We wish to generalize this observation.

\subsection{Holomorphicity of the functions $\mu_W^{\Gamma}$}

\begin{proposition} \label{prop_mu_holomorphic}
    Let $\Gamma$ be a lattice in $V$. Suppose $\xi \in (V^*_{\mathbb{C}})^\Gamma$. Then for any rational cone $\mathfrak{k}$, we have that $\mu_V^{\Gamma}(\mathfrak{k}; \alpha)$ is holomorphic at $\alpha = \xi$.
\end{proposition}

\begin{proof}
    The proof essentially follows in the same way that Property (3) in Theorem \ref{thm_berline_vergne} is proven in \cite{berline_vergne}.

    Suppose $\xi \in (V^*_{\mathbb{C}})^\Gamma$. Let $W$ be some rational subspace of $V$. We claim that $\xi^W \in (W^*_{\mathbb{C}})^{\Gamma^W}$. To see this: suppose $\gamma \in \Gamma^W$ and that $e^{\langle \xi, \gamma \rangle} = 1$. Because $\Gamma_W$ is finite index in $\Gamma^W$, we must have that some non-zero integer multiple $k \cdot \gamma$ is in $\Gamma_W$, and hence in $\Gamma$. We then clearly have that $e^{\langle \xi, k \cdot \gamma \rangle} = 1$, and thus $\langle \xi, k \cdot \gamma \rangle = 0$, which implies that $\langle \xi, \gamma \rangle = 0$. Therefore $\xi^W \in (W^*_{\mathbb{C}})^{\Gamma^W}$.

    We now proceed by induction on the dimension of $V$. The base case amounts to the discussion in Section \ref{sec_BV_euler_maclaurin}. Let $\mathfrak{k}$ be a rational pointed cone with vertex $\mathfrak{s}$. By the valuation property of $\mu$, i.e. Property (6) in Theorem \ref{thm_berline_vergne}, together with Barvinok's signed decomposition of a rational cone into simplicial unimodular cones \cite{barvinok_signed_decomp}, we can assume that $\mathfrak{k}$ is simplicial and unimodular. The function $\mu_{V}^{\Gamma}(\mathfrak{k}; \alpha)$ is defined in \cite{berline_vergne} inductively by the formula
    \begin{gather}
        S_\Gamma(\mathfrak{k}; \alpha) = e^{\langle \alpha, \mathfrak{s} \rangle} \mu_{V}^{\Gamma}(\mathfrak{k}; \alpha) + \sum_{\mathfrak{f} \in \textnormal{Face}(\mathfrak{k}), \textnormal{dim}(\mathfrak{f}) > 0} \mu_{\mathfrak{f}^\perp}^{\Gamma^{\mathfrak{f}^\perp}}(\mathfrak{t}_{\mathfrak{f}}^{\mathfrak{k}}; \alpha) \cdot I^{\Gamma^{\mathfrak{f}^\perp}}(\mathfrak{f}; \alpha). \label{def_BV}
    \end{gather}
    Utilizing the discussion in the previous paragraph, we inductively know that $\mu_{\mathfrak{f}^\perp}^{\Gamma^{\mathfrak{f}^\perp}}(\mathfrak{t}_{\mathfrak{f}}^{\mathfrak{k}}; \alpha)$ is holomorphic at $\alpha = \xi$. When we isolate $\mu_{V}^{\Gamma}(\mathfrak{k}; \alpha)$ in \eqref{def_BV}, we obtain an expression
    \begin{gather*}
        \mu_{V}^{\Gamma}(\mathfrak{k}; \alpha) = \frac{f(\alpha)}{\prod_j (1 - e^{\langle \alpha, v_j \rangle}) \cdot \langle \alpha, v_j \rangle},
    \end{gather*}
    where $v_j$ are the primitive lattice points on extremal rays of $\mathfrak{k}$, and where $f(\alpha)$ is holomorphic at any point of $(V^*_{\mathbb{C}})^\Gamma$. We furthermore have that the function
    \begin{gather*}
        \frac{\langle \alpha, v_j \rangle}{1 - e^{\langle \alpha, v_j \rangle}}
    \end{gather*}
    is holomorphic on the locus $\langle \alpha, v_j \rangle \neq 2 \pi i \mathbb{Z} \setminus \{0\}$; in particular it is holomorphic on the locus $\langle \alpha, v_j \rangle = 0$. Thus it is holomorphic at any point of $(V^*_{\mathbb{C}})^\Gamma$. We therefore have
    \begin{gather*}
        \mu_{V}^{\Gamma}(\mathfrak{k}; \alpha) = \frac{g(\alpha)}{\prod_j \langle \alpha, v_j \rangle^2},
    \end{gather*}
    where $g(\alpha)$ is holomorphic at any point in $(V^*_{\mathbb{C}})^\Gamma$. Thus the only points in $(V^*_{\mathbb{C}})^\Gamma$ on which $\mu_{V}^{\Gamma}(\mathfrak{k}; \alpha)$ could be singular are points also lying on some $\langle \alpha, v_j \rangle = 0$.

    In \cite{berline_vergne}, in the course of proving Property (3), Berline-Vergne show that $\Big( \prod_j \langle \alpha, v_j \rangle \Big) \mu_V^\Gamma(\mathfrak{k}; \alpha)$ is holomorphic near zero; in fact their proof shows that it is holomorphic on $(V_{\mathbb{C}}^*)^\Gamma$. They also show that $\langle \alpha, v_k \rangle \mu_V^\Gamma(\mathfrak{k}; \alpha)$ is identically zero on $\langle \alpha, v_k \rangle = 0$. Thus, by considering the Taylor series for $\Big( \prod_j \langle \alpha, v_j \rangle \Big) \mu_V^\Gamma(\mathfrak{k}; \alpha)$ at any point on $\{\langle \alpha, v_k \rangle = 0\} \cap (V_\mathbb{C}^*)^\Gamma$, we conclude that the function $\Big( \prod_{j \neq k} \langle \alpha, v_j \rangle \Big) \mu_V^\Gamma(\mathfrak{k}; \alpha)$ is also holomorphic on $(V_{\mathbb{C}}^*)^\Gamma$. We may thus repeatedly remove linear factors to finally conclude that $\mu_V^\Gamma(\mathfrak{k}; \alpha)$ is holomorphic on $(V_{\mathbb{C}}^*)^\Gamma$.
\end{proof}

\begin{remark}
    It may seem strange to discuss holomorphicity of functions on the set $(V_{\mathbb{C}}^*)^\Gamma$ as this set is not open. For a fixed rational cone $\mathfrak{k}$ and a point $\xi \in (V_{\mathbb{C}}^*)^\Gamma$ we can find an open neighborhood of $\xi$ on which $\mu_V^{\Gamma}(\mathfrak{k}; \alpha)$ is holomorphic, but the size of that neighborhood depends on $\mathfrak{k}$. Thus, when considering \textit{all} rational cones $\mathfrak{k}$, we cannot always find a common open set around $\xi$ on which all of the functions $\mu_V^{\Gamma}(\mathfrak{k}; \alpha)$ are holomorphic.
\end{remark}

\section{Degenerate Brion's formula: discrete setting}

Throughout this section we assume that we are working in a rational inner product space $V$ with a choice of lattice $\Lambda$. We shall furthermore usually assume that $\xi \in (V^*_{\mathbb{C}})^{\Lambda}$. The results of this section easily extend to arbitrary $\xi \in V^*_{\mathbb{C}}$ using \eqref{eqn_reduce_to_adapted}.

\subsection{Proto-degenerate Brion's formula: discrete setting}

The valuation property of $S_{\Lambda}(\cdot; \cdot)$ and the fact that it sends polyhedra containing lines to zero, together with Theorem \ref{polytope_decomp}, immediately imply:

\begin{proposition}[Proto-degenerate Brion's formula: discrete setting] \label{prop_proto_brion_discrete}
    Let $\mathfrak{p}$ be a rational polytope with respect to $\Lambda$, and let $\xi \in V^*_{\mathbb{C}}$. Then we have the identity of meromorphic functions:
    \begin{gather*}
        S_{\Lambda}(\mathfrak{p}; \alpha) = \sum_{\mathfrak{g} \in \{\mathfrak{p}\}_\xi} S_{\Lambda}(\mathfrak{g}^{\mathfrak{g}} \times \textnormal{LC}_{\mathfrak{g}}^{\mathfrak{p}}(\xi); \alpha).
    \end{gather*}
\end{proposition}

First off, it is not yet apparent just from this formula that each individual term is holomorphic at $\xi$. However, we shall ultimately prove this when $\xi \in (V^*_{\mathbb{C}})^{\Lambda}$. Furthermore, it would be nice to re-rexpress $S_{\Lambda}(\mathfrak{g}^{\mathfrak{g}} \times \textnormal{LC}_{\mathfrak{g}}^{\mathfrak{p}}(\xi); \xi)$ in terms of ``$S_\Lambda(\mathfrak{g}^{\mathfrak{g}}; \xi)$'' and ``$S_\Lambda(\textnormal{LC}_{\mathfrak{g}}^{\mathfrak{p}}(\xi); \xi)$''. However, the subspaces $\textnormal{lin}(\mathfrak{g})$ and $\textnormal{lin}(\mathfrak{g})^\perp$ carry two natural lattices, namely the intersection of $\Lambda$ with the subspace, and the projection of $\Lambda$ onto the subspace. It is not immediately clear which of these lattices we should use in such a decomposition; in Section \ref{sec_decomposing} we will see that in some sense we should use both.

\subsection{Holomorphicity of $S_\Lambda(\textnormal{LC}_{\mathfrak{g}}^{\mathfrak{p}}(\xi); \tau)$ at $\tau = \xi$}

Let $\xi \in (V^*_{\mathbb{C}})^{\Lambda}$. We first prove that $S_\Lambda(\textnormal{LC}_{\mathfrak{g}}^{\mathfrak{p}}(\xi); \tau)$ with $\tau \in (\textnormal{lin}(\mathfrak{g})^\perp)^*_{\mathbb{C}}$ is holomorphic at $\tau = \xi$. We use the notation $\{(\mathfrak{g}, \mathfrak{p})\}_\xi$ as follows:
\begin{gather*}
    \{(\mathfrak{g}, \mathfrak{p})\}_\xi := \{(\mathfrak{f}, \mathfrak{m}) : \mathfrak{g} \subseteq \mathfrak{f} \subseteq \mathfrak{m}, \textnormal{ and } \mathfrak{f} \in \{\mathfrak{p}\}_\xi, \textnormal{ and either } \mathfrak{m} = \mathfrak{f} \textnormal{ or } \mathfrak{m} \notin \{\mathfrak{p}\}_\xi \}.
\end{gather*}
Given $(\mathfrak{f}, \mathfrak{m}) \in \{(\mathfrak{g}, \mathfrak{p})\}_\xi$, we define
\begin{align*}
    W(\mathfrak{f}, \mathfrak{m}) &:= \textnormal{lin}(\mathfrak{f})^\perp \cap \textnormal{lin}(\mathfrak{m}) \subseteq \textnormal{lin}(\mathfrak{g})^\perp, \\
    W(\mathfrak{f}, \mathfrak{m})^\perp & := (\textnormal{lin}(\mathfrak{f}) \cap \textnormal{lin}(\mathfrak{g})^\perp) \oplus \textnormal{lin}(\mathfrak{m})^\perp \subseteq \textnormal{lin}(\mathfrak{g})^\perp.
\end{align*}
Note that $\textnormal{lin}(\mathfrak{g})^\perp = W(\mathfrak{f}, \mathfrak{m}) \oplus W(\mathfrak{f}, \mathfrak{m})^\perp$. We shall use the notation
\begin{align*}
    \mu^{(\mathfrak{f}, \mathfrak{m})}(\cdot; \cdot) &:= \mu^{\Lambda^{W(\mathfrak{f}, \mathfrak{m})^\perp}}_{W(\mathfrak{f}, \mathfrak{m})^\perp}(\cdot; \cdot), \\
    I^{(\mathfrak{f}, \mathfrak{m})}(\cdot, \cdot) & := I^{\Lambda_{W(\mathfrak{f}, \mathfrak{m})}}(\cdot; \cdot).
\end{align*}
Let $\textnormal{pFl}_{\mathfrak{g}}^{\mathfrak{f}}$ denote the set of all partial flags that start at $\mathfrak{g}$ and end at $\mathfrak{f}$.

\begin{theorem} \label{thm_levi_cone_euler_maclaurin}
    Suppose $\mathfrak{p}$ is a rational polytope with respect to $\Lambda$. Let $\xi \in V^*_{\mathbb{C}}$, and let $\mathfrak{g} \in \{\mathfrak{p}\}_\xi$. Then we have the identity of meromorphic functions:
    \begin{align}
    & S_{\Lambda}(\textnormal{LC}_{\mathfrak{g}}^{\mathfrak{p}}(\xi); \tau) =  \notag \\
    & \sum_{(\mathfrak{f}, \mathfrak{m}) \in \{(\mathfrak{g}, \mathfrak{p})\}_\xi} \mu^{(\mathfrak{f}, \mathfrak{m})} \Big( \sum_{(\mathfrak{g} = \mathfrak{h}_0 \subset \mathfrak{h}_1 \subset \dots \subset \mathfrak{h}_\ell = \mathfrak{f}) \in \textnormal{pFl}_{\mathfrak{g}}^{\mathfrak{f}}} (-1)^\ell \mathfrak{t}_{\mathfrak{g}}^{\mathfrak{h}_1} \times \mathfrak{t}_{\mathfrak{h}_1}^{\mathfrak{h}_2} \times \dots \times \mathfrak{t}_{\mathfrak{h}_{\ell-1}}^{\mathfrak{f}} \times \mathfrak{t}_{\mathfrak{m}}^{\mathfrak{p}}; \tau \Big) I^{(\mathfrak{f}, \mathfrak{m})}(\textnormal{LC}_{\mathfrak{f}}^{\mathfrak{m}}(\xi); \tau). \label{eqn_EM_LC}
\end{align}
If in fact $\xi \in (V^*_{\mathbb{C}})^\Lambda$, then $S_\Lambda(\textnormal{LC}_{\mathfrak{g}}^{\mathfrak{p}}; \tau)$ is holomorphic at $\tau = \xi$.
\end{theorem}

\begin{proof}
    We are tasked with computing $S_\Lambda(\textnormal{LC}_{\mathfrak{g}}^{\mathfrak{p}}(\xi); \tau)$, where
\begin{gather}
    \textnormal{LC}_{\mathfrak{g}}^{\mathfrak{p}}(\xi) = \sum_{(\mathfrak{g} = \mathfrak{h}_0 \subset \dots \subset \mathfrak{h}_\ell) \in \textnormal{mFl}_{\mathfrak{g}}^{\mathfrak{p}}(\xi)} (-1)^{\ell} \mathfrak{t}_{\mathfrak{g}}^{\mathfrak{h}_1} \times \dots \times \mathfrak{t}_{\mathfrak{h}_\ell}^{\mathfrak{p}}. \label{eqn_full_cone}
\end{gather}
We can apply the Euler-Maclaurin formula to each individual term. The faces of a product of polyhedra naturally correspond to products of faces with a choice of one face from each constituent polyhedron. Given a flag $\mathfrak{H} = (\mathfrak{g} = \mathfrak{h}_0 \subset \mathfrak{h}_1 \subset \dots \subset \mathfrak{h}_\ell \subset \mathfrak{p} = \mathfrak{h}_{\ell+1})$, let $\textnormal{iFl}(\mathfrak{H})$ denote all chains of the form $\mathfrak{m}_0 \subseteq \mathfrak{m}_1 \subseteq \dots \subseteq \mathfrak{m}_\ell$ satisfying the interlacing property $\mathfrak{h}_j \subseteq \mathfrak{m}_j \subseteq \mathfrak{h}_{j+1}$. We then get that
\begin{gather}
    S_{\Lambda}(\mathfrak{t}_{\mathfrak{g}}^{\mathfrak{h}_1} \times \dots \times \mathfrak{t}_{\mathfrak{h}_\ell}^{\mathfrak{p}}; \tau) = \sum_{(\mathfrak{m}_0 \subseteq \dots \subseteq \mathfrak{m}_\ell) \in \textnormal{iFl}(\mathfrak{H})} \mu'(\mathfrak{t}_{\mathfrak{m}_0}^{\mathfrak{h}_1} \times \dots \times \mathfrak{t}_{\mathfrak{m}_\ell}^{\mathfrak{p}}; \tau) \cdot I'(\mathfrak{t}_{\mathfrak{g}}^{\mathfrak{m}_0} \times \dots \times \mathfrak{t}_{\mathfrak{h}_{\ell}}^{\mathfrak{m}_\ell}; \tau), \label{eqn_one_cone}
\end{gather}
where $\mu'$ means the $\mu$ function with respect to the proper choice of lattice and subspace, and $I'(\cdot; \cdot)$ means we integrate with respect to the proper normalization of the Lebesgue measure as in Theorem \ref{thm_berline_vergne}.

We can use \eqref{eqn_one_cone} to expand out all terms in \eqref{eqn_full_cone}. Suppose we fix a cone $\mathfrak{k} = \mathfrak{t}_{\mathfrak{m}_0}^{\mathfrak{h}_1} \times \dots \times \mathfrak{t}_{\mathfrak{m}_{\ell-1}}^{\mathfrak{h}_\ell} \times \mathfrak{t}_{\mathfrak{m}_\ell}^{\mathfrak{p}}$. We wish to understand the coefficient of $\mu'(\mathfrak{k}; \tau)$ when we group terms after this expansion. Given a flag $\mathfrak{F} = (\mathfrak{g} = \mathfrak{f}_0 \subset \mathfrak{f}_1 \subset \dots \subset \mathfrak{f}_t \subset \mathfrak{p} = \mathfrak{f}_{t+1})$ with $(\mathfrak{f}_0 \subset \dots \subset \mathfrak{f}_{t}) \in \textnormal{mFl}_{\mathfrak{g}}^{\mathfrak{p}}(\xi)$, the corresponding terms of the form $\mu'(\mathfrak{a}; \tau)$ that we can get from expanding \eqref{eqn_full_cone} are cones $\mathfrak{a}$ of the form
\begin{gather*}
    \mathfrak{a} = \mathfrak{t}_{\mathfrak{n}_0}^{\mathfrak{f}_1} \times \mathfrak{t}_{\mathfrak{n}_1}^{\mathfrak{f}_2} \times \dots \times \mathfrak{t}_{\mathfrak{n}_{t-1}}^{\mathfrak{f}_t} \times \mathfrak{t}_{\mathfrak{n}_t}^{\mathfrak{p}},
\end{gather*}
where $(\mathfrak{n}_0 \subseteq \dots \subseteq \mathfrak{n}_t) \in \textnormal{iFl}(\mathfrak{F})$. We see that if $\mathfrak{a} = \mathfrak{k}$, then there must exist $k_0, k_1, \dots, k_\ell$ such that $\mathfrak{n}_{k_j} = \mathfrak{m}_{j}$ and $\mathfrak{f}_{k_j+1} = \mathfrak{h}_{j+1}$. Furthermore, for the remaining indices we must have that $\mathfrak{n}_0 = \mathfrak{f}_1, \mathfrak{n}_1 = \mathfrak{f}_2, \dots, \mathfrak{n}_{k_0-1} = \mathfrak{f}_{k_0}$, and $\mathfrak{n}_{k_0+1} = \mathfrak{f}_{k_0+2}, \dots, \mathfrak{n}_{k_1 - 1} = \mathfrak{f}_{k_1}$, etc. all the way to $\mathfrak{n}_{k_{\ell-1}+1} = \mathfrak{f}_{k_{\ell-1}+2}, \dots, \mathfrak{n}_{k_{\ell}-1} = \mathfrak{f}_{k_\ell}$. This contributes the following term towards the coefficient of $\mu'(\mathfrak{k}; \xi)$:
\begin{gather*}
    (-1)^t I'(\mathfrak{t}_{\mathfrak{g}}^{\mathfrak{f}_1} \times \mathfrak{t}_{\mathfrak{f}_1}^{\mathfrak{f}_2} \times \dots \times \mathfrak{t}_{\mathfrak{f}_{k_0}}^{\mathfrak{m}_0} \times \mathfrak{t}_{\mathfrak{h}_1 = \mathfrak{f}_{k_0+1}}^{\mathfrak{f}_{k_0+2}} \times \dots \times \mathfrak{t}_{\mathfrak{f}_{k_1}}^{\mathfrak{m}_1} \times \mathfrak{t}_{\mathfrak{h}_2 = \mathfrak{f}_{k_1+1}}^{\mathfrak{f}_{k_1+2}} \times \dots \times \mathfrak{t}_{\mathfrak{f}_{k_\ell} = \mathfrak{f}_t}^{\mathfrak{m}_\ell}; \tau).
\end{gather*}
Let $\textnormal{len}(\mathfrak{F}) = \ell$, and let $\mathfrak{t}_{\mathfrak{F}} := \mathfrak{t}_{\mathfrak{f}_0}^{\mathfrak{f}_1} \times \dots \times \mathfrak{t}_{\mathfrak{f}_\ell}^{\mathfrak{p}}$. We thus get that altogether the coefficient of $\mu'(\mathfrak{k}; \tau)$ is
\begin{align*}
    & I'\Bigg((-1)^\ell \Big(\sum_{\mathfrak{F}_0 \in \textnormal{mFl}_{\mathfrak{g}}^{\mathfrak{m}_0}(\xi)} (-1)^{\textnormal{len}(\mathfrak{F}_0)} \mathfrak{t}_{\mathfrak{F}_0} \Big) \Big(\sum_{\mathfrak{F}_1 \in \textnormal{mFl}_{\mathfrak{h}_1}^{\mathfrak{m}_1}(\xi)} (-1)^{\textnormal{len}(\mathfrak{F}_1)} \mathfrak{t}_{\mathfrak{F}_1}\Big) \dots \Big(\sum_{\mathfrak{F}_\ell \in \textnormal{mFl}_{\mathfrak{h}_\ell}^{\mathfrak{m}_\ell}(\xi)} (-1)^{\textnormal{len}(\mathfrak{F}_\ell)} \mathfrak{t}_{\mathfrak{F}_\ell}\Big); \tau \Bigg) \\
    & = (-1)^\ell I'(\textnormal{LC}_{\mathfrak{g}}^{\mathfrak{m}_0}(\xi); \tau^{\textnormal{lin}(\mathfrak{m}_0) \cap \textnormal{lin}(\mathfrak{g})^\perp}) \cdot I'(\textnormal{LC}_{\mathfrak{h}_1}^{\mathfrak{m}_1}(\xi); \tau^{\textnormal{lin}(\mathfrak{m}_1) \cap \textnormal{lin}(\mathfrak{h}_1)^\perp}) \dots I'(\textnormal{LC}_{\mathfrak{h}_\ell}^{\mathfrak{m}_\ell}(\xi); \tau^{\textnormal{lin}(\mathfrak{m}_\ell) \cap \textnormal{lin}(\mathfrak{h}_\ell)^\perp}).
\end{align*}
Making use of Proposition \ref{prop_triv_or_one}, notice that this expression is zero unless $\mathfrak{g} = \mathfrak{m}_0, \mathfrak{h}_1 = \mathfrak{m}_1, \dots, \mathfrak{h}_{\ell-1} = \mathfrak{m}_{\ell-1}$ which puts restrictions on which $\mu'(\mathfrak{k}; \tau)$ have non-zero coefficient. The non-vanishing terms are thus of the form
\begin{gather*}
    (-1)^\ell \mu'(\mathfrak{t}_{\mathfrak{g}}^{\mathfrak{h}_1} \times \mathfrak{t}_{\mathfrak{h}_1}^{\mathfrak{h}_2} \times \dots \times \mathfrak{t}_{\mathfrak{h}_{\ell-1}}^{\mathfrak{h}_\ell} \times \mathfrak{t}_{\mathfrak{m}_\ell}^{\mathfrak{p}}; \tau) \cdot I'(\textnormal{LC}_{\mathfrak{h}_\ell}^{\mathfrak{m}_\ell}(\xi); \tau).
\end{gather*}
We can now group the terms that involve a given $I'(\textnormal{LC}_{\mathfrak{f}}^{\mathfrak{m}}(\xi); \tau)$. We then get that
\begin{align*}
    & S_{\Lambda}(\textnormal{LC}_{\mathfrak{g}}^{\mathfrak{p}}(\xi); \tau) \\
    &= \sum_{(\mathfrak{f}, \mathfrak{m}) \in \{(\mathfrak{g}, \mathfrak{p})\}_\xi} \mu' \Big( \sum_{(\mathfrak{g} = \mathfrak{h}_0 \subset \mathfrak{h}_1 \subset \dots \subset \mathfrak{h}_\ell = \mathfrak{f}) \in \textnormal{pFl}_{\mathfrak{g}}^{\mathfrak{f}}} (-1)^\ell \mathfrak{t}_{\mathfrak{g}}^{\mathfrak{h}_1} \times \mathfrak{t}_{\mathfrak{h}_1}^{\mathfrak{h}_2} \times \dots \times \mathfrak{t}_{\mathfrak{h}_{\ell-1}}^{\mathfrak{f}} \times \mathfrak{t}_{\mathfrak{m}}^{\mathfrak{p}}; \tau \Big) I'(\textnormal{LC}_{\mathfrak{f}}^{\mathfrak{m}}(\xi); \tau).
\end{align*}

By Theorem \ref{thm_holomorphic}, we know that $I^{(\mathfrak{f}, \mathfrak{m})}(\textnormal{LC}_{\mathfrak{f}}^{\mathfrak{m}}(\xi); \tau)$ is holomorphic at $\tau = \xi$. By Proposition \ref{prop_mu_holomorphic}, we also know that $\mu^{(\mathfrak{f}, \mathfrak{m})}(\mathfrak{k}; \tau)$ is holomorphic at $\tau = \xi$ if $\xi \in (V^*_{\mathbb{C}})^\Lambda$, with $\mathfrak{k}$ as one of the cones appearing in \eqref{eqn_EM_LC}. From the formula we thus obtain that $S_\Lambda(\textnormal{LC}_{\mathfrak{g}}^{\mathfrak{p}}(\xi); \tau)$ is holomorphic at $\tau = \xi$ if $\xi \in (V^*_{\mathbb{C}})^\Lambda$. 
\end{proof}

\begin{example}
Let's do an example illustrating Theorem \ref{thm_levi_cone_euler_maclaurin}. Suppose we have that $\mathfrak{f}$ is $\xi$-maximal, $\mathfrak{g}$ is a facet in $\mathfrak{f}$, and $\mathfrak{f}$ is the only $\xi$-maximal face containing $\mathfrak{g}$. Then $\textnormal{LC}_{\mathfrak{g}}^{\mathfrak{p}}(\xi) = \mathfrak{t}_{\mathfrak{g}}^{\mathfrak{p}} - \mathfrak{t}_{\mathfrak{g}}^{\mathfrak{f}} \times \mathfrak{t}_{\mathfrak{f}}^{\mathfrak{p}}$. We get
\begin{align*}
    S_{\Lambda}(\mathfrak{t}_{\mathfrak{g}}^{\mathfrak{p}}; \tau) &= \sum_{\mathfrak{f} \subseteq \mathfrak{m}} \mu'(\mathfrak{t}_{\mathfrak{m}}^{\mathfrak{p}}; \tau) \cdot I'(\mathfrak{t}_{\mathfrak{g}}^{\mathfrak{m}}; \xi) + \sum_{\mathfrak{g} \subseteq \mathfrak{m}, \mathfrak{f} \nsubseteq \mathfrak{m}} \mu'(\mathfrak{t}_{\mathfrak{m}}^{\mathfrak{p}}; \tau) \cdot I'(\mathfrak{t}_{\mathfrak{g}}^{\mathfrak{m}}; \tau) \\
    S_{\Lambda}(\mathfrak{t}_{\mathfrak{g}}^{\mathfrak{f}} \times \mathfrak{t}_{\mathfrak{f}}^{\mathfrak{p}}; \tau) &= \sum_{\mathfrak{f} \subseteq \mathfrak{m}} \mu'(\mathfrak{t}_{\mathfrak{g}}^{\mathfrak{f}} \times \mathfrak{t}_{\mathfrak{m}}^{\mathfrak{p}}; \tau) \cdot I'(\mathfrak{t}_{\mathfrak{g}}^{\mathfrak{g}} \times \mathfrak{t}_{\mathfrak{f}}^{\mathfrak{m}}; \tau) + \mu'(\mathfrak{t}_{\mathfrak{f}}^{\mathfrak{f}} \times \mathfrak{t}_{\mathfrak{m}}^{\mathfrak{p}}; \tau) \cdot I'(\mathfrak{t}_{\mathfrak{g}}^{\mathfrak{f}} \times \mathfrak{t}_{\mathfrak{f}}^{\mathfrak{m}}; \tau) \\
    S_{\Lambda}(\textnormal{LC}_{\mathfrak{g}}^{\mathfrak{p}}(\xi); \tau) &= \sum_{\mathfrak{g} \subseteq \mathfrak{m}, \mathfrak{f} \nsubseteq \mathfrak{m}} \mu'(\mathfrak{t}_{\mathfrak{m}}^{\mathfrak{p}}; \tau) \cdot I'(\mathfrak{t}_{\mathfrak{g}}^{\mathfrak{m}}; \tau) + \sum_{\mathfrak{f} \subseteq \mathfrak{m}} \mu'(\mathfrak{t}_{\mathfrak{m}}^{\mathfrak{p}}; \tau) \cdot I'(\mathfrak{t}_{\mathfrak{g}}^{\mathfrak{m}} - \mathfrak{t}_{\mathfrak{g}}^{\mathfrak{f}} \times \mathfrak{t}_{\mathfrak{f}}^{\mathfrak{m}}; \tau) \\
    & \ \ - \sum_{\mathfrak{f} \subseteq \mathfrak{m}} \mu'(\mathfrak{t}_{\mathfrak{g}}^{\mathfrak{f}} \times \mathfrak{t}_{\mathfrak{m}}^{\mathfrak{p}}; \tau) \cdot I' (\mathfrak{t}_{\mathfrak{f}}^{\mathfrak{m}}; \tau) \\
    &= \sum_{(\mathfrak{g}, \mathfrak{m}) \in \{(\mathfrak{g}, \mathfrak{p})\}_{\xi}} \mu'(\mathfrak{t}_{\mathfrak{g}}^{\mathfrak{g}} \times \mathfrak{t}_\mathfrak{m}^{\mathfrak{p}}; \tau) \cdot I'(\textnormal{LC}_{\mathfrak{g}}^{\mathfrak{m}}(\xi); \tau) \\
    & \ \ + \sum_{(\mathfrak{f}, \mathfrak{m}) \in \{(\mathfrak{g}, \mathfrak{p})\}_{\xi}} \mu'(-(\mathfrak{t}_{\mathfrak{g}}^{\mathfrak{f}} \times \mathfrak{t}_{\mathfrak{m}}^{\mathfrak{p}}); \tau) \cdot I'(\textnormal{LC}_{\mathfrak{f}}^{\mathfrak{m}}(\xi); \tau).
\end{align*}
\end{example}

\subsection{Decomposing $S_{\Lambda}(\mathfrak{q}_1 \times \mathfrak{q}_2; \alpha)$} \label{sec_decomposing}

Suppose $\Lambda = \Lambda_1 \oplus \Lambda_2$ and $\Lambda_1$ and $\Lambda_2$ are orthogonal. Let $W_i = \Lambda_i \otimes_{\mathbb{Z}} \mathbb{R}$, so then we get the orthogonal decomposition $V = W_1 \oplus W_2$. Suppose $\alpha = \alpha_1 + \alpha_2$ with respect to this decomposition. If we have rational polyhedra $\mathfrak{q}_i \subset W_i$, then $S_\Lambda(\mathfrak{q}_1 \times \mathfrak{q}_2; \alpha) = S_{\Lambda_1}(\mathfrak{q}_1; \alpha_1) \cdot S_{\Lambda_2}(\mathfrak{q}_2; \alpha_2)$. This is the direct analogue of Proposition \ref{prop_multiplicative}.

However, suppose we merely have a decomposition of $V$ into orthogonal rational subspaces $V = W_1 \oplus W_2$, and suppose $\mathfrak{q}_i \subset W_i$ are rational polyhedra. Then we no longer have multiplicativity of $S_\Lambda(\mathfrak{p}_1 \times \mathfrak{p}_2; \alpha)$. In fact it is not even clear what such multiplicativity would mean as each subspace $W_i$ contains two distinct natural choices of lattice, namely $\Lambda^i := \Lambda^{W_i}$ and $\Lambda_i := \Lambda_{W_i}$.

\begin{proposition} \label{prop_ortho_lattice}
    Suppose $V = W_1 \oplus W_2$ decomposes into a direct sum of orthogonal rational subspaces. Then
    \begin{gather}
        \Lambda^{1}/\Lambda_{1} \xleftarrow[\phi_1]{\sim} \Lambda/(\Lambda_{1} \oplus \Lambda_{2}) \xrightarrow[\phi_2]{\sim} \Lambda^{2}/\Lambda_{2}. \label{eqn_maps_phi}
    \end{gather}
\end{proposition}
\begin{proof}
    We clearly have a homomorphism $\Lambda \to \Lambda^1$ via orthogonal projection whose kernel is $\Lambda_2$. We have a further homomorphism from $\Lambda^1 \to \Lambda^1/\Lambda_1$ whose kernel is of course $\Lambda_1$. In composing these maps we get the result.
\end{proof}

Note that the maps $\phi_1$ and $\phi_2$ in \eqref{eqn_maps_phi} simply correspond to mapping $[\gamma] \in \Lambda/(\Lambda_1 \oplus \Lambda_2)$ to $[\gamma^{W_1}] + \Lambda_1$ and $[\gamma^{W_2}] + \Lambda_2$ respectively. See Figure \ref{fig_lattice} for an illustration of Proposition \ref{prop_ortho_lattice}. Recall that if $V$ carries an inner product and $\Gamma \subset W$ is a lattice, then $\textnormal{gr}(\Gamma)$ is the covolume of $\Gamma$ with respect to the volume form on $W$ coming from the inner product.

\begin{figure} 
\centering
\begin{tikzpicture}[scale=1.0]
\begin{scope}[xshift = 8 cm]
\node[yshift = -5mm] at (3, -1) {\textcolor{blue}{$W_1$}};
\node[yshift = 5mm, xshift = -3mm] at (1, 3) {\textcolor{red}{$W_2$}};
\node at (0,-1) [circle,fill,inner sep=1.0pt]{};
\node at (1,-1) [circle,fill,inner sep=1.0pt]{};
\node at (2,-1) [circle,fill,inner sep=1.0pt]{};
\node at (3,-1) [circle,fill,inner sep=1.0pt]{};
\node at (4,-1) [circle,fill,inner sep=1.0pt]{};
\node at (0,0) [circle,fill,inner sep=1.0pt]{};
\node at (1,0) [circle,fill,inner sep=1.0pt]{};
\node at (2,0) [circle,fill,inner sep=1.0pt]{};
\node at (3,0) [circle,fill,inner sep=1.0pt]{};
\node at (4,0) [circle,fill,inner sep=1.0pt]{};
\node at (0,1) [circle,fill,inner sep=1.0pt]{};
\node at (1,1) [circle,fill,inner sep=1.0pt]{};
\node at (2,1) [circle,fill,inner sep=1.0pt]{};
\node at (3,1) [circle,fill,inner sep=1.0pt]{};
\node at (4,1) [circle,fill,inner sep=1.0pt]{};
\node at (0,2) [circle,fill,inner sep=1.0pt]{};
\node at (1,2) [circle,fill,inner sep=1.0pt]{};
\node at (2,2) [circle,fill,inner sep=1.0pt]{};
\node at (3,2) [circle,fill,inner sep=1.0pt]{};
\node at (4,2) [circle,fill,inner sep=1.0pt]{};
\node at (0,3) [circle,fill,inner sep=1.0pt]{};
\node at (1,3) [circle,fill,inner sep=1.0pt]{};
\node at (2,3) [circle,fill,inner sep=1.0pt]{};
\node at (3,3) [circle,fill,inner sep=1.0pt]{};
\node at (4,3) [circle,fill,inner sep=1.0pt]{};
\draw[color = blue] (-0.5, 0.166667) -- (4.5, -1.5);
\draw[color = red] (1.16667, 3.5) -- (-0.5, -1.5);
\draw[color = blue, dashed] (1, 3) -- (4, 2);
\draw[color = red, dashed] (3, -1) -- (4, 2);

\draw[dotted] (1, 0) -- (0.1, 0.3);
\draw[dotted] (2, 0) -- (0.2, 0.6);
\draw[dotted] (3, 0) -- (0.3, 0.9);
\draw[dotted] (1, 1) -- (0.4, 1.2);
\draw[dotted] (2, 1) -- (0.5, 1.5);
\draw[dotted] (3, 1) -- (0.6, 1.8);
\draw[dotted] (1, 2) -- (0.7, 2.1);
\draw[dotted] (2, 2) -- (0.8, 2.4);
\draw[dotted] (3, 2) -- (0.9, 2.7);
\node at (0.1,0.3) [circle,fill,inner sep=0.7pt]{};
\node at (0.2,0.6) [circle,fill,inner sep=0.7pt]{};
\node at (0.3,0.9) [circle,fill,inner sep=0.7pt]{};
\node at (0.4,1.2) [circle,fill,inner sep=0.7pt]{};
\node at (0.5,1.5) [circle,fill,inner sep=0.7pt]{};
\node at (0.6,1.8) [circle,fill,inner sep=0.7pt]{};
\node at (0.7,2.1) [circle,fill,inner sep=0.7pt]{};
\node at (0.8,2.4) [circle,fill,inner sep=0.7pt]{};
\node at (0.9,2.7) [circle,fill,inner sep=0.7pt]{};
\end{scope}

\node[yshift = -5mm] at (3, -1) {\textcolor{blue}{$W_1$}};
\node[yshift = 5mm, xshift = -3mm] at (1, 3) {\textcolor{red}{$W_2$}};
\node at (0,-1) [circle,fill,inner sep=1.0pt]{};
\node at (1,-1) [circle,fill,inner sep=1.0pt]{};
\node at (2,-1) [circle,fill,inner sep=1.0pt]{};
\node at (3,-1) [circle,fill,inner sep=1.0pt]{};
\node at (4,-1) [circle,fill,inner sep=1.0pt]{};
\node at (0,0) [circle,fill,inner sep=1.0pt]{};
\node at (1,0) [circle,fill,inner sep=1.0pt]{};
\node at (2,0) [circle,fill,inner sep=1.0pt]{};
\node at (3,0) [circle,fill,inner sep=1.0pt]{};
\node at (4,0) [circle,fill,inner sep=1.0pt]{};
\node at (0,1) [circle,fill,inner sep=1.0pt]{};
\node at (1,1) [circle,fill,inner sep=1.0pt]{};
\node at (2,1) [circle,fill,inner sep=1.0pt]{};
\node at (3,1) [circle,fill,inner sep=1.0pt]{};
\node at (4,1) [circle,fill,inner sep=1.0pt]{};
\node at (0,2) [circle,fill,inner sep=1.0pt]{};
\node at (1,2) [circle,fill,inner sep=1.0pt]{};
\node at (2,2) [circle,fill,inner sep=1.0pt]{};
\node at (3,2) [circle,fill,inner sep=1.0pt]{};
\node at (4,2) [circle,fill,inner sep=1.0pt]{};
\node at (0,3) [circle,fill,inner sep=1.0pt]{};
\node at (1,3) [circle,fill,inner sep=1.0pt]{};
\node at (2,3) [circle,fill,inner sep=1.0pt]{};
\node at (3,3) [circle,fill,inner sep=1.0pt]{};
\node at (4,3) [circle,fill,inner sep=1.0pt]{};
\draw[color = blue] (-0.5, 0.166667) -- (4.5, -1.5);
\draw[color = red] (1.16667, 3.5) -- (-0.5, -1.5);
\draw[color = blue, dashed] (1, 3) -- (4, 2);
\draw[color = red, dashed] (3, -1) -- (4, 2);

\draw[dotted] (1, 2) -- (0.3, -0.1);
\draw[dotted] (1, 1) -- (0.6, -0.2);
\draw[dotted] (1, 0) -- (0.9, -0.3);
\draw[dotted] (2, 2) -- (1.2, -0.4);
\draw[dotted] (2, 1) -- (1.5, -0.5);
\draw[dotted] (2, 0) -- (1.8, -0.6);
\draw[dotted] (3, 2) -- (2.1, -0.7);
\draw[dotted] (3, 1) -- (2.4, -0.8);
\draw[dotted] (3, 0) -- (2.7, -0.9);

\node at (0.3, -0.1) [circle,fill,inner sep=0.7pt]{};
\node at (0.6,-0.2) [circle,fill,inner sep=0.7pt]{};
\node at (0.9,-0.3) [circle,fill,inner sep=0.7pt]{};
\node at (1.2, -0.4) [circle,fill,inner sep=0.7pt]{};
\node at (1.5,-0.5) [circle,fill,inner sep=0.7pt]{};
\node at (1.8,-0.6) [circle,fill,inner sep=0.7pt]{};
\node at (2.1,-0.7) [circle,fill,inner sep=0.7pt]{};
\node at (2.4,-0.8) [circle,fill,inner sep=0.7pt]{};
\node at (2.7,-0.9) [circle,fill,inner sep=0.7pt]{};

\end{tikzpicture}
\caption{The space $\mathbb{R}^2$ decomposes into $W_1 \oplus W_2$, i.e. the blue and red lines. Lattice points in the square form a set of representatives for $\mathbb{Z}^2/\big((\mathbb{Z}^2 \cap W_1) \oplus (\mathbb{Z}^2 \cap W_2)\big)$. On the left we orthogonally project onto $W_1$ and on the right onto $W_2$. In both cases the map provides a bijection onto a set of representatives for $(\mathbb{Z}^2 \cap W_1)/P_{W_1}(\mathbb{Z}^2)$ and $(\mathbb{Z}^2 \cap W_2)/P_{W_2}(\mathbb{Z}^2)$, respectively.}
\label{fig_lattice}
\end{figure}

\begin{corollary} \label{cor_lattices}
    Suppose $V = W_1 \oplus W_2$ decomposes into a direct sum of orthogonal rational subspaces. Then
    \begin{gather*}
        \textnormal{gr}(\Lambda^1) \cdot \textnormal{gr}(\Lambda_2) = \textnormal{gr}(\Lambda_1) \cdot \textnormal{gr}(\Lambda^2) = \textnormal{gr}(\Lambda).
    \end{gather*}
\end{corollary}

\begin{proposition} \label{prop_multiplicative_discrete}
    Let $\mathfrak{q}_1$ and $\mathfrak{q}_2$ be polyhedra. We have the identity of meromorphic functions:
\begin{align*}
    S_\Lambda(\mathfrak{q}_1 \times \mathfrak{q}_2; \alpha) &= \sum_{[\gamma] \in \Lambda/(\Lambda_1 \oplus \Lambda_2)} S_{\phi_1([\gamma]) + \Lambda_1} (\mathfrak{q}_1; \alpha_1) \cdot S_{\phi_2([\gamma]) + \Lambda_2}(\mathfrak{q}_2; \alpha_2)  \\
    &= \sum_{[\gamma] \in \Lambda/(\Lambda_1 \oplus \Lambda_2)} e^{\langle \alpha, \gamma \rangle} S_{\Lambda_1}(-\gamma_1 + \mathfrak{q}_1; \alpha_1) S_{\Lambda_2}(-\gamma_2 + \mathfrak{q}_2; \alpha_2)
\end{align*}
In the second line we choose any representative $\gamma$ of $[\gamma] \in \Lambda/(\Lambda_1 \oplus \Lambda_2)$ and write $\gamma = \gamma_1 + \gamma_2$ with respect to the decomposition $V = W_1 \oplus W_2$.
\end{proposition}

\begin{proof}
    Translates of $W_2$ are parametrized by elements in $W_1$. Given $x \in W_1$, we have that $x + W_2$ intersects $\Lambda$ if and only if $x \in \Lambda^1$. In that case there exists a $\lambda \in \Lambda$ such that $\lambda = x + y$ with $y \in W_2$. The set $(x + W_2) \cap \Lambda$ is equal to $\lambda + \Lambda_2$. If we orthogonally project $\lambda + \Lambda_2$ onto $W_2$, we get $y + \Lambda_2$.

    We now compute
    \begin{align*}
        \sum_{\theta \in (x +  \mathfrak{q}_2) \cap \Lambda} e^{\langle \alpha, \theta \rangle} &= \sum_{y + \pi \in \mathfrak{q}_2 \cap (y + \Lambda_2)} e^{\langle \alpha, x + y + \pi \rangle} \\
        &= e^{\langle \alpha_1, x \rangle} \sum_{\pi \in (-y + \mathfrak{q}_2) \cap \Lambda_2}  e^{\langle \alpha_2, y + \pi \rangle} \\
        &= e^{\langle \alpha_1, x \rangle} e^{\langle \alpha_2, y \rangle} S_{\Lambda_2}(-y + \mathfrak{q}_2; \alpha_2).
    \end{align*}
    We now need to sum up over $(x + \Lambda_1) \cap \mathfrak{q}_1$. We then get
    \begin{align*}
        e^{\langle \alpha_2, y \rangle} S_{\Lambda_2}(-y + \mathfrak{q}_2; \alpha_2) \sum_{x + \pi \in (x + \Lambda_1) \cap \mathfrak{q}_1} e^{\langle \alpha_1, x + \pi \rangle} &= e^{\langle \alpha_2, y \rangle} S_{\Lambda_2}(-y + \mathfrak{q}_2; \alpha_2) e^{\langle \alpha_1, x \rangle} \sum_{\theta \in (-x + \mathfrak{q}_1) \cap \Lambda_1} e^{\langle \alpha_1, \theta \rangle} \\
        &= e^{\langle \alpha, \lambda \rangle} S_{\Lambda_2}(-y + \mathfrak{q}_2; \alpha_2) S_{\Lambda_1}(-x + \mathfrak{q}_1; \alpha_1) \\
        &= S_{y + \Lambda_2}(\mathfrak{q}_2; \alpha_2) S_{x + \Lambda_1}(\mathfrak{q}_1; \alpha_1)
    \end{align*}
    We then sum up over a collection of representatives of $\Lambda/(\Lambda_1 \oplus \Lambda_2)$.
\end{proof}

\subsection{Degenerate Brion's formula: discrete setting}

We can now prove the various versions of degenerate Brion's formula in the discrete setting discussed in the introduction. Propositions \ref{prop_proto_brion_discrete} and \ref{prop_multiplicative_discrete} immediately imply Corollary \ref{cor_discrete_1}. The Berline-Vergne local Euler-Maclaurin formula allows us to express
\begin{gather}
    S_{\Gamma}(\mathfrak{f}; 0) = \sum_{\mathfrak{g} \in \textnormal{Face}(\mathfrak{f})} \mu_{\mathfrak{g}^\perp}^{\Gamma^{\mathfrak{g}^\perp}}(\mathfrak{t}_\mathfrak{g}^{\mathfrak{f}}; 0) \cdot \textnormal{vol}^{\Gamma_{\mathfrak{g}}}(\mathfrak{g}). \label{lattice_point_vol}
\end{gather}
When we expand the terms in \eqref{eqn_brion_1} using \eqref{lattice_point_vol} and then collect terms, we obtain Corollary \ref{thm_intro_degen_brion_2}.

The proof of Corollary \ref{thm_intro_degen_brion_3} is a bit more involved. We obtain it by applying the local Euler-Maclaurin formula directly to $I(\mathfrak{p}; \xi)$.

\begin{proof}[Proof of Corollary \ref{thm_intro_degen_brion_3}]
    We begin by discussing our notation for different normalizations of Lebesgue measures on rational subspaces of $V$. The volume form on $V$ coming from $\langle \cdot, \cdot \rangle$ is denoted by $\textnormal{vol}$. The volume form coming from the restriction of $\langle \cdot, \cdot \rangle$ to a subspace parallel to a face $\mathfrak{g}$ is denoted $\textnormal{vol}_{\textnormal{lin}(\mathfrak{g}), \langle \cdot, \cdot \rangle}$. The volume form on a subspace parallel to a face $\mathfrak{g}$ normalized so that a lattice $\Gamma$ in $\textnormal{lin}(\mathfrak{g})$ has covolume 1 is denoted by $\textnormal{vol}^{\Gamma}$. If $\mathfrak{g} \subseteq \mathfrak{f}$ are faces and $\Gamma$ is a lattice in $\textnormal{lin}(\mathfrak{f})$, we let $\textnormal{vol}^{\Gamma}_{\textnormal{lin}(\mathfrak{g}), \langle \cdot, \cdot \rangle}$ denote the Lebesgue measure on $\textnormal{lin}(\mathfrak{g})$ obtained by first taking the restriction of the inner product to $\textnormal{lin}(\mathfrak{f})$, rescaling it so that $\Gamma$ has covolume 1, and then restricting the rescaled inner product to the subspace $\textnormal{lin}(\mathfrak{f})$. We use analogous notation $I_{\textnormal{lin}(\mathfrak{g}), \langle \cdot, \cdot \rangle}(\cdot; \cdot)$, $I^{\Gamma}(\cdot; \cdot)$, and $I^\Gamma_{\textnormal{lin}(\mathfrak{g}), \langle \cdot, \cdot \rangle} (\cdot; \cdot)$ for the usual integral defining $I(\cdot; \cdot)$ but with respect to the measures $\textnormal{vol}_{\textnormal{lin}(\mathfrak{g}), \langle \cdot, \cdot \rangle}$, $\textnormal{vol}^{\Gamma}$ and $\textnormal{vol}^{\Gamma}_{\textnormal{lin}(\mathfrak{f}), \langle \cdot, \cdot \rangle}$, respectively.

    The remainder of the proof is primarily a matter of bookkeeping. Using Theorem \ref{thm_berline_vergne} (local Euler-Maclaurin formula) and Corollary \ref{thm_degen_brion_cont_intro} (degenerate Brion's formula: continuous setting), we have:
    \begin{align*}
        S_{\Lambda}(\mathfrak{p}; \xi) &= \sum_{\mathfrak{f} \in \textnormal{Face}(\mathfrak{p})} \mu_{\textnormal{lin}(\mathfrak{f})^\perp}^{\Lambda^{\mathfrak{f}^\perp}}(\mathfrak{t}_{\mathfrak{f}}^{\mathfrak{p}}; \xi) \cdot I^{\Lambda_{\mathfrak{f}}}(\mathfrak{f}; \xi) \\
        &= \sum_{\mathfrak{f} \in \textnormal{Face}(\mathfrak{p})} \mu_{\textnormal{lin}(\mathfrak{f})^\perp}^{\Lambda^{\mathfrak{f}^\perp}}(\mathfrak{t}_{\mathfrak{f}}^{\mathfrak{p}}; \xi) \cdot \sum_{\mathfrak{g} \in \{\mathfrak{f}\}_\xi} \textnormal{vol}^{\Lambda_{\mathfrak{f}}}_{\textnormal{lin}(\mathfrak{g}), \langle \cdot, \cdot \rangle}(\mathfrak{g}) \cdot I^{\Lambda_{\mathfrak{f}}}_{\textnormal{lin}(\mathfrak{f}) \cap \textnormal{lin}(\mathfrak{g})^\perp, \langle \cdot, \cdot \rangle} (\textnormal{LC}_{\mathfrak{g}}^{\mathfrak{f}}(\xi); \xi) \\
        &= \sum_{\mathfrak{g} \in \{\mathfrak{p}\}_\xi} \ \ \ \sum_{\mathfrak{f} \supseteq \mathfrak{g} \textnormal{ s.t. } \mathfrak{f} = \mathfrak{g} \textnormal{ or } \mathfrak{f} \notin \{\mathfrak{p}\}_\xi}
        \mu_{\textnormal{lin}(\mathfrak{f})^\perp}^{\Lambda^{\mathfrak{f}^\perp}}(\mathfrak{t}_{\mathfrak{f}}^{\mathfrak{p}}; \xi) \cdot \textnormal{vol}^{\Lambda_{\mathfrak{f}}}_{\textnormal{lin}(\mathfrak{g}), \langle \cdot, \cdot \rangle}(\mathfrak{g}) \cdot I^{\Lambda_{\mathfrak{f}}}_{\textnormal{lin}(\mathfrak{f}) \cap \textnormal{lin}(\mathfrak{g})^\perp, \langle \cdot, \cdot \rangle}(\textnormal{LC}_{\mathfrak{g}}^{\mathfrak{f}}(\xi); \xi).
    \end{align*}

Note that we have by \eqref{eqn_renormalize_integral}:
    \begin{align*}
        \textnormal{vol}_{\textnormal{lin}(\mathfrak{g}), \langle \cdot, \cdot \rangle}^{\Lambda_{\mathfrak{f}}}(\mathfrak{g}) \cdot I_{\textnormal{lin}(\mathfrak{f}) \cap \textnormal{lin}(\mathfrak{g})^\perp, \langle \cdot, \cdot \rangle}^{\Lambda_{\mathfrak{f}}}(\textnormal{LC}_{\mathfrak{g}}^{\mathfrak{f}}(\xi); \xi) &= \textnormal{gr}(\Lambda_{\mathfrak{f}})^{-1} \cdot \textnormal{vol}_{\textnormal{lin}(\mathfrak{g}), \langle \cdot, \cdot \rangle}(\mathfrak{g}) \cdot I_{\textnormal{lin}(\mathfrak{f}) \cap \textnormal{lin}(\mathfrak{g})^\perp, \langle \cdot, \cdot \rangle}(\textnormal{LC}_{\mathfrak{g}}^{\mathfrak{f}}(\xi); \xi) \\       \textnormal{vol}_{\textnormal{lin}(\mathfrak{g}), \langle \cdot, \cdot \rangle}(\mathfrak{g}) &= \textnormal{gr}(\Lambda_{\mathfrak{g}}) \cdot \textnormal{vol}^{\Lambda_{\mathfrak{g}}}(\mathfrak{g}) \\
        I_{\textnormal{lin}(\mathfrak{f}) \cap \textnormal{lin}(\mathfrak{g})^\perp, \langle \cdot, \cdot \rangle} (\textnormal{LC}_{\mathfrak{g}}^{\mathfrak{f}}(\xi); \xi) &= \textnormal{gr}\big((\Lambda^{\mathfrak{g}^\perp})_{\mathfrak{f}} \big) \cdot I^{(\Lambda^{\mathfrak{g}^\perp})_{\mathfrak{f}}}(\textnormal{LC}_{\mathfrak{g}}^{\mathfrak{f}}(\xi); \xi).
    \end{align*}
    Since $\mathfrak{g} \subseteq \mathfrak{f}$, we have that $\textnormal{lin}(\mathfrak{g})^\perp \supseteq \textnormal{lin}(\mathfrak{f})^\perp$. We also have the orthogonal decompositions $V = \textnormal{lin}(\mathfrak{f}) \oplus \textnormal{lin}(\mathfrak{f})^\perp$, $V = \textnormal{lin}(\mathfrak{g}) \oplus \textnormal{lin}(\mathfrak{g})^\perp$, and $\textnormal{lin}(\mathfrak{g})^\perp = \textnormal{lin}(\mathfrak{f})^\perp \oplus (\textnormal{lin}(\mathfrak{f}) \cap \textnormal{lin}(\mathfrak{g})^\perp)$. Using Corollary \ref{cor_lattices} these imply:
    \begin{align*}
        \textnormal{gr}(\Lambda_{\mathfrak{f}}) \cdot \textnormal{gr}(\Lambda^{\mathfrak{f}^\perp}) &= \textnormal{gr}(\Lambda) \\
        \textnormal{gr}(\Lambda_{\mathfrak{g}}) \cdot \textnormal{gr}(\Lambda^{\mathfrak{g}^\perp})& = \textnormal{gr}(\Lambda) \\
        \textnormal{gr}\big( (\Lambda^{\mathfrak{g}^\perp})_\mathfrak{f} \big) \cdot \textnormal{gr} \big( (\Lambda^{\mathfrak{g}^\perp})_{\mathfrak{f}^\perp} \big) &= \textnormal{gr}(\Lambda^{\mathfrak{g}^\perp}) \\
        (\Lambda^{\mathfrak{g}^\perp})^{\mathfrak{f}^\perp} &= \Lambda^{\mathfrak{f}^\perp}.
    \end{align*}
    We therefore get that
    \begin{align*}
        \frac{\textnormal{gr}(\Lambda_{\mathfrak{g}}) \cdot \textnormal{gr}\big( (\Lambda^{\mathfrak{g}^\perp})_\mathfrak{f} \big)}{\textnormal{gr}(\Lambda_{\mathfrak{f}})} &= \frac{\textnormal{gr}(\Lambda_{\mathfrak{g}}) \cdot \textnormal{gr}(\Lambda^{\mathfrak{g}^\perp})}{\textnormal{gr}(\Lambda_{\mathfrak{f}}) \cdot \textnormal{gr}\big( (\Lambda^{\mathfrak{g}^\perp})^{\mathfrak{f}^\perp} \big)} \\
        &= \frac{\textnormal{gr}(\Lambda)}{\textnormal{gr}(\Lambda_{\mathfrak{f}}) \cdot \textnormal{gr}(\Lambda^{\mathfrak{f}^\perp})} \\
        &= 1.
    \end{align*}
    We then finally get that
    \begin{align*}
        \textnormal{vol}_{\textnormal{lin}(\mathfrak{g}), \langle \cdot, \cdot \rangle}(\mathfrak{g}) \cdot I^{\Lambda_\mathfrak{f}}_{\textnormal{lin}(\mathfrak{f}) \cap \textnormal{lin}(\mathfrak{g})^\perp, \langle \cdot, \cdot \rangle}(\textnormal{LC}_{\mathfrak{g}}^{\mathfrak{f}}(\xi); \xi) &= \frac{\textnormal{gr}(\Lambda_{\mathfrak{g}}) \cdot \textnormal{gr}((\Lambda^{\mathfrak{g}^\perp})_{\mathfrak{f}})}{\textnormal{gr}(\Lambda_{\mathfrak{f}})} \cdot \textnormal{vol}^{\Lambda_{\mathfrak{g}}}(\mathfrak{g}) \cdot I^{(\Lambda^{\mathfrak{g}^\perp})_{\mathfrak{f}}}(\textnormal{LC}_{\mathfrak{g}}^{\mathfrak{f}}(\xi); \xi) \\
        &= \textnormal{vol}^{\Lambda_{\mathfrak{g}}}(\mathfrak{g}) \cdot I^{(\Lambda^{\mathfrak{g}^\perp})_{\mathfrak{f}}}(\textnormal{LC}_{\mathfrak{g}}^{\mathfrak{f}}(\xi); \xi)
    \end{align*}
    This proves \eqref{eqn_intro_degen_brion_3}.
\end{proof}

Suppose $\mathfrak{g}$ is $\xi$-maximal in $\mathfrak{p}$. This implies that it is $\xi$-maximal in any face $\mathfrak{f} \supseteq \mathfrak{g}$. Thus the coefficient of $\textnormal{vol}^{\Lambda_{\mathfrak{g}}}(\mathfrak{g})$ in \eqref{eqn_intro_degen_brion_3} is
\begin{gather}
    \sum_{\mathfrak{f} \supseteq \mathfrak{g}} \mu_{\textnormal{lin}(\mathfrak{f})^\perp}^{\Lambda^{\mathfrak{f}^\perp}}(\mathfrak{t}_{\mathfrak{f}}^{\mathfrak{p}}; \xi) \cdot I^{(\Lambda^{\mathfrak{g}^\perp})_\mathfrak{f}}(\mathfrak{t}_{\mathfrak{g}}^{\mathfrak{f}}; \xi) = \sum_{\mathfrak{f} \supseteq \mathfrak{g}} \mu_{\textnormal{lin}(\mathfrak{f})^\perp}^{(\Lambda^{\mathfrak{g}^\perp})^{\mathfrak{f}^\perp}}(\mathfrak{t}_{\mathfrak{f}}^{\mathfrak{p}}; \xi) \cdot I^{(\Lambda^{\mathfrak{g}^\perp})_\mathfrak{f}}(\mathfrak{t}_{\mathfrak{g}}^{\mathfrak{f}}; \xi). \label{eqn_degen_brion_xi_max}
\end{gather}
We recognize that $\mathfrak{t}_{\mathfrak{g}}^{\mathfrak{f}}$ are exactly the faces of the cone $\mathfrak{t}_{\mathfrak{g}}^{\mathfrak{p}}$. Thus \eqref{eqn_degen_brion_xi_max} is exactly equal to $S_{\Lambda^{\mathfrak{g}^\perp}}(\mathfrak{t}_{\mathfrak{g}}^{\mathfrak{p}}; \xi)$ by Theorem \ref{thm_berline_vergne}. Of course this formula is ultimately equivalent to \eqref{eqn_intro_degen_brion_2} from which we already showed this.

Now suppose we set $\xi = 0$. Then $\textnormal{LC}_{\mathfrak{g}}^{\mathfrak{f}}(0) = \emptyset$ unless $\mathfrak{f} = \mathfrak{g}$ in which case we get that $I^{(\Lambda^{\mathfrak{g}^\perp})_{\mathfrak{f}}} (\textnormal{LC}_{\mathfrak{g}}^{\mathfrak{f}}(\xi); \xi) = 1$. Therefore the formula turns into
\begin{gather*}
    S_{\Lambda}(\mathfrak{p}; 0) = \sum_{\mathfrak{g} \in \textnormal{Face}(\mathfrak{p})} \textnormal{vol}^{\Lambda_{\mathfrak{g}}}(\mathfrak{g}) \cdot \mu_{\textnormal{lin}(\mathfrak{g})^\perp}^{\Lambda^{\mathfrak{g}^\perp}}(\mathfrak{t}_{\mathfrak{g}}^{\mathfrak{p}}; 0),
\end{gather*}
just like we got before. One can show using an adapted version of the argument used to prove Theorem \ref{thm_levi_cone_euler_maclaurin} that the corresponding terms in \eqref{eqn_intro_degen_brion_2} and \eqref{eqn_intro_degen_brion_3} are in fact equal in general.

\section{Laplace's method and the method of stationary phase}

We now explain how our degenerate Brion's formula is analogous to Laplace's method and the method of stationary phase.

\subsection{Laplace's method}
Laplace's method is concerned with the behavior of integrals of the form $\int_{M} g(x) e^{t \cdot f(x)} dx$ as $t \to \infty$ where $M$ is a manifold and $f(x)$ and $g(x)$ are real-valued functions. Laplace's method states the following:
\begin{proposition}[Laplace's method] \label{prop_laplace}
    Suppose $M$ is a compact $n$-dimensional Riemannian manifold. Suppose $f(x)$ is a $C^2$ real-valued function on $M$, and $g(x)$ is a continuous function on $M$. Suppose $x_0 \in M$ is such that $f(x_0) > f(x)$ for all other $x \in M$. Suppose the Hessian of $f(x)$ at $x_0$ is negative definite and that $g(x_0) > 0$. Then as $t \to \infty$,
    \begin{gather*}
        \int_{M} g(x) e^{t \cdot f(x)} dx \to (2 \pi)^{n/2} \frac{1}{\sqrt{|\det \textnormal{Hess}_{x_0}(f)|}} \cdot g(x_0) \cdot \frac{1}{t^{n/2}} e^{t \cdot f(x_0)}.
    \end{gather*}
\end{proposition}

Now suppose $\mathcal{C}$ is a compact strictly convex $n$-dimensional set in $V$ with smooth boundary. Suppose $\xi \in V^*$ is a real functional. We compute the asymptotics as $t \to \infty$ of $\int_{t \cdot \mathcal{C}} e^{\langle \xi, x \rangle} dx$. We shall use Stokes' theorem in a similar way to how Proposition \ref{prop_stokes} is proven. Identifying $V$ with $V^*$, let $u_1, \dots, u_{n-1}$ be such that $u_1, \dots, u_{n-1}, \hat{\xi}$ is an orthonormal basis for $V$. Consider the $(n-1)$-form defined by
\begin{gather*}
    \omega(x) = \frac{1}{\|\xi\|} e^{\langle \xi, x \rangle} u_1 \wedge \dots \wedge u_{n-1}.
\end{gather*}
Then we have
\begin{gather*}
    d \omega(x) = e^{\langle \xi, x \rangle} d \textnormal{vol}.
\end{gather*}
By Stokes' theorem we have
\begin{gather*}
    \int_{\mathcal{C}} e^{\langle \xi, x \rangle} d \textnormal{vol} = \frac{1}{\|\xi\|} \int_{\partial \mathcal{C}} e^{\langle \xi, x \rangle} \cdot \langle \hat{\xi}, \eta_x \rangle dx,
\end{gather*}
where $\eta_x$ is the outward pointing normal at $x$. Let $x_0$ be the unique point on $\partial \mathcal{C}$ maximizing $\langle \xi, x \rangle$. At such a point we necessarily have that $\eta_{x_0} = \hat{\xi}$.

Therefore using Proposition \ref{prop_laplace} we have
\begin{align*}
    \int_{t \cdot \mathcal{C}} e^{\langle \xi, x \rangle} d \textnormal{vol} &= \frac{t^{n-1}}{\|\xi\|} \int_{\partial \mathcal{C}} e^{t \cdot \langle \xi, x \rangle} \cdot \langle \hat{\xi}, \eta_x \rangle dx \to \frac{(2 \pi)^{\frac{(n-1)}{2}}}{\|\xi\| \sqrt{|\det \textnormal{Hess}_{x_0}(e^{\langle \xi, x \rangle})|}} \cdot t^{\frac{(n-1)}{2}} \cdot e^{t \cdot \langle \xi, x_0 \rangle}.
\end{align*}
Furthermore, we have that
\begin{gather*}
    |\det \textnormal{Hess}_{x_0}(e^{\langle \xi, x \rangle})| = \|\xi\|^{n-1} \cdot \kappa(x_0),
\end{gather*}
where $\kappa(x_0)$ is the Gaussian curvature at $x_0$ of the hypersurface $\partial \mathcal{C}$.
Thus we get that as $t \to \infty$:
\begin{gather}
    \int_{t \cdot \mathcal{C}} e^{\langle \xi, x \rangle} d \textnormal{vol} \to \frac{(2 \pi)^{\frac{(n-1)}{2}} }{\|\xi\|^{\frac{(n+1)}{2}} \sqrt{\kappa(x_0)}} \cdot t^{\frac{(n-1)}{2}} \cdot e^{t \langle \xi, x_0 \rangle}. \label{eqn_laplaces_method}
\end{gather}

Now suppose that $\mathfrak{p}$ is a polytope in $V$ and $\xi \in V^*$. We get that
\begin{gather*}
    \int_{t \cdot \mathfrak{p}} e^{\langle \xi, x \rangle} d \textnormal{vol} = \sum_{\mathfrak{g} \in \{\mathfrak{p}\}_\xi} \textnormal{vol}(\mathfrak{g}) \cdot I({}^0 \textnormal{LC}_{\mathfrak{g}}^{\mathfrak{p}}(\xi); \xi) \cdot t^{\textnormal{dim}(\mathfrak{g})} \cdot e^{t \cdot \langle \xi, \mathfrak{g} \rangle}.
\end{gather*}
Let $\mathfrak{g}_0$ be the unique $\xi$-maximal face in $\{\mathfrak{p}\}_\xi$ on which $\langle \xi, \cdot \rangle$ is maximized. We then get that
\begin{gather}
    \int_{t \cdot \mathfrak{p}} e^{\langle \xi, x \rangle} d \textnormal{vol} = \textnormal{vol}(\mathfrak{g}_0) \cdot I({}^0 \mathfrak{t}_{\mathfrak{g}_0}^{\mathfrak{p}}; \xi) \cdot t^{\textnormal{dim}(\mathfrak{g}_0)} \cdot e^{t \langle \xi, \mathfrak{g}_0 \rangle} + O_{t \to \infty}(t^{\textnormal{dim}(\mathfrak{g}_0) - 1} \cdot e^{t \cdot \langle \xi, \mathfrak{g}_0 \rangle}). \label{eqn_polytope_laplace_continuous}
\end{gather}

We now compare \eqref{eqn_laplaces_method} and \eqref{eqn_polytope_laplace_continuous}. In \eqref{eqn_laplaces_method} the value grows like $t^{\frac{n-1}{2}} \cdot e^{t \langle \xi, x_0 \rangle}$, i.e. polynomial times exponential. Furthermore the component $\frac{(2 \pi)^{\frac{(n-1)}{2}}}{\sqrt{\kappa(x_0)}}$ can be interpreted as relating to the size of slices of $\mathcal{C}$ near $x_0$ parallel to the hyperplane $\langle \xi, \cdot \rangle = 0$ (such slices look like ellipsoids). The term $\frac{1}{\|\xi\|^{\frac{(n+1)}{2}}}$ is of course homogeneous in $\xi$ of degree $-\frac{(n+1)}{2} = \frac{(n-1)}{2} - n$. Similarly in \eqref{eqn_polytope_laplace_continuous} the value grows like $t^{\textnormal{dim}(\mathfrak{g}_0)} \cdot e^{t \langle \xi, \mathfrak{g}_0 \rangle}$. The term $\textnormal{vol}(\mathfrak{g}_0)$ is clearly related to the size of slices of $\mathfrak{p}$ near $\mathfrak{g}_0$ parallel to the hyperplane $\langle \xi, \cdot \rangle = 0$. Furthermore $I({}^0 \mathfrak{t}_{\mathfrak{g}_0}^{\mathfrak{p}}; \xi)$ is homogeneous in $\xi$ of degree $\textnormal{dim}(\mathfrak{g}_0) - n$.

Now suppose $\mathfrak{p}$ is a rational polytope with respect to a lattice $\Lambda$. Let $\xi \in V^*$ be a real functional; then we necessarily have $\xi \in (V^*_{\mathbb{C}})^\Lambda$. We thus get that
\begin{align*}
    \sum_{\lambda \in t \cdot \mathfrak{p} \cap \Lambda} e^{\langle \xi, \lambda \rangle} = & \sum_{\mathfrak{g} \in \{\mathfrak{p}\}_\xi} \textnormal{vol}^{\Gamma_{\mathfrak{g}}}(\mathfrak{g}) \cdot t^{\textnormal{dim}(\mathfrak{g})} e^{t \cdot \langle \xi, \mathfrak{g} \rangle} \\
    \times & \Bigg( \sum_{\mathfrak{f} \supseteq \mathfrak{g}, \ \mathfrak{f} = \mathfrak{g} \textnormal{ or } \mathfrak{f} \notin \{\mathfrak{p}\}_\xi} \mu_{\textnormal{lin}(\mathfrak{f})^\perp}^{\Lambda^{\mathfrak{f}^\perp}}(t \cdot \mathfrak{f}^{\mathfrak{f}^\perp} + {}^0 \mathfrak{t}_{\mathfrak{f}}^{\mathfrak{p}}; 0) \cdot I^{(\Lambda^{\mathfrak{g}^\perp})_{\mathfrak{f}}} ({}^0 \textnormal{LC}_{\mathfrak{g}}^{\mathfrak{f}}(\xi); \xi) \Bigg).
\end{align*}
Let $\mathfrak{g}_0$ be the unique $\xi$-maximal face in $\{\mathfrak{p}\}_\xi$ on which $\langle \xi, \cdot \rangle$ is maximized. We then get that
\begin{align}
    \sum_{\lambda \in t \cdot \mathfrak{p} \cap \Lambda} e^{\langle \xi, \lambda \rangle} &= \textnormal{vol}^{\Lambda_{\mathfrak{g}_0}}(\mathfrak{g}_0) \cdot t^{\textnormal{dim}(\mathfrak{g}_0)} \cdot S_{\Lambda^{\mathfrak{g}_0^\perp}}(t \cdot \mathfrak{t}_{\mathfrak{g}_0}^{\mathfrak{p}}; \xi) + O_{t \to \infty}(t^{\textnormal{dim}(\mathfrak{g}_0) - 1} e^{t \cdot \langle \xi, \mathfrak{g}_0 \rangle}) \notag \\
    &= \textnormal{vol}^{\Lambda_{\mathfrak{g}_0}}(\mathfrak{g}_0) \cdot t^{\textnormal{dim}(\mathfrak{g}_0)} \cdot e^{t \cdot \langle \xi, \mathfrak{g}_0 \rangle} \cdot \Big( \sum_{\mathfrak{f} \supseteq \mathfrak{g}_0} \mu^{\Lambda^{\mathfrak{f}^\perp}}_{\textnormal{lin}(\mathfrak{f})^\perp}(t \cdot \mathfrak{f}^{\mathfrak{f}^\perp} + \mathfrak{t}_{\mathfrak{f}}^{\mathfrak{p}}; 0) \cdot I^{(\Lambda^{\mathfrak{g}^\perp})_{\mathfrak{f}}}({}^0 \mathfrak{t}_{\mathfrak{g}_0}^{\mathfrak{f}}; \xi) \Big) \label{eqn_laplace_brion} \\
    & + O_{t \to \infty}(t^{\textnormal{dim}(\mathfrak{g}_0) - 1} e^{t \cdot \langle \xi, \mathfrak{g}_0 \rangle}). \notag
\end{align}
The term in parentheses in \eqref{eqn_laplace_brion} is of course periodic in $t$. In case $\mathfrak{g}_0$ contains a lattice point, then the expression simplifies to
\begin{gather*}
    \sum_{\lambda \in t \cdot \mathfrak{p} \cap \Lambda} e^{\langle \xi, \lambda \rangle} = \textnormal{vol}^{\Lambda_{\mathfrak{g}_0}}(\mathfrak{g}_0) \cdot S_{\Lambda^{\mathfrak{g}_0^\perp}}({}^0 \mathfrak{t}_{\mathfrak{g}_0}^{\mathfrak{p}}; \xi) \cdot
    t^{\textnormal{dim}(\mathfrak{g}_0)} \cdot e^{t \cdot \langle \xi, \mathfrak{g}_0 \rangle} + O_{t \to \infty}(t^{\textnormal{dim}(\mathfrak{g}_0) - 1} e^{t \cdot \langle \xi, \mathfrak{g}_0 \rangle}).
\end{gather*}
This is completely analogous to \eqref{eqn_polytope_laplace_continuous}.

\subsection{Method of stationary phase}

The method of stationary phase is concerned with the behavior of integrals of the form $\int_M g(x) e^{i \cdot t \cdot f(x)} dx$ as $t \to \infty$ where $M$ is a manifold and $f(x)$ is a real-valued function. The method of stationary phase states the following:

\begin{proposition} [method of stationary phase]
    Suppose $M$ is a compact $n$-dimensional Riemannian manifold. Suppose $f(x)$ is a $C^2$ real-valued function on $M$ and $g(x)$ a continuous function. Let $x_1, \dots, x_k$ be the points on $M$ such that $\nabla f = 0$. Suppose $\textnormal{Hess}_{x_j}(f)$ is nondegenerate for all $j$. Then as $t \to \infty$,
    \begin{gather*}
        \int_{M} g(x) e^{i \cdot t \cdot f(x)} dx \to \sum_{j} (2 \pi)^{n/2} \frac{1}{\sqrt{|\det \textnormal{Hess}_{x_j}(f)|}} \cdot e^{\frac{i \pi}{4} \textnormal{sgn}(\textnormal{Hess}_{x_j}(f))} \cdot g(x_j) \cdot \frac{1}{t^{n/2}} \cdot e^{i \cdot t \cdot f(x_j)},
    \end{gather*}
    where $\textnormal{sgn}(\textnormal{Hess}_{x_j}(f))$ is the number of positive eigenvalues of $\textnormal{Hess}_{x_j}(f)$ minus the number of negative eigenvalues.
\end{proposition}

As before, suppose now $\mathcal{C}$ is a compact strictly convex $n$-dimensional set in $V$ with smooth boundary. Suppose $\xi \in V^*$ is a real functional. We compute the asymptotics as $t \to \infty$ of $\int_{t \cdot \mathcal{C}} e^{i \cdot \langle \xi, x \rangle} dx$. Let $u_1, \dots, u_{n-1}$ be such that $u_1, \dots, u_{n-1}, \hat{\xi}$ is an orthonormal basis for $V$. Consider the $(n-1)$-form defined by
\begin{gather*}
    \omega(x) = \frac{1}{i \|\xi\|} e^{\langle i \cdot \xi, x \rangle} u_1 \wedge \dots \wedge u_{n-1}.
\end{gather*}
We then have
\begin{gather*}
    d \omega(x) = e^{i \langle \xi, x \rangle} d \textnormal{vol}.
\end{gather*}
By Stokes' theorem we have
\begin{gather*}
    \int_{C} e^{i \langle \xi, x \rangle} d \textnormal{vol} = \frac{1}{i \|\xi\|} \int_{\partial \mathcal{C}} e^{i \langle \xi, x \rangle} \cdot \langle \hat{\xi}, \eta_x \rangle dx.
\end{gather*}
Let $x_1$ be the unique point on $\partial \mathcal{C}$ maximizing $\langle \xi, \cdot \rangle$, and let $x_2$ be the unique point minimizing it. These are the only stationary points. Then $\eta_{x_1} = \hat{\xi}$, and $\eta_{x_2} = -\hat{\xi}$. The Hessian at $x_1$ is negative definite, and at $x_2$ is positive definite. We therefore get that
\begin{align*}
    \int_{t \cdot \mathcal{C}} e^{i \langle \xi, x \rangle} dx & = \frac{t^{n-1}}{i \|\xi\|} \int_{\partial \mathcal{C}} e^{i \cdot t \cdot \langle \xi, x \rangle} \cdot \langle \hat{\xi}, \eta_x \rangle dx \\
    & \to \frac{(2 \pi)^{\frac{n-1}{2}}}{i \|\xi\|^{\frac{n+1}{2}}} \Big( \frac{e^{-\frac{i \pi (n-1)}{4}} \cdot e^{i \cdot t \cdot \langle \xi, x_1 \rangle}}{\sqrt{\kappa(x_1)}} - \frac{e^{\frac{i \pi (n-1)}{4}} \cdot e^{i \cdot t \cdot \langle \xi, x_2 \rangle}}{\sqrt{\kappa(x_2)}} \Big) \cdot t^{\frac{n-1}{2}}.
\end{align*}

Now suppose that $\mathfrak{p}$ is a polytope in $V$ and $\xi \in V^*$. Let $\mathfrak{g}_1, \dots, \mathfrak{g}_r$ be the faces of maximal dimension in $\{\mathfrak{p}\}_\xi$. All such faces must be $\xi$-maximal. Suppose the dimension of all of these faces is $k$. Then
\begin{gather*}
    \int_{t \cdot \mathfrak{p}} e^{i \langle \xi, x \rangle} d \textnormal{vol} = \big( \sum_j \textnormal{vol}(\mathfrak{g}_j) \cdot I({}^0 \mathfrak{t}_{\mathfrak{g}_j}^{\mathfrak{p}}; i \cdot \xi) \cdot e^{i \cdot t \cdot \langle \xi, \mathfrak{g}_j \rangle} \Big) \cdot t^k + O_{t \to \infty}(t^{k-1}).
\end{gather*}

Suppose $\mathfrak{p}$ is a rational polytope with respect to the lattice $\Lambda$. Suppose $\xi \in i V^*$ and $\xi \in (V^*_{\mathbb{C}})^\Lambda$. Let $\mathfrak{g}_1, \dots \mathfrak{g}_r$ be the $\xi$-constant faces of maximal dimension. All such faces must be $\xi$-maximal. Let $k$ be the dimension of these faces. We then get
\begin{gather*}
    \sum_{\lambda \in t \cdot \mathfrak{p} \cap \Lambda} e^{ \langle \xi, \lambda \rangle} = \Big(\sum_{j} \textnormal{vol}^{\Lambda_{\mathfrak{g}_j}}(\mathfrak{g}_j) \cdot S_{\Lambda^{\mathfrak{g}_j^\perp}}(t \cdot \mathfrak{t}_{\mathfrak{g}_j}^{\mathfrak{p}}; \xi) \Big) \cdot t^k + O_{t \to \infty}(t^{k-1}). 
\end{gather*}
If all of the faces $\mathfrak{g}_j$ contain a lattice point, then this simplifies to
\begin{gather}
    \sum_{\lambda \in t \cdot \mathfrak{p} \cap \Lambda} e^{\langle \xi, \lambda \rangle} = \Big(\sum_{j} \textnormal{vol}^{\Lambda_{\mathfrak{g}_j}}(\mathfrak{g}_j) \cdot S_{\Lambda^{\mathfrak{g}_j^\perp}}({}^0 \mathfrak{t}_{\mathfrak{g}_j}^{\mathfrak{p}}; \xi) \cdot e^{\cdot t \cdot \langle \xi, \mathfrak{g}_j \rangle} \Big) \cdot t^k + O_{t \to \infty}(t^{k-1}). \label{eqn_stationary_phase_brion}
\end{gather}
If, for example, one of the $\mathfrak{t}^{\mathfrak{p}}_{\mathfrak{g}_j}$ is unimodular with respect to $\Lambda^{\mathfrak{g}_j^\perp}$, it implies that the coefficient of $t^k$ in \eqref{eqn_stationary_phase_brion} is not identically zero as $t$ varies. We suspect that in general the coefficient of $t^k$ is never identically zero if $\xi \in (V^*_{\mathbb{C}})^\Lambda$, though we do not see an immediate explanation for this.

We can use \eqref{eqn_reduce_to_adapted} to get that for general $\xi \in i V^*$:
\begin{gather*}
    \sum_{\lambda \in t \cdot \mathfrak{p} \cap \Lambda} e^{\langle \xi, \lambda \rangle} = \Big( \sum_j \textnormal{vol}^{\tilde{\Lambda}_{\mathfrak{g}_j}}(\mathfrak{g}_j) \big( \sum_{[\gamma] \in \Lambda/\tilde{\Lambda}} e^{\langle \xi, \gamma^{\tilde{V}} \rangle} S_{[\gamma^{\mathfrak{g}_j^\perp}] + (\tilde{\Lambda})^{\mathfrak{g}_j^\perp}}(t \cdot \mathfrak{t}_{\mathfrak{g}_j}^{\mathfrak{p}}; \tilde{\xi}) \big) \Big) \cdot t^k + O_{t \to \infty}(t^{k-1}).
\end{gather*}
However, in this case it is certainly possible that the coefficient of $t^k$ is identically zero. For example, suppose that $\xi \in \pi i \cdot V^*_{\mathbb{Q}}$. Then $\tilde{\xi} = 0$. Our formulas then express $S_{\Lambda}(t \cdot \mathfrak{p}; \xi)$ as
\begin{gather*}
    S_{\Lambda}(t \cdot \mathfrak{p}; \xi) = \sum_{[\gamma] \in \Lambda/\tilde{\Lambda}} e^{\langle \xi, \gamma \rangle} S_{[\gamma] + \tilde{\Lambda}}(t \cdot \mathfrak{p}; 0) = \sum_{[\gamma] \in \Lambda/\tilde{\Lambda}} e^{\langle \xi, \gamma \rangle} \textnormal{vol}^{\tilde{\Lambda}}(\mathfrak{p}) \cdot t^n + O_{t \to \infty}(t^{n-1}).
\end{gather*}
However, unless $\Lambda = \tilde{\Lambda}$, i.e. $\xi \in 2 \pi i \cdot \Lambda^*$, we have that
\begin{gather*}
    \sum_{[\gamma] \in \Lambda/\tilde{\Lambda}} e^{\langle \xi, \gamma \rangle} = 0,
\end{gather*}
as $e^{\langle \xi, \cdot \rangle}$ is a non-trivial character of $\Lambda/\tilde{\Lambda}$ and thus its average value is zero. Therefore we actually have that
\begin{gather*}
    S_{\Lambda}(t \cdot \mathfrak{p}; \xi) = O_{t \to \infty}(t^{n-1}).
\end{gather*}
Finding the right asymptotic growth in this case and in general requires a more refined understanding of the relationship between $\mathfrak{p}$,  $\Lambda$, and $\xi$.

\begin{example}
A simple example of the above phenomenon is letting $\mathfrak{p}$ be the interval $[a, b]$ with $a, b \in \mathbb{Z}$, letting $\Lambda = \mathbb{Z}$, and letting $\xi = \pi i$. Then $\tilde{\xi} = 0$ and $\tilde{\Lambda} = 2 \mathbb{Z}$. We clearly have that $S_{\mathbb{Z}}(t \cdot [a, b]; \pi i) = (-1)^{t \cdot (a-b)}$ which is of order $t^0$ rather than of order $t^1$. 
\end{example}

\begin{example}
    We now show a more interesting example to show that obtaining the right growth rate of $S_{\Lambda}(t \cdot \mathfrak{p}; \xi)$ for $\xi \in \pi i \cdot V_{\mathbb{Q}}^*$ requires a more refined geometric understanding of the relationship between $\Lambda$, $\tilde{\Lambda}$, and $\mathfrak{p}$. Let $\mathfrak{p}$ be the unit square, i.e. the polytope with vertices at $(0, 0), (1, 0), (0, 1)$, and $(1, 1)$, and let $\Lambda = \mathbb{Z}^2$. First suppose $\xi_1 = \pi i \cdot (1, 0)$. Then $\tilde{\Lambda} = \{(x, y): x, y \in \mathbb{Z}, \ x \textnormal{ even} \}$. We get that 
    \begin{gather*}
        S_{\Lambda}(t \cdot \mathfrak{p}; \xi_1) = \begin{cases}
            (t+1) & \textnormal{$t$ even} \\
            0 & \textnormal{$t$ odd}
        \end{cases} \ \ = O_{t \to \infty}(t).
    \end{gather*}
    Now suppose $\xi_2 = \pi i \cdot (1, 1)$. Then $\tilde{\Lambda} = \{(x, y) : x, y \in \mathbb{Z}, \ x + y \textnormal{ even}\}$, and thus 
    \begin{gather*}
        S_\Lambda(t \cdot \mathfrak{p}; \xi_2) = \begin{cases}
            1 & \textnormal{$t$ even} \\
            0 & \textnormal{$t$ odd}
        \end{cases} 
        \ \ = O_{t \to \infty}(1).
    \end{gather*}
    It would be interesting to pin down the precise growth rate of $S_\Lambda(t \cdot \mathfrak{p}; \xi)$ for general rational $\mathfrak{p}$ and $\xi \in \pi i \cdot V_{\mathbb{Q}}^*$, though we do not pursue this question in this paper. This question is clearly closely related to how the Ehrhart polynomial changes upon passage to a finite index sublattice.
\end{example}

\section*{Acknowledgements}
I would like to thank Matthias Beck, who encouraged me at an early stage of this project. I would also like to thank Mikolaj Fraczyk who suggested to me to look into the Euler-Maclaurin formula. I would like to thank the referee for helpful feedback, in particular for helping to improve the notation, and for bringing to my attention the connection of the degenerate Brion's formula in this paper and the Atiyah-Bott-Berline-Vergne localization formula. This project was supported by the Deutsche Forschungsgemeinschaft (DFG, German Research Foundation) Grant SFB-TRR 358/1 2023 - 491392403. The project has received funding from
the European Union’s Horizon 2020 research and
innovation programme under the Marie Skłodowska-Curie
grant agreement No 101034255. This work was partially supported by NSF Grant DMS-2503324.


\bibliographystyle{elsarticle-harv} 
\bibliography{biblio}

\end{document}